\pdfoutput=1
\documentclass[11pt]{article} 
\def\arXiv{1} 

\usepackage[utf8]{inputenc} %
\usepackage[T1]{fontenc}    %
\usepackage{url}            %
\usepackage{booktabs}       %
\usepackage{amsfonts}       %
\usepackage{nicefrac}       %
\usepackage{microtype}      %
\usepackage{xcolor}         %

\usepackage{enumerate}
\usepackage{amssymb}
\usepackage{amsbsy}
\usepackage{amsmath}
\usepackage{amsthm}
\usepackage{amsfonts}
\usepackage{latexsym}
\usepackage{graphicx}
\usepackage{mathtools}
\usepackage{xcolor}
\usepackage{enumitem}
\usepackage{thmtools}
\usepackage{thm-restate}
\usepackage{graphicx}
\usepackage{color}
\usepackage{bbm}
\usepackage{xifthen}
\usepackage{xspace}
\usepackage{mathtools}
\usepackage{titlesec}
\usepackage{setspace}
\usepackage{etoolbox}
\usepackage{xfrac}
\usepackage{esvect}
\usepackage{nicefrac}
\usepackage{cases}
\usepackage{empheq}
\usepackage{booktabs}
\usepackage{multirow}
\usepackage{xparse}
\usepackage{caption}
\usepackage{bbm}

\newcommand{\notarxiv}[1]{foo}
\newcommand{\arxiv}[1]{ba}
\ifdefined \arXiv
	\renewcommand{\arxiv}[1]{#1}%
	\renewcommand{\notarxiv}[1]{\ignorespaces}%
\else%
	\renewcommand{\arxiv}[1]{\ignorespaces}%
	\renewcommand{\notarxiv}[1]{#1}%
\fi%

\arxiv{
	\usepackage[top=1in, right=1in, left=1in, bottom=1in]{geometry}
	\usepackage{url}
	\usepackage[hidelinks]{hyperref}
	\hypersetup{
		colorlinks=true,
		linkcolor=blue!70!black,
		citecolor=blue!70!black,
		urlcolor=blue!70!black}
	\usepackage[numbers, square]{natbib}
}

\notarxiv{
	\usepackage[subtle, all=normal, mathspacing=normal, floats=tight, 
	sections=tight]{savetrees}
	\usepackage[hidelinks=true]{hyperref}
}

\usepackage[linesnumbered,ruled,vlined]{algorithm2e}

\usepackage{cleveref}

\definecolor{darkblue}{rgb}{0,0,.75}
\newcommand{\mc}[1]{\mathcal{#1}}

\DeclarePairedDelimiter{\abs}{\lvert}{\rvert} %
\DeclarePairedDelimiter{\brk}{[}{]}
\DeclarePairedDelimiter{\crl}{\{}{\}}
\DeclarePairedDelimiter{\prn}{(}{)}

\DeclarePairedDelimiter{\norm}{\|}{\|}

\DeclarePairedDelimiter{\ceil}{\lceil}{\rceil}
\DeclarePairedDelimiter{\floor}{\lfloor}{\rfloor}

\newcommand{\inner}[2]{\left<#1,#2\right>}

\newcommand{\overeq}[1]{\overset{#1}{=}}
\newcommand{\overle}[1]{\overset{#1}{\le}}
\newcommand{\overge}[1]{\overset{#1}{\ge}}

\NewDocumentCommand\Ex{s O{} m }{%
	\mathbb{E}%
	\begingroup
	\IfBooleanTF{#1}
	{\ExInn*{#3}}
	{\ExInn[#2]{#3}}%
	\endgroup
}

\DeclarePairedDelimiterX\ExInn[1]{[}{]}{%
	\activatebar
	#1%
}

\RenewDocumentCommand\Pr{sO{}r()}{%
	\mathbb{P}%
	\begingroup
	\IfBooleanTF{#1}
	{\PrInn*{#3}}
	{\PrInn[#2]{#3}}%
	\endgroup
}

\DeclarePairedDelimiterX\PrInn[1](){%
	\activatebar
	#1%
}

\newcommand{\activatebar}{%
	\begingroup\lccode`~=`|
	\lowercase{\endgroup\def~}{\;\delimsize\vert\;}%
	\mathcode`|=\string"8000
}

\newcommand{\ltwo}[1]{\norm{#1}_2} %

\newcommand{\defeq}{\coloneqq}

\newcommand{\wt}[1]{\widetilde{#1}} %

\newcommand{\half}{\frac{1}{2}}
\newcommand{\indic}[1]{\mathbbm{1}_{\!\left\{#1\right\}}} %
\newcommand{\R}{\mathbb{R}}
\newcommand{\N}{\mathbb{N}}

\newcommand{\ones}{\mathbf{1}}

\makeatletter
\long\def\@makecaption#1#2{
  \vskip 0.8ex
  \setbox\@tempboxa\hbox{\small {\bf #1:} #2}
  \parindent 1.5em  %
  \dimen0=\hsize
  \advance\dimen0 by -3em
  \ifdim \wd\@tempboxa >\dimen0
  \hbox to \hsize{
    \parindent 0em
    \hfil 
    \parbox{\dimen0}{\def\baselinestretch{0.96}\small
      {\bf #1.} #2
    } 
    \hfil}
  \else \hbox to \hsize{\hfil \box\@tempboxa \hfil}
  \fi
}
\makeatother

\definecolor{innerboxcolor}{rgb}{.9,.95,1}
\definecolor{outerlinecolor}{rgb}{.6,0,.2}

\newcommand{\E}{\mathbb{E}} %
\renewcommand{\P}{\mathbb{P}} %
\newcommand{\simiid}{\stackrel{\rm iid}{\sim}}

\newcommand{\ball}{\mathbb{B}}

\newcommand{\normal}{\mathsf{N}}  %
\newcommand{\uniform}{\mathsf{Unif}}  %
\newcommand{\bernoulli}{\mathsf{Bernoulli}}  %
\providecommand{\argmin}{\mathop{\rm argmin}}

\providecommand{\abs}{\mathop{\rm abs}}

\providecommand{\minimize}{\mathop{\rm minimize}}

\newtheorem{theorem}{Theorem}
\newtheorem{lemma}{Lemma}
\newtheorem{corollary}[theorem]{Corollary}

\newtheorem{proposition}{Proposition}
\newtheorem{definition}{Definition}

\newtheorem{fact}{Fact}
\newtheorem{assumption}{Assumption}

\newtheorem{remark}{Remark}

\newcounter{example}

\newenvironment{example*}[1][]{
  \ifthenelse{\isempty{#1}}{%
    \noindent \textbf{Example:}\hspace*{.05em}
  }{%
    \noindent \textbf{Example} ({#1})\textbf{:}\hspace*{.05em}
  }
}{%
  $\Diamond$ \bigskip
}

\def\keywordname{{\bfseries \emph Keywords}}%
\def\keywords#1{\par\addvspace\medskipamount{\rightskip=0pt plus1cm
    \def\and{\ifhmode\unskip\nobreak\fi\ $\cdot$
    }\noindent\keywordname\enspace\ignorespaces#1\par}}
    
\usepackage[textsize=scriptsize,textwidth=2cm]{todonotes}
\newcommand{\yairside}[1]{\todo[color=blue!10]{Yair: #1}}
\newcommand{\yair}[1]{{\bf \color{blue} Yair: #1}}

\newcommand{\sidford}[1]{{\bf \color{purple} Sidford: #1}}
\newcommand{\arun}[1]{{\bf \color{orange} Arun: #1}}
\newcommand{\ha}[1]{{\bf \color{brown} Hilal: #1}}

\renewcommand{\yairside}[1]{\ignorespaces}
\renewcommand{\yair}[1]{\ignorespaces}
\renewcommand{\sidford}[1]{\ignorespaces}
\renewcommand{\arun}[1]{\ignorespaces}
\renewcommand{\ha}[1]{\ignorespaces}

\newcommand{\Otil}[1]{\widetilde{O}( #1 )}

\newcommand{\eps}{\epsilon}
\newcommand{\grad}{\nabla}
\newcommand{\del}{\partial}
\newcommand{\xopt}{x_\star}

\SetCommentSty{mycommfont}
\SetKwInput{KwInput}{Input}                %
\SetKwInput{KwOutput}{Output}              %
\SetKwInput{KwReturn}{Return}              %
\SetKwInOut{KwParameters}{Parameters}
\SetKwInput{KwInput}{Input} 
\SetKwInput{KwParameter}{Parameters}                %
\SetKwInput{KwOutput}{Output}              %
\SetKwInput{KwReturn}{Return}              %
\SetKwComment{Comment}{$\triangleright$\ }{}
\SetCommentSty{mycommfont}

\newcommand{\optest}{\textsc{OptEst}}

\newcommand{\gradestfsm}{\textsc{SoftmaxGradEst}}

\newcommand{\gradest}{\textsc{MorGradEst}}

\newcommand{\epochSGD}{\textsc{EpochSGD}}

\newcommand{\ODC}{\textsc{ODC}}
\newcommand{\Ellipsoid}{\textsc{Ellipsoid}}

\newcommand{\nextLambda}{\textsc{NextLambda}}

\newcommand{\oracle}{\mathcal{O}}
\newcommand{\proj}{\mathsf{Proj}}

\newcommand{\prox}{\mathsf{P}}
\newcommand{\appProx}{\wt{\prox}}

\newcommand{\linesearch}{\lambda\textsc{-Bisection}}
\newcommand{\BOO}{BROO\xspace}
\newcommand{\oracles}[2][\lambda,\delta]{\mathcal{O}_{#1}(#2)}

\newcommand{\xhat}{\hat{x}}
\newcommand{\hx}{\xhat}

\newcommand{\Kmax}{K_{\max}}
\newcommand{\Amax}{A_{\max}}
\newcommand{\vphi}{\varphi}

\newcommand{\fmax}{f_{\max}}
\newcommand{\fsm}{f_{\mathrm{smax}}}

\newcommand{\reps}{r_{\eps}}
\newcommand{\bx}{\bar{x}}

\notarxiv{
\titlespacing*{\section}
{0pt}{6pt plus 2pt minus 1pt}{3pt plus 0pt minus 3pt}
\titlespacing*{\subsection}
{0pt}{3pt plus 0pt minus 3pt}{0pt plus 0pt minus 3pt}
\titlespacing*{\paragraph}
{0pt}{0pt plus 0pt minus 0pt}{6pt}
}

\title{Stochastic Bias-Reduced Gradient Methods}

\notarxiv{
\author{%
  Hilal Asi\thanks{Stanford University, \texttt{\{asi,jmblpati,yujiajin,sidford\}@stanford.edu}}\And
  Yair Carmon\thanks{Tel Aviv University, \texttt{ycarmon@tauex.tau.ac.il}}\And
  Arun Jambulapati\footnotemark[1]\And
  Yujia Jin\footnotemark[1]\And
  Aaron Sidford\footnotemark[1]\And
}
}

\arxiv{
\author{Hilal Asi ~~~ Yair Carmon ~~~ Arun Jambulapati ~~~ Yujia Jin ~~~  Aaron Sidford\\
	\texttt{\{\href{mailto:asi@stanford.edu}{asi},%
	    \href{mailto:jmblpati@stanford.edu}{jmblpati},%
		\href{mailto:yujiajin@stanford.edu}{yujiajin},%
		\href{mailto:sidford@stanford.edu}{sidford}%
		\}@stanford.edu}~~
		 \texttt{\href{mailto:ycarmon@cs.tau.ac.il}{ycarmon}@cs.tau.ac.il}
		}
\date{}
}

\begin{document}

\maketitle

\begin{abstract}
  We develop a new primitive for stochastic optimization: a low-bias, low-cost  estimator of the minimizer $x_\star$ of any Lipschitz strongly-convex function. 
  In particular, we use a multilevel Monte-Carlo approach due to \citet{blanchet2015unbiased} to turn any optimal stochastic gradient method into an estimator of $x_\star$ with bias $\delta$, variance $O(\log(1/\delta))$, and an expected sampling cost of $O(\log(1/\delta))$ stochastic gradient evaluations.
 As an immediate consequence, we obtain cheap and nearly unbiased gradient estimators for the Moreau-Yoshida envelope of any Lipschitz convex function, allowing us to perform dimension-free randomized smoothing.
 \arxiv{
 
}
 We demonstrate the potential of our estimator through four applications.
 First, we develop a method for minimizing the maximum of $N$ functions, improving on recent results and matching a lower bound up to logarithmic factors.
 Second and third, we recover state-of-the-art rates for projection-efficient and gradient-efficient optimization using simple algorithms with a transparent analysis. 
 Finally, we show that an improved version of our estimator would yield a nearly linear-time, optimal-utility, differentially-private non-smooth stochastic optimization method.
\end{abstract}

\section{Introduction}\label{sec:intro}

\newcommand{\cutoff}{j_{\max}}
\newcommand{\Tmax}{T_{\max}}
\newcommand{\xmlmc}{\hat{x}_\star}

\newcommand{\ghat}{\hat{\nabla}}
\newcommand{\gest}[1][f]{\ghat #1}
\newcommand{\fhat}[1][]{{f_{(#1)}}}

\newcommand{\xset}{\mathcal{X}}
\newcommand{\geoRV}{\mathsf{Geom}}
\newcommand{\gestME}[1][f_\lambda]{\hat{\nabla} {#1}}
\newcommand{\fME}{f_\lambda}
\newcommand{\Nquery}[1][f]{\mathcal{N}_{\gest[#1]}}
\newcommand{\NqueryCustom}[1]{\mathcal{N}_{#1}}

Consider the fundamental problem of minimizing a $\mu$-strongly convex function $F: \xset \to \R$ given access to a stochastic (sub-)gradient estimator $\gest[F]$ satisfying $\E\, \gest[F](x) \in \partial F(x)$ and $\E \norm{\gest[F](x)}^2 \le G^2$ for every $x\in\xset$. 
Is it possible to transform the unbiased estimator $\gest[F]$ into a (nearly)
unbiased estimator of the minimizer $\xopt \defeq \argmin_{x\in\xset} F(x)$? In particular, can we improve upon the $O(G/(\mu \sqrt{T}))$ bias achieved by $T$ iterations of stochastic gradient descent (SGD)? %

In this paper, we answer this question in the affirmative, proposing an \emph{optimum estimator} $\xmlmc$, which (for any fixed $\delta>0$) has
\notarxiv{
\begin{equation*}
\text{bias } \norm{ \E \xmlmc - \xopt} = O\prn*{\delta}
\mbox{ and variance }  \E \norm{\xmlmc - \E \xmlmc}^2 = O\prn*{\frac{G^2}{\mu^2} \log \prn*{\frac{G}{\mu \delta}}},
\end{equation*}
}
\arxiv{
\begin{equation*}
\text{bias } \norm{ \E \xmlmc - \xopt} = \delta
\mbox{ and variance }  \E \norm{\xmlmc - \E \xmlmc}^2 = O\prn*{\frac{G^2}{\mu^2} \log \prn*{\frac{G}{\mu \delta}}},
\end{equation*}
}
and, \emph{in expectation}, costs $O(\log(\frac{G}{\mu\delta}))$ evaluations of $\gest[F]$.\footnote{When $\xset = \ball_R(x_0) \subset \R^d$, $F(x)=\frac{1}{n}\sum_{i\in[n]}\hat{F}(x;i)$, and $\gest[F]$ is the subgradient of a uniformly random $\hat{F}(x;i)$ we can also get an estimator with bias $0$ and expected cost %
 $O(\log (nd))$. 
	See \Cref{app:remark-zero-bias} for details.}
Setting $\delta = G/(\mu \sqrt{T})$, we obtain the same bias bound as $T$ iterations of SGD, but with expected cost of only $O(\log T)$ stochastic gradient evaluations (the worst-case cost is  $T$). Further, the bias can be made arbitrarily small with only logarithmic increase in the variance and the stochastic gradient evaluations of our estimator, and therefore---paralleling the term ``nearly linear-time''~\cite{gurevich1989nearly}---we call  $\xmlmc$ \emph{nearly unbiased}.

Our estimator is an instance of the multilevel Monte Carlo technique for de-biasing estimator sequences~\cite{giles2015multilevel} and more specifically the method of~\citet{blanchet2015unbiased}.  Our key observation is that this method is readily applicable to  strongly-convex variants of SGD, or indeed any stochastic optimization method with the same (optimal) rate of convergence.  

\subsection{Estimating proximal points and Moreau-Yoshida envelope gradients}
Given a convex function $f$ and regularization level $\lambda$, the proximal point of $y$ is $\prox_{f,\lambda}(y)\defeq \argmin_{x\in\R^d}\crl*{f(x) + \tfrac{\lambda}{2}\norm{x-y}^2}$. Since computing $\prox_{f,\lambda}$ amounts to solving a $\lambda$-strongly-convex problem, our technique provides low-bias and cheap proximal point estimators. Proximal points are ubiquitous in optimization~\cite{ParikhBo13,Drusvyatskiy18,Tseng08,MonteiroS13a} and estimating them efficiently with low bias opens up new algorithmic possibilities.
One of these possibilities is estimating the gradient of the Moreau-Yoshida  envelope $f_\lambda(y)\defeq \min_{x\in\R^d}\crl*{f(x) + \tfrac{\lambda}{2}\norm{x-y}^2}$, which is a $\lambda$-smooth,  $G^2/(2\lambda)$-accurate approximation of any $G$-Lipschitz $f$ 
(see, e.g., \cite{ParikhBo13,hiriart1993convex_ii} and \Cref{app:moreau-properties}). Since $\grad f_\lambda(y) = \lambda (y-\prox_{f,\lambda}(y))$, our optimum estimator provides a low-bias estimator for $\grad f_\lambda(y)$ with second moment and expected cost greater than those of $\gest[f]$ by only a logarithmic factor. 
Thus, for any non-smooth $f$ we can turn $\gest[f]$ into a gradient estimator for the smooth surrogate $f_\lambda$, whose smoothness is independent of the problem dimension, allowing us to perform \emph{dimension-free} randomized smoothing~\cite{duchi2012randomized}. 

\subsection{Applications via accelerated bias-reduced methods}\label{sec:intro-apps}
\newcommand{\Op}{\mathcal{O}}

Our optimum estimator is a new primitive in stochastic convex optimization and we expect it to find multiple applications. We now describe three such applications:  
the first improves on previously known complexity bounds while the latter two recover existing bounds straightforwardly.
For simplicity of presentation we assume (in the introduction only) $\E\norm{\gest[f]}^2\le 1$ and unit domain size.

In each application, we wish to minimize an objective function given access to a cheap subgradient estimator $\gest[f]$ as well as an expensive application-specific operation $\Op$ (e.g., a projection to a complicated set). 
Direct use of the standard stochastic gradient method finds an $\epsilon$-accurate solution using $O(\epsilon^{-2})$ computations of both $\gest[f]$ and $\Op$, and our goal is to improve the $\Op$ complexity without  hurting the $\gest[f]$ complexity.

To that end, we design stochastic accelerated methods consisting of $T$ iterations, each one involving only a constant number of $\Op$ and proximal point computations, which we approximate by averaging copies of our optimum estimator.\footnote{While averaging is parallelizable, our optimum estimator itself is sequential. Consequently, our approach does not yield improve parallelism; see~\Cref{app:remark-parallel} for further discussion}
Its low bias allows us to bound $T\ll \epsilon^{-2}$ as though our proximal points were exact, while maintaining an $\Otil{\epsilon^{-2}}$ bound on the total expected number of $\gest[f]$ calls.\footnote{It is easy to turn expected complexity bounds into deterministic ones; see~\Cref{app:remark-total-work}.} 
Thus, we save expensive operations without substantially increasing the gradient estimation cost. \Cref{table:summary} summarizes each application, and we briefly describe them below.

\begin{table}[]
	\centering
	\begin{tabular}{@{}llll@{}}
		\toprule
		Objective                                                      & Expensive operation $\Op$          & $\NqueryCustom{\Op}$                      & $\E \Nquery$                                             \\ \midrule
		$\max_{i\in[N]} \fhat[i](x)$ (Sec.~\ref{sec:mtm})              & $\fhat[1](x), \ldots, \fhat[N](x)$ & 
		$\Otil{\eps^{-2/3}}$
		& $\Otil{\eps^{-2}}$ \\
		$f(x)$ in domain $\xset$ (Sec.~\ref{sec:proj-eff})             & $\proj_{\xset}(x)$                 & $O(\eps^{-1})$                            &  ~~~\texttt{"}                                                        \\
		$\Lambda(x)+f(x)$ for $L$-smooth $\Lambda$ (Sec.~\ref{sec:gs}) & $\grad \Lambda(x)$                 & $O\prn[\big]{\sqrt{L/\eps}}$               &    ~~~\texttt{"}                                                      \\ \bottomrule
	\end{tabular}
	\caption{Summary of our applications of accelerated bias-reduced stochastic gradient methods. We use $\NqueryCustom{\Op}$ and $\Nquery$ to denote the number of expensive operations and subgradient estimations, respectively. The $\wt{O}$ notation hides polylogarithmic factors. See \Cref{sec:intro-apps} for additional description.}
	\label{table:summary}
\end{table}

\paragraph{Minimizing the maximal loss (\Cref{sec:mtm}).} 
Given $N$ convex, $1$-Lipschitz functions $\fhat[1], \ldots, \fhat[N]$ we would like to find an $\epsilon$-approximate minimizer of their maximum $\fmax(x) = \max_{i\in [N]} \fhat[i](x)$. This problem naturally arises  when optimizing worst-case behavior, as in maximum margin classification and robust optimization~\cite{vapnik1999overview,clarkson2012sublinear,ShalevWe16,Ben-TalGhNe09}. We measure complexity by the number of individual function and subgradient evaluations, so that the expensive operation of evaluating  $\fhat[1],\ldots,\fhat[N]$ at a single point has complexity $O(N)$ and the subgradient method solves this problem with complexity $O(N\epsilon^{-2})$. 
\citet{carmon2021thinking} develop an algorithm for minimizing $\fmax$ with complexity $\Otil{ N\eps^{-2/3} + \eps^{-8/3} }$, improving on the subgradient method for sufficiently large $N$. 
Using our bias-reduced Moreau gradient envelope estimator in a Monteiro-Svaiter-type accelerated proximal point method~\cite{carmon2020acceleration,bubeck2019complexity,MonteiroS13a}, we obtain improved complexity $\Otil{ N\epsilon^{-2/3} + \epsilon^{-2} }$.
This matches (up to logarithmic factors) a lower bound shown in \cite{carmon2021thinking}, settling the complexity of minimizing the maximum of $N$ non-smooth functions. Our result reveals a surprising fact: for $N \ll (GR/\epsilon)^{-4/3}$, minimizing the maximum of $N$ functions is no harder than minimizing their average.

\paragraph{Projection-efficient optimization via dimension-free randomized smoothing~(\Cref{sec:proj-eff}).}
Consider the problem of minimizing a convex function $f$ using an unbiased gradient estimator $\gest[f]$ over convex set $\xset$ for which Euclidean projections are expensive to compute (for example, the cone of PSD matrices). When $f$ is $L$-smooth, a stochastic version of Nesterov's accelerated gradient descent (AGD)~\cite{cohen2018acceleration} performs only $O(\sqrt{L/\epsilon})$ projections. 
\arxiv{
When $f$ is non-smooth, we may replace it with a smooth surrogate. Randomized smoothing~\cite{duchi2012randomized} allows us to find a solution using $O(d^{1/4}/\epsilon)$ projections, where $d$ is the problem dimension. 
Alternatively, applying AGD on the Moreau envelope $f_{\lambda}$ (for $\lambda = O(\epsilon^{-1})$) yields a solution in $O(\epsilon^{-1})$ projections, removing the dimension dependence inherent to standard randomized smoothing. However, computing $\grad f_{\lambda}$ to sufficient accuracy requires a total of $O(\epsilon^{-3})$  calls to $\gest[f]$. 
Instead, we apply stochastic AGD with our  Moreau envelope gradient estimator, improving the number of $\gest[f]$ calls to the near optimal $\Otil{\epsilon^{-2}}$. 
Our algorithm provides a simple alternative to the recent work of \citet{thekumparampil2020projection} whose performance guarantees are identical up to a logarithmic factor.
}
\notarxiv{
For non-smooth $f$ we instead apply AGD to the Moreau envelope smoothing of $f$ (with appropriate $\lambda=O(\epsilon^{-1})$) using our nearly-unbiased stochastic estimator for $\grad f_\lambda$. This yields a solution in $O(\epsilon^{-1})$ projections and $\Otil{\eps^{-2}}$ evaluations of $\gest[f]$. 
Our algorithm provides a simple alternative to the recent work of \citet{thekumparampil2020projection} whose performance guarantees are identical up to a logarithmic factor.
}

\paragraph{Gradient-efficient composite optimization~(\Cref{sec:gs}).}

We would like to minimize  $\Psi(x) = \Lambda(x) + f(x)$, where $\Lambda$ is convex and $L$-smooth but we can access it only via computing (expensive) exact gradients, while $f$ is a non-smooth convex functions for which we have a (cheap) unbiased subgradient estimator $\gest[f]$. Problems of this type include inverse problems with sparsity constraints and regularized  loss minimization in machine learning~\cite{lan2016gradient}. To save $\grad \Lambda$ computations, it is possible to use composite AGD~\cite{nesterov2013gradient} which solves $O(\sqrt{L/\eps})$ subproblems of the form $\minimize_{x}\crl[\big]{\inner{\grad \Lambda(y)}{x} + f(x) + \frac{\beta}{2}\norm{x-x'}^2}$. \citet{lan2016gradient} designed a specialized method, gradient sliding, for which the total subproblem solution cost is $O(\epsilon^{-2})$ evaluations of $\gest[f]$. We show that a simple alternative---estimating the subproblem solutions via our low-bias optimum estimator---recovers its guarantees up to logarithmic factors.

\subsection{Non-smooth differentially private stochastic convex optimization}
We now discuss a potential application of our technique that is conditional on the existence of an improved optimum estimator. In it, we minimize the population objective function $f(x) = \E_{S \sim P}\hat f(x;S)$ under the well-known constraint of differential privacy~\cite{DworkMcNiSm06}.
Given $n$ i.i.d.\ samples $S_i \sim P$ and assuming that each $\hat f$ is $1$-Lipschitz, convex and sufficiently smooth, \citet{FeldmanKoTa20} develop algorithms that obtain the optimal error and compute $O(n)$ subgradients of $\hat f$. The non-smooth case is more challenging and the best existing bound is $O(n^{11/8})$ for the high-dimensional setting $d=n$~\cite{KulkarniLeLi21,AsiFeKoTa21}.
In \Cref{sec:DP-SCO} we show that our optimum estimator, combined with recent localization techniques~\cite{FeldmanKoTa20}, 
 reduces the problem to private mean estimation. Unfortunately, our estimator is heavy-tailed, leading to insufficient utility. Nevertheless, assuming a version of our estimator that has bounded outputs, we give an algorithm that queries $\Otil{n}$ subgradients for non-smooth functions, solving a longstanding open problem in private optimization~\cite{ChaudhuriMoSa11,BassilySmTh14}.
This motivates the study of improved versions of our estimators that have constant sensitivity.

\subsection{Related work}
Multilevel Monte-Carlo (MLMC) techniques originate from the literature on parametric integration for solving integral and differential equations~\cite{giles2015multilevel}.
Our approach is based on an MLMC variant put forth by \citet{blanchet2015unbiased} for estimating functionals of expectations. Among several applications, they propose \cite[Section 5.2]{blanchet2015unbiased} an estimator for $\argmin_x \E_{S\sim P} \hat{f}(x;S)$ where $\hat{f}(\cdot;s)$ is convex for all $s$ and assuming access to minimizers of empirical objectives of the form $\sum_{i\in[N]} \hat{f}(x;s_i)$.
The authors provide a preliminary analysis of the estimator's variance (later elaborated in~\cite{blanchet2019unbiased}) using an asymptotic Taylor expansion around the population minimizer. In comparison, we study the more general setting of stochastic gradient estimators and provide a complete algorithm based on SGD, along with a non-asymptotic analysis and concrete settings where our estimator is beneficial.

A number of works have used the Blanchet-Glynn estimator in the context of optimization and machine learning. These applications include estimating the ratio of expectations for semisupervised learning~\cite{blanchet2018semi}, estimating gradients of distributionally robust optimization objectives~\cite{levy2020large}, and estimating gradients in deep latent variable models~\cite{shi2021multilevel}. Our estimator is similar to that of \citet{levy2020large} in that we also have to pick a ``critical'' doubling probability for the (random) computational budget, which makes the expected  cost and variance of our estimators depend logarithmically on the bias.

\subsection{Limitations}
Our paper demonstrates that our proposed optimum estimator is a useful \emph{proof device}: it allows us to easily prove upper bounds on the complexity of structured optimization problems, and at least in one case (minimizing the maximum loss) improve over previously known bounds. 
However, our work does  not investigate the \emph{practicality} of our optimum estimator, as implementation and experiments are outside its scope.

Nevertheless, let us briefly discuss the practical prospects of the algorithms we propose.
On the one hand, our optimum estimator itself is fairly easy to implement, adding only a few parameters on top of a basic gradient method. 
On the other hand, in the settings of \Cref{sec:proj-eff,sec:gs}, gradient-sliding based methods~\cite{lan2016gradient,thekumparampil2020projection} are roughly as simple to implement and enjoy slightly stronger convergence bounds (better by logarithmic factors) than our optimum estimator. Consequently, in these settings we have no reason to assume that our algorithms are better in practice. In the setting of \Cref{sec:mtm} (minimizing the maximum loss) our algorithm does enjoy a stronger guarantee than the previous best bound~\cite{carmon2021thinking}. However, both our algorithm and~\cite{carmon2021thinking} are based on an accelerated proximal point method that, in its current form, is not practical~\cite[Sec. 6.2]{carmon2021thinking}. Thus, evaluating the benefit of stochastic bias reduction in the context of minimizing the maximum loss would require us to first develop a practical accelerated proximal point algorithm, which is an open question under active research~\cite[see, e.g.,][]{song2021unified}.

Another limitation of our optimum estimator is that, while it has a bounded second moment, its higher moments are unbounded. While this does not matter for most of our results, the lack of higher moment bounds prevents us from setting the complexity of non-smooth private stochastic convex optimization in \Cref{sec:DP-SCO}. Finding an optimum estimator that is bounded with high probability---or proving that one does not exist---remains an open question for future work.

Finally, our analyses are limited to convex objective functions. 
However, while outside the scope of the paper, we believe our results are possibly relevant for non-convex settings as well. In particular, for smooth non-convex functions (and weakly-convex functions~\cite{davis2019stochastic} more broadly) the problem of computing proximal points with sufficiently high regularization is strongly convex and our estimator applies. Such non-convex proximal points play an important role in non-convex optimization~\cite{davis2019stochastic} with applications in deep learning~\cite[see, e.g.,][]{sinha2018certifying}. 
Applying the optimum-estimator technique in non-convex optimization is therefore a viable direction for future work.

\subsection{Notation}
We let $\ball_{R}(x) = \{ y \in \R^d : \norm{y - x} \le R \}$ denote the ball of radius $R$ around $x$, where $\norm{\cdot}$ is the Euclidean norm throughout. We write $\proj_{\mc{S}}$ for the Euclidean projection to $\mc{S}$. We write $\indic{A}$ for the indicator of event $A$, i.e., $\indic{A}=1$ when $A$ holds and $0$ otherwise.
Throughout the paper,  $\gest$ denotes a (stochastic) subgradient estimator for the function $f$, and $\xset \subset \R^d$ denotes the optimization domain, which we always assume is closed and convex. 
We use $\prox_{f,\lambda}$ to denote the proximal operator~\eqref{eq:def-prox} and $f_\lambda$ to denote the Moreau envelope~\eqref{eq:def-ME} associated with function $f$ and regularization parameter $\lambda$.
 Finally, we use $\NqueryCustom{f}$ and $\Nquery$ to denote function and subgradient estimator evaluation complexity, respectively.
\section{A multilevel Monte-Carlo optimum estimator}\label{sec:estimator}

In this section, we construct a low-bias estimator for the minimizer of any strongly convex function $F:\xset \to \R$. This estimator is the key component of our algorithms in the subsequent sections, which use it to approximate proximal points and Moreau envelope gradients. We assume that $F$ is of the form $F=f+\psi$, where the function $\psi$ is ``simple'' and that $f$ satisfies the following.
\begin{assumption}\label{ass:gest}
	The function $f:\xset\to \R$ is convex (with closed and convex domain $\xset$) and is accessible via an unbiased subgradient estimator $\gest[f]$ which satisfies $\E \norm{\gest[f](x)}^2 \le G^2$ for all $x$. 
\end{assumption}

Our applications only use  $\psi$ of the form $\psi(x)=\frac{\lambda}{2}\norm{x-x'}^2$ but our estimator applies more broadly to cases where $\argmin_{x}\crl[\big]{ \inner{v}{x} + \psi(x) + \frac{1}{2\eta} \norm{x-y}^2}$ is easy to compute for all $v$ and $y$.

\subsection{ODC algorithms}
Our estimator can use, in a black-box fashion, any method for minimizing $F$ with sufficiently fast convergence to $\xopt = \argmin_{x\in\xset} F(x)$. We highlight the required convergence property as follows.
\begin{definition}\label{def:ODC}
	An \emph{optimal-distance-convergence} algorithm $\ODC$ takes as input $\gest$ satisfying \Cref{ass:gest}, a simple function $\psi$ and a budget $T\ge 1$. If $F=f+\psi$ is $\mu$-strongly convex with minimizer $\xopt$, the algorithm's output $x = \ODC(\gest[f], \psi, T)$ requires at most $T$ evaluations of $\gest[f]$ to compute and satisfies $\E\norm{x-\xopt}^2 \le c \frac{G^2}{\mu^2 T}$ for some constant $c>0$.
\end{definition}

Standard lower bound constructions imply that the $O(\frac{G^2}{\mu^2 T})$ squared distance convergence rate is indeed optimal; see \Cref{app:remark-odc} for additional discussion.  Conversely, 
ODC algorithms are readily available in the literature~\cite{rakhlin2012making,hazan2014beyond} since any point $x$ satisfying $\E F(x) - F(\xopt) =  O(\frac{G^2}{\mu T})$ (the optimal rate of convergence in strongly convex, Lipschitz optimization) also satisfies $\E\norm{x-\xopt}^2 \le O(\frac{G^2}{\mu^2 T})$ by due to the strong convexity of $F$. We provide a concrete ODC algorithm consisting of a
generalization of epoch SGD~\cite{hazan2014beyond}, which allows us to optimize over the composite objective $F=f+\psi$ instead of only $f$ as in the prior study of epoch SGD.

\begin{restatable}{lemma}{epochlemma}\label{lem:epoch-SGD} 
	 $\epochSGD$ (\Cref{alg:epochSGD} in \Cref{app:epochSGD}) is an ODC algorithm with constant $c=32$.
\end{restatable}

\begin{figure}[t]
	\begin{minipage}[t]{0.46\linewidth}%
		\centering
		\begin{algorithm}[H]
			\DontPrintSemicolon
			\LinesNotNumbered
				\caption{$\optest(\gest[f],\psi, \mu, \delta,\sigma^2,\xset)$}	\label{alg:optest}
			\Comment{$\delta,\sigma^2$ are required bias and square error}
			\Comment{$c$ is the ODC algorithm constant}
			$\Tmax = \ceil*{\frac{4cG^2}{\mu^2\min\crl{\delta^2, \half\sigma^2}}}$\;
			$N = \ceil*{ \frac{32c G^2 \log (\Tmax)}{\mu^2\sigma^2} }$\;
			\For{$i=1,\ldots,N$}{
				$\xmlmc^{(i)} = $ a draw of the estimator~\eqref{eq:def-mlmc}
			}
			\Return $\frac{1}{N}\sum_{i \in [N]} \xmlmc^{(i)}$
		\end{algorithm}
	\end{minipage}\hfill
	\begin{minipage}[t]{0.52\linewidth}%
		\centering
		\begin{algorithm}[H]
			\DontPrintSemicolon
			\LinesNotNumbered
			\caption{$\gradest(\gest[f],y,\lambda,\delta,\sigma^2,\xset)$}	\label{alg:gradest}
			\Comment{$\delta,\sigma^2$ are required bias and square error}
			\Comment{$\lambda$ is the regularization level}
			\Comment{$y$ is the point at which to estimate $\grad f_\lambda(y)$}
			$\psi_\lambda(x) = \frac{\lambda}{2}\norm{x-y}^2$\;
			$\xmlmc = \optest\prn*{\gest[f],\psi_\lambda,\lambda,\frac{\delta}{\lambda},\frac{\sigma^2}{\lambda^2},\xset}$\;
			\Return $\lambda(y-\xmlmc)$\;
		\end{algorithm}
	\end{minipage}
\end{figure}

\subsection{Constructing an optimum estimator}
To turn any ODC algorithm into a low-bias, low-cost and near-constant variance optimum estimator, we use the multilevel Monte Carlo (MLMC) technique of \citet{blanchet2015unbiased}. Given a problem instance $\gest[f],\psi$, an algorithm $\ODC$ and a cutoff parameter $\Tmax\in\N$, our estimator $\xmlmc$ is:
\begin{gather}
	\text{Draw}~J  \sim \geoRV\left(\tfrac{1}{2}\right)\in\N\mbox{ and, writing}~
	x_j \defeq \ODC(\gest[f], \psi, 2^{j}),
	\mbox{ compute}
	\nonumber \\ 
	\xmlmc = x_0 +
	\begin{cases}
		2^J \prn*{ x_{J} - x_{J-1} } & 2^J \le \Tmax \\
		0 & \mbox{otherwise}. \\
	\end{cases}
	\label{eq:def-mlmc}
\end{gather}

We note that for certain ODC algorithms it is possible to extract $x_0, x_{J-1}$ from the intermediate steps of computing $x_J$, so that we only need to invoke $\ODC$ once. This is particularly simple to do for $\epochSGD$, as we explain in \Cref{app:epochSGD}. The key properties of our estimator are as follows.

\newcommand{\jmax}{j_{\max}}

\begin{proposition}\label{prop:est}
\notarxiv{
Let $f$ and $\gest[f]$ satisfy \Cref{ass:gest}, $F=f+\psi$ be $\mu$-strongly convex with minimizer $\xopt$ and $\Tmax\in\N$. For any ODC algorithm with constant $c$, the estimator~\eqref{eq:def-mlmc} has bias $\norm{\E\xmlmc - \xopt} \le \sqrt{2c}\frac{G}{\mu\sqrt{\Tmax}}$ and variance $\E\norm{\xmlmc - \E\xmlmc}^2 \le 16 c \frac{G^2}{\mu^2}\log_2(\Tmax)$. Moreover, the expected number of $\gest$ evaluations required to compute $\xmlmc$ is $O(\log \Tmax)$.
}
\arxiv{
Let $f$ and $\gest[f]$ satisfy \Cref{ass:gest}, $F=f+\psi$ be $\mu$-strongly convex with minimizer $\xopt$ and $\Tmax\in\N$. For any ODC algorithm with constant $c$, the estimator~\eqref{eq:def-mlmc} has 
\begin{align*}
	\text{bias}~~\norm{\E\xmlmc - \xopt} \le \sqrt{2c}\frac{G}{\mu\sqrt{\Tmax}}~~\text{ and variance }~~\E\norm{\xmlmc - \E\xmlmc}^2 \le 16 c \frac{G^2}{\mu^2}\log_2(\Tmax).
\end{align*}
 Moreover, the expected number of $\gest$ evaluations required to compute $\xmlmc$ is $O(\log \Tmax)$.
}
\end{proposition}
\begin{proof}
	Let $\jmax = \max\crl{j\in\N \mid 2^j \le \Tmax} = \floor{\log_2 \Tmax)}$. The expectation of $\xmlmc$ is
	\begin{equation*}
		\E \xmlmc = \E x_0 + \sum_{j=1}^{\jmax} \P(J=j) 2^j (\E x_{j} - \E x_{j-1} ) = \E x_{\jmax},
	\end{equation*}
	where the second equality follows from $\P(J=j) = 2^{-j}$ and the sum telescoping. Noting that $x_{\jmax} = \ODC(\gest[f],\psi, T)$ for $T=2^{\jmax} \ge \Tmax/2$, we have that 
	\begin{equation*}
	\norm{\E x_{\jmax} - \xopt} \le \sqrt{\E \norm{ x_{\jmax}-\xopt}^2 } \le \sqrt{c}\frac{G}{\mu\sqrt{\Tmax/2}}
	\end{equation*}
	by \Cref{def:ODC}. To bound the variance we use $\norm{a+b}^2 \le 2\norm{a}^2 + 2\norm{b}^2$  and note that
	\begin{equation*}
		\E \norm{\xmlmc - \E \xmlmc}^2 \le 
		\E \norm{\xmlmc - \xopt}^2 
		\le 
		2\E \norm{\xmlmc - x_0}^2 + 
		2\E \norm{x_0 - \xopt}^2.
	\end{equation*}
	The ODC property implies that $\E \norm{x_0 - \xopt}^2 \le c G^2 / \mu^2$. For the term $\E \norm{\xmlmc - x_0}^2$ we have
	\begin{align*}
		& \E \norm{\xmlmc - x_0}^2   = \sum_{j=1}^{\jmax} \P(J=j) 2^{2j} \E\norm{x_j - x_{j-1}}^2 = \sum_{j=1}^{\jmax} 2^{j} \E\norm{x_j - x_{j-1}}^2, 
		\mbox{~~and} \\
	& \E\norm{x_j - x_{j-1}}^2  \le 2\E \norm{x_j - \xopt}^2 + 2\E \norm{x_{j-1} - \xopt}^2 \le 6c \frac{G^2}{\mu^2} 2^{-j}.
	\end{align*}
	\notarxiv{
	Substituting, we get $\E \norm{\xmlmc - x_0}^2\le 6c \frac{G^2}{\mu^2} \jmax$ and 
	$\E \norm{\xmlmc - \E \xmlmc}^2 %
	\le 16 c \frac{G^2}{\mu^2} \log_2(\Tmax)$. }
	\arxiv{
	Substituting, we get 
	\[\E \norm{\xmlmc - x_0}^2\le 6c \frac{G^2}{\mu^2} \jmax~~\text{and}~~\E \norm{\xmlmc - \E \xmlmc}^2  
	\le 16 c \frac{G^2}{\mu^2} \log_2(\Tmax).\]
	}
	Finally, the expected number of $\gest[f]$ evaluations is
	$1 + \sum_{j=1}^{\jmax} \P(J=j) (2^j + 2^{j-1} ) = O(\jmax)$. 	
\end{proof}

The function $\optest$ in \Cref{alg:optest} computes an estimate of  $\xopt$ with and desired bias $\delta$ and square error $\sigma^2$ by averaging independent draws of the MLMC estimator~\eqref{eq:def-mlmc}. 
The following guarantees are immediate from \Cref{prop:est}; see \Cref{app:optest} for a short proof.

\begin{restatable}{theorem}{restateThmOptest}\label{thm:optest}
	Let $f$ and $\gest[f]$ satisfy \Cref{ass:gest}, $F=f+\psi$ be $\mu$-strongly convex with minimizer $\xopt\in \xset$, and $\delta,\sigma>0$. The function $\optest(\gest[f],\psi, \mu, \delta,\sigma^2,\xset)$ outputs $\xmlmc$ satisfying 
 \[\|\E \xmlmc-\xopt\| \le \delta~~\text{ and }~~\E\|\xmlmc-\xopt\|^2\le \sigma^2\]  using $\Nquery$ stochastic gradient computations, where
\[\E \Nquery = O\left(\frac{G^2}{\mu^2\sigma^2}\log^2 \left(\frac{G}{\mu\min\{\delta,\sigma\}}\right)
+\log \prn*{\frac{G}{\mu\min\{\delta,\sigma\}}}
\right).\] 
\end{restatable}

\subsection{Estimating proximal points and Moreau envelope gradients}\label{ssec:estimator-grad}
The proximal point of function $f:\xset\rightarrow \R$ with regularization level $\lambda$ at point $y$ is
\begin{equation}
	\prox_{f,\lambda}(y) \defeq \argmin_{x\in\xset}\crl*{f(x) + \tfrac{\lambda}{2}\norm{x-y}^2}.\label{eq:def-prox}
\end{equation}
When $f$ satisfies \Cref{ass:gest},
 we may use $\optest$ (with $\psi(x) = \frac{\lambda}{2} \norm{x-y}^2$ and $\mu=\lambda$) to obtain a reduced-bias proximal point estimator. The proximal point $\prox_{f,\lambda}(y)$ is closely related to the Moreau envelope
\begin{equation}
	f_\lambda(y) \defeq \min_{x\in\xset}\crl*{f(x) + \tfrac{\lambda}{2}\norm{x-y}^2}~\label{eq:def-ME}
\end{equation}
via the relationship $\grad f_\lambda(y) = \lambda \prn*{ y -\prox_{f,\lambda}(y) }$ (see \Cref{app:moreau-properties}). 
Therefore, we can use our optimum estimator to turn $\Otil{1}$ calls to $\gest[f]$ into a nearly unbiased estimator for $\grad f_\lambda$. We formulate this as:

\notarxiv{
\begin{corollary}\label{coro:gradest}
	Let $f$ and $\gest[f]$ satisfy \Cref{ass:gest}, let $y\in\xset$ and let $\lambda,\sigma,\delta >0$. 
	The function $\gradest(\gest[f],\lambda, y,\delta,\sigma^2,\xset)$ outputs $\gestME(y)$ satisfying 
	$\|\E \gestME(y)-\grad f_\lambda(y)\| \le \delta$ and $\E\| \gestME(y)-\grad f_\lambda(y)\|^2\le \sigma^2$  and has complexity $\E\Nquery = O\left(\frac{G^2}{\sigma^2}\log^2 \left(\frac{G}{\min\{\delta,\sigma\}}\right)
	+ \log \left(\frac{G}{\min\{\delta,\sigma\}}\right)
	\right)$. 
\end{corollary}
}

\arxiv{
\begin{corollary}\label{coro:gradest}
	Let $f$ and $\gest[f]$ satisfy \Cref{ass:gest}, let $y\in\xset$ and let $\lambda,\sigma,\delta >0$. 
	The function $\gradest(\gest[f],\lambda, y,\delta,\sigma^2,\xset)$ outputs $\gestME(y)$ satisfying 
	\[\|\E \gestME(y)-\grad f_\lambda(y)\| \le \delta~~\text{and}~~\E\| \gestME(y)-\grad f_\lambda(y)\|^2\le \sigma^2\] 
	  using $\Nquery$ stochastic gradient computations, where
	 \[
	 \E\Nquery = O\left(\frac{G^2}{\sigma^2}\log^2 \left(\frac{G}{\min\{\delta,\sigma\}}\right)
	 + \log \left(\frac{G}{\min\{\delta,\sigma\}}\right)
	 \right).
	 \] 
\end{corollary}
}

\section{Projection-efficient convex optimization}
\label{sec:proj-eff}
\newcommand{\moreau}{f_{\lambda}}

In this section, we combine the bias-reduced Moreau envelope gradient estimator with a standard accelerated gradient method to recover the result of \citet{thekumparampil2020projection}. We consider the problem of minimizing a function $f$ satisfying \Cref{ass:gest} over the domain $\ball_R(0)$ subject to the constraint $x\in\xset$, where $\xset \subset \ball_R(0)$ is a complicated convex set that we can only access via (expensive) projections of the form
	$\proj_{\xset}(x) \defeq \argmin_{y \in \xset}\norm{y-x}$.
We further assume that an initial point $x_0\in\xset$ satisfies $\norm{x_0-\xopt} \le D$. 
\notarxiv{
\begin{algorithm}[h]
	\DontPrintSemicolon
	\KwInput{A gradient estimator $\gest[f]$ satisfying \Cref{ass:gest} in $\ball_R(0)$, projection oracle $\proj_{\xset}$, and initial point $x_0=v_0$ with $\norm{x_0 - \xopt} \leq D$.}
	\KwParameters{Iteration budget $T$ , Moreau regularization $\lambda$, approximation parameters $\delta_k, \sigma_k^2$}
	\For{$k=1,\cdots, T$}{
	$y_{k-1} = \frac{k-1}{k+1} x_k+\frac{2}{k+1} v_{k-1}$\;
	$g_{k} = \gradest(\gest[f],y_{k-1},\lambda,\delta_k,\sigma^2_k,\ball_R(0))$ \label{line:gradest_proj}\;%
	$x_{k} = \proj_{\xset} \left(y_{k-1} - \frac{1}{3 \lambda} g_{k}\right)$\;
	$v_{k} = \proj_{\ball_R(0)}\left(v_{k-1} - \frac{k}{6 \lambda} g_k \right)$\; 
	}
		\Return $x_{T}$
	\caption{Stochastic accelerated gradient descent on the Moreau envelope}
	\label{alg:proj_eff}
\end{algorithm}
}
\arxiv{
\begin{algorithm}[t]
	\DontPrintSemicolon
		\caption{Stochastic accelerated gradient descent on the Moreau envelope}
	\label{alg:proj_eff}
	\KwInput{A gradient estimator $\gest[f]$ satisfying \Cref{ass:gest} in $\ball_R(0)\supset\xset$, projection oracle $\proj_{\xset}$, and initial point $x_0=v_0$ with $\norm{x_0 - \xopt} \leq D$.}
	\KwParameters{Iteration budget $T$ , Moreau regularization $\lambda$, approximation parameters $\delta_k, \sigma_k^2$}
	\For{$k=1,\cdots, T$}{
	$y_{k-1} = \frac{k-1}{k+1} x_k+\frac{2}{k+1} v_{k-1}$\;
	$g_{k} = \gradest(\gest[f],y_{k-1},\lambda,\delta_k,\sigma^2_k,\ball_R(0))$ \label{line:gradest_proj}\;%
	$x_{k} = \proj_{\xset} \left(y_{k-1} - \frac{1}{3 \lambda} g_{k}\right)$\;
	$v_{k} = \proj_{\ball_R(0)}\left(v_{k-1} - \frac{k}{6 \lambda} g_k \right)$\; 
	}
		\Return $x_{T}$
\end{algorithm}
}

\Cref{alg:proj_eff} applies a variant of Nesterov's accelerated gradient descent  method (related to~\cite{ZhuO17,Zhu17}) on the ($\lambda$-smooth) Moreau envelope $f_\lambda$ defined in~\cref{eq:def-ME}. Since computing the Moreau envelope does not involve projection to $\xset$, for sufficiently accurate approximation of $\grad f_\lambda$ we require only $T=O(\sqrt{\lambda D^2/\epsilon})$ projections to $\xset$  for finding an $O(\epsilon)$-suboptimal point of $f_\lambda$ constrained to $\xset$. For that point to be also $\epsilon$-suboptimal for $f$ itself, we must choose $\lambda$ of the order of $G^2/\epsilon$, so that the number of projections is $O(GD/\epsilon)$. 

As noted in \cite{thekumparampil2020projection} computing $\grad f_\lambda$ to accuracy $O(\epsilon/R)$  is sufficient for the above guarantee to hold, but doing so using a stochastic gradient method requires $O( (GD/\epsilon)^2 )$ evaluations of $\gest[f]$ per iteration, and $O( (GD/\epsilon)^3 )$ evaluations in total. To improve this, we employ Algorithm~\ref{alg:gradest} to compute nearly-unbiased estimates for $\nabla \moreau$ and bound the error incurred by their variance. 
Our result matches the gradient sliding-based technique of \citet{thekumparampil2020projection} up to polylogarithmic factors  while retaining the conceptual simplicity of directly applying AGD on the Moreau envelope. We formally state the guarantees of our method below, and provide a self-contained proof in \Cref{app:proj-eff}.

\notarxiv{
\begin{restatable}{theorem}{thmPE}
Let $f: \ball_R(0) \rightarrow \R$ and $\gest$ satisfy \Cref{ass:gest}. Let $\xset \subseteq \ball_R(0)$ be a convex set admitting a projection oracle $\proj_{\xset}$. Let $x_0 \in \xset$ be an initial point with $\norm{x - \xopt} \leq D$ for some $\xopt\in\xset$. With $\lambda = \frac{2 G^2}{\eps}$, $\delta_k = \frac{\eps}{8 R}$, $\sigma_k^2 = \frac{2\eps \lambda}{k+1}$, and $T = \frac{7 G D}{\eps}$ Algorithm \ref{alg:proj_eff} computes $x \in \xset$ with $\mathbb{E}\left[ f(x) \right] \leq f(\xopt) + \eps$ with complexity $\E\Nquery = O\left(\frac{G^2 D^2}{\eps^2} \log^2 \left( \frac{G R}{\eps} \right) \right)$ and $O\left(\frac{G D}{\eps}\right)$ calls to $\proj_{\xset}$. 
\label{thm:PE}
\end{restatable}
}

\arxiv{
\begin{restatable}{theorem}{thmPE}\label{thm:PE}
Let $f: \ball_R(0) \rightarrow \R$ and $\gest$ satisfy \Cref{ass:gest}. Let $\xset \subseteq \ball_R(0)$ be a convex set admitting a projection oracle $\proj_{\xset}$. Let $x_0 \in \xset$ be an initial point with $\norm{x_0 - \xopt} \leq D$ for some $\xopt\in\xset$. With $\lambda = \frac{2 G^2}{\eps}$, $\delta_k = \frac{\eps}{8 R}$, $\sigma_k^2 = \frac{2\eps \lambda}{k+1}$, and $T = \frac{7 G D}{\eps}$ Algorithm \ref{alg:proj_eff} computes $x \in \xset$ such that   $\mathbb{E}\left[ f(x) \right] \leq f(\xopt) + \eps$ using $O\left(\frac{G D}{\eps}\right)$ calls to $\proj_{\xset}$ and $\Nquery$ stochastic gradient computations, where
\[\E\Nquery = O\left(\frac{G^2 D^2}{\eps^2} \log^2 \left( \frac{G R}{\eps} \right) \right).\]
\end{restatable}
}

\section{Accelerated proximal methods and minimizing the maximal loss}\label{sec:mtm}

In this section we apply our estimator in an accelerated proximal point method and use it to obtain an optimal rate for minimizing the maximum of $N$ convex functions (up to logarithmic factors). 

\subsection{Accelerated proximal point method via Moreau gradient estimation}\label{sec:mtm-sapm}

\Cref{alg:MS} is an Monteiro-Svaiter-type~\cite{MonteiroS13a,carmon2020acceleration} accelerated proximal point method~\cite{lin2015universal,frostig2015unregularizing} that leverages our reduced-bias Moreau envelope gradient estimator.
To explain the method, we contrast it with stochastic AGD  on the Moreau envelope (\Cref{alg:proj_eff}). First and foremost, \Cref{alg:proj_eff} provides a suboptimality bound on the Moreau envelope $f_\lambda$ (which for small $\lambda$ is far from $f$) while \Cref{alg:MS} minimizes $f$ itself. 

Second, while \Cref{alg:proj_eff} uses a fixed regularization parameter $\lambda$, \Cref{alg:MS} handles an arbitrary sequence $\{\lambda_{k}\}$ given by a black-box function $\nextLambda$. To facilitate our application of the method to minimizing the maximal loss---where gradient estimation is only tractable in small Euclidean balls around a reference point---we include an optional parameter $r$ such that the proximal point movement bound $\norm{\prox_{f,\lambda_{k+1}}(y_k)-y_k}\le r$ holds for all $k$. However, most of our analysis of \Cref{alg:MS} does not require this parameter (i.e., holding for $r=\infty$), making it potentially applicable to other settings that use accelerated proximal point methods~\cite{bubeck2019complexity,MonteiroS13a,song2021unified}.
 
\notarxiv{
\begin{algorithm}
	\caption{Stochastic accelerated proximal point method\label{alg:MS}}
	\DontPrintSemicolon
	\KwInput{Gradient estimator $\gest$, function $\nextLambda$, initialization $x_0=v_0$ and $A_0\ge0$.}
	\KwParameters{Approximation parameters $\crl{\vphi_k,\delta_k,\sigma_k}$, stopping parameters $\Amax$ and $\Kmax$, optional movement bound $r>0$.}
	\For{$k=0,1,\ldots$}{
		$\lambda_{k+1} =\nextLambda(x_k, v_k, A_k)$\label{line:next-lambda}
		\Comment{guaranteeing that $\norm{\prox_{f,\lambda_{k+1}}(y_k)-y_k}\le r$}
		$a_{k+1} = \frac{1}{2 \lambda_{k+1}} \sqrt{1 + 4 \lambda_{k+1} A_{k}}$ and $A_{k+1} = A_k + a_{k+1}$ and $\xset_k = \xset \cap \ball_r(y_k)$\;
		$y_k = \frac{A_k}{A_{k+1}} x_{k}+\frac{a_{k+1}}{A_{k+1}} v_k$ and~$x_{k+1}=\appProx_{f,\lambda_{k+1}}^{\vphi_{k+1}}(y_k)$ \label{line:sapm-x}
		\Comment{defined in \cref{eq:app-prox-def}}
		$g_{k+1} = \gradest(\gest[f],y_k,\lambda_{k+1},\delta_k,\sigma^2_k,\xset_k)$ and $v_{k+1} = \proj_{\xset} \prn*{v_k - \half a_{k+1} g_{k+1}}$ 	\label{line:sapm-v}\;
		\lIf{$A_{k+1} \ge \Amax$ \textbf{\textup{or}}
			$k+1 = \Kmax$}
		{\Return{$x_{k+1}$}}
	}
\end{algorithm}
}
\arxiv{
\begin{algorithm}[b]
	\caption{Stochastic accelerated proximal point method\label{alg:MS}}
	\DontPrintSemicolon
	\KwInput{Gradient estimator $\gest$, function $\nextLambda$, initialization $x_0=v_0$ and $A_0\ge0$.}
	\KwParameters{Approximation parameters $\crl{\vphi_k,\delta_k,\sigma_k}$, stopping parameters $\Amax$ and $\Kmax$, optional movement bound $r>0$.}
	\For{$k=0,1,\ldots$}{
		$\lambda_{k+1} =\nextLambda(x_k, v_k, A_k)$\label{line:next-lambda}
		\Comment{guaranteeing that $\norm{\prox_{f,\lambda_{k+1}}(y_k)-y_k}\le r$}
		$a_{k+1} = \frac{1}{2 \lambda_{k+1}} \sqrt{1 + 4 \lambda_{k+1} A_{k}}$ and $A_{k+1} = A_k + a_{k+1}$\;
		$y_k = \frac{A_k}{A_{k+1}} x_{k}+\frac{a_{k+1}}{A_{k+1}} v_k$\;
		$x_{k+1}=\appProx_{f,\lambda_{k+1}}^{\vphi_{k+1}}(y_k)$ \label{line:sapm-x}
		\Comment{defined in \cref{eq:app-prox-def}}
		$g_{k+1} = \gradest(\gest[f],y_k,\lambda_{k+1},\delta_{k+1},\sigma^2_{k+1},\xset \cap \ball_r(y_k))$\; $v_{k+1} = \proj_{\xset} \prn*{v_k - \half a_{k+1} g_{k+1}}$ 	\label{line:sapm-v}\;
		\lIf{$A_{k+1} \ge \Amax$ \textbf{\textup{or}}
			$k+1 = \Kmax$}
		{\Return{$x_{k+1}$}}
	}
\end{algorithm}
}

The third and final notable difference between  \Cref{alg:proj_eff,alg:MS} is the method of updating the $x_k$ iteration sequence. While a projected stochastic gradient descent step suffices for \Cref{alg:proj_eff}, here we require a more direct approximation of function value decrease attained by the exact proximal mapping $\prox_{f,\lambda}$ (see \cref{eq:def-prox}). For a given accuracy $\vphi$, we define the $\vphi$-approximate proximal mapping 
\begin{equation}\label{eq:app-prox-def}
	\appProx_{f,\lambda}^{\vphi}(y)\defeq\mbox{  any $x\in\xset$ such that }\E F(x) \le F\prn*{\prox_{f,\lambda}(y)}+ \vphi~\mbox{for}~F(z)  \defeq f(z) + \frac{\lambda}{2}\norm{z-y}^2.
\end{equation}
Note that $\appProx_{f,\lambda}^{0}=\prox_{f,\lambda}$ and that for $\vphi>0$ we can compute $\appProx_{f,\lambda}^{\vphi}$ with an appropriate SGD variant (such as $\epochSGD$) using $O(G^2/(\lambda \vphi))$ evaluations of $\gest[f]$. 

With the differences between the algorithms explained, we emphasize their key similarity: both algorithms update the $v_k$ sequence using our bias reduction method $\gradest$ (\Cref{alg:gradest}), which holds the key to their efficiency. 
The following proposition shows that
 \Cref{alg:MS} has the same bound on $\Kmax$ as an exact accelerated proximal point method~\cite{carmon2020acceleration}, while requiring at most $\Otil{G^2R^2 \epsilon^{-2}}$ stochastic gradient evaluations; see proof in  \Cref{app:mtm-sapm}. 

\notarxiv{
\begin{restatable}{proposition}{propSAPM}\label{prop:sapm}
	Let $f:\xset\to\R$ and $\gest$ satisfy Assumption~\ref{ass:gest}, and let $\xset \subseteq \ball_R(x_0)$. For a target accuracy $\epsilon\le GR$ let 
$\vphi_{k+1} = \tfrac{\epsilon}{60\lambda_{k+1}a_{k+1}} 
		$, $\delta_{k+1} = \tfrac{\epsilon}{120R}$, $\sigma_{k+1}^2 =\frac{\epsilon}{60a_{k+1}}$, $A_0 = \frac{R}{G}$, and $\Amax=\frac{9R^2}{\epsilon}$.	If $\lambda_k \ge \lambda_{\min} \ge \frac{1}{\Amax} =\Omega(\frac{\epsilon}{R^2})$ for all $k \le \Kmax$, then  \cref{line:sapm-x,line:sapm-v} of \Cref{alg:MS} have total complexity
	 $\E \Nquery = O\prn[\Big]{ \Kmax\log \frac{GR}{\epsilon} +  \frac{G^2 R^2}{\epsilon^2} \log^2 \frac{GR}{\epsilon}}$. If  in addition $\norm{\prox_{f,\lambda_{k}}(y_{k-1}) - y_{k-1}} \ge 3r/4$ whenever $\lambda_k \ge 2\lambda_{\min}$ then for
$		\Kmax = O\prn[\Big]{ \prn*{\frac{R}{r}}^{2/3} \log\prn*{\frac{GR}{\epsilon}} + \sqrt{\frac{\lambda_{\min} R^2}{\epsilon} }}$, 
	the algorithm's output $x_K$ satisfies $f(x_K) - f(\xopt) \le \epsilon$ with probability at least $\frac{2}{3}$.
\end{restatable}
}
\arxiv{
\begin{restatable}{proposition}{propSAPM}\label{prop:sapm}
	Let $f:\xset\to\R$ and $\gest$ satisfy Assumption~\ref{ass:gest}, and let $\xset \subseteq \ball_R(x_0)$. For a target accuracy $\epsilon\le GR$ let 
\begin{equation*}
	\text{
			$\vphi_{k} = \frac{\epsilon}{60\lambda_{k}a_{k}} 
		$, $\delta_{k} = \frac{\epsilon}{120R}$, $\sigma_{k}^2 =\frac{\epsilon}{60a_{k}}$, $A_0 = \frac{R}{G}$, and $\Amax=\frac{9R^2}{\epsilon}$.
	}
\end{equation*}	
		If $\lambda_k \ge \lambda_{\min} \ge \frac{1}{\Amax} =\Omega\prn*{\frac{\epsilon}{R^2}}$ for all $k \le \Kmax$, then  \cref{line:sapm-x,line:sapm-v} of \Cref{alg:MS} have total complexity
	 \[\E \Nquery = O\prn*{ \prn*{\frac{G R}{\epsilon}}^2 \log^2 \frac{GR}{\epsilon}+ \Kmax\log \frac{GR}{\epsilon} }.\] If  in addition $\norm{\prox_{f,\lambda_{k}}(y_{k-1}) - y_{k-1}} \ge 3r/4$ whenever $\lambda_k \ge 2\lambda_{\min}$ then for
\[		\Kmax = O\prn*{ \prn*{\frac{R}{r}}^{2/3} \log\prn*{\frac{GR}{\epsilon}} + \sqrt{\frac{\lambda_{\min} R^2}{\epsilon} }},\] 
	the algorithm's output $x_K$ satisfies $f(x_K) - f(\xopt) \le \epsilon$ with probability at least $\frac{2}{3}$.
\end{restatable}
}

\subsection{Minimizing the maximal loss}\label{sec:mtm-mtm}

We now consider objectives of the form $\fmax(x) \defeq \max_{i\in[N]} \fhat[i](x)$ where each function $\fhat[i]:\xset\to \R$ is convex and $G$-Lipschitz. Our approach to minimizing $\fmax$ largely follows~\citet{carmon2021thinking}; the main difference is that we approximate proximal steps via \Cref{alg:MS} and our reduced-bias bias estimator. The first step of the approach is to replace $\fmax$ with the ``softmax'' function, defined for a given target accuracy $\epsilon$ as 
\begin{equation*}
	\fsm(x) \defeq \epsilon' \log\prn[\Bigg]{\sum_{i\le N} \exp\prn*{\fhat[i](x)/\epsilon'}},~\mbox{where}~\epsilon' \defeq \frac{\epsilon}{2\log N}.
\end{equation*}
Since $\fsm(x) - \fmax(x)\in [0,\frac{\epsilon}{2}]$, any $\frac{\epsilon}{2}$-accurate solution of $\fsm$ is $\epsilon$-accurate for $\fmax$. 

The second step is to develop an efficient gradient estimator for $\fsm$; this is non-trivial because $\fsm$ is not a finite sum or expectation. In  \cite{carmon2021thinking} this is addressed via an ``exponentiated softmax'' trick; we develop an alternative, rejection sampling-based approach that fits \Cref{alg:MS} more directly (see \Cref{alg:fsm-grad}). To produce an unbiased estimate for $\grad \fsm(x)$ for $x$ in a ball of radius $\reps=\eps'/G$ we require a single $\grad \fhat[i](x)$ evaluation (for some $i$), $O(1)$ evaluations of $\fhat[i](x)$ in expectation, and evaluation of the $N$ functions $\fhat[1](y),\ldots, \fhat[N](y)$ for pre-processing. Plugging this estimator into \Cref{alg:MS} with $r=\reps$, the total pre-processing overhead of \cref{line:sapm-x,line:sapm-v} is $O(\Kmax N)$. 

The final step is to find a function $\nextLambda$ such that  $\norm{\prox_{\fsm,\lambda_{t+1}}(y_k) -y_k} \le \reps$ for all $k$ (enabling  gradient estimation), and $\norm{\prox_{\fsm,\lambda_{t+1}}(y_k)- y_k} \ge \frac{3}{4}\reps$ when $\lambda_{k+1} > 2\lambda_{\min}$  (allowing us to bound $\Kmax$ with \Cref{prop:sapm}). 
Here we use the bisection subroutine from~\cite{carmon2021thinking} as is (see \Cref{alg:bisecx}). 
By judiciously choosing $\lambda_{\min}$---an improvement over the analysis in~\cite{carmon2021thinking}---we obtain the following complexity guarantee on $\NqueryCustom{\fhat[i]}$ and $\NqueryCustom{\del\fhat[i]}$, the total numbers of individual function and subgradient evaluations, respectively. (See proof \Cref{app:mtm-main}).

\notarxiv{
\begin{restatable}{theorem}{thmMinTheMax}\label{thm:min-the-max}
	Let $\fhat[1], \ldots, \fhat[N]: \xset\to \R$ be convex and $G$-Lipschitz and let $\xset \subseteq \ball_R(x_0)$. For any $\epsilon < \half GR/\log N$, \Cref{alg:MS} (with $\appProx_{\fsm, \lambda}^\vphi$ implemented  in~\Cref{alg:epochSGD}, $\gest[\fsm]$ given by \Cref{alg:fsm-grad} and 
	$\nextLambda$ given by \Cref{alg:bisecx} with $\lambda_{\min}= \Otil{\epsilon / (\reps^{4/3} R^{2/3})}$ outputs $x\in\xset$ that with probability at least $\half$ is $\epsilon$-suboptimal for $\fmax(x) = \max_{i\in[N]} \fhat[i](x)$ and has complexity $\E\NqueryCustom{\fhat[i]} = O\prn*{ \brk*{N \prn*{ \frac{GR\log N}{\epsilon}}^{2/3}  + \prn*{\frac{GR}{\epsilon}}^{2}} \log^2\frac{GR}{\epsilon} }$ and $E \NqueryCustom{\del\fhat[i]} = O\prn*{ \prn*{ \frac{GR}{\epsilon}}^{2} \log^2\frac{GR}{\epsilon} }$.
\end{restatable}
\noindent
}
\arxiv{
\begin{restatable}{theorem}{thmMinTheMax}\label{thm:min-the-max}
	Let $\fhat[1], \ldots, \fhat[N]: \xset\to \R$ be convex and $G$-Lipschitz and let $\xset \subseteq \ball_R(x_0)$. For any $\epsilon < \half GR/\log N$, \Cref{alg:MS} (with $\appProx_{\fsm, \lambda}^\vphi$ implemented  in~\Cref{alg:epochSGD}, $\gest[\fsm]$ given by \Cref{alg:fsm-grad}, and 
	$\nextLambda$ given by \Cref{alg:bisecx} with $\lambda_{\min}= \wt{\Theta}({\epsilon / (\reps^{4/3} R^{2/3})})$ outputs $x\in\xset$ that is $\epsilon$-suboptimal for $\fmax(x) = \max_{i\in[N]} \fhat[i](x)$ with probability at least $\half$ and has complexity 
	\[\E\NqueryCustom{\fhat[i]} = O\prn*{ \brk*{N \prn*{ \frac{GR\log N}{\epsilon}}^{2/3}  + \prn*{\frac{GR}{\epsilon}}^{2}} \log^2\frac{GR}{\epsilon} }~\text{and}~E \NqueryCustom{\del\fhat[i]} = O\prn*{ \prn*{ \frac{GR}{\epsilon}}^{2} \log^2\frac{GR}{\epsilon} }.\]
\end{restatable}
}
The rate given by \Cref{thm:min-the-max} matches (up to logarithmic factors) the lower bound $\Omega( N (GR/\epsilon)^2 + (GR/\epsilon)^2)$ shown in \cite{carmon2021thinking} and is therefore near-optimal.  
\arxiv{
When $N \log N = O( (GR/\epsilon)^{4/3})$ the cost of our algorithm is dominated by $\Otil{(GR/\epsilon)^2}$ which is within a logarithmic factor of the optimal rate for minimizing the average of $\fhat[1], \ldots, \fhat[N]$, a seemingly much easier task.
} 
\section{Gradient-efficient composite optimization}\label{sec:gs}

Consider the problem of finding a minimizer of the following convex composite optimization problem
\begin{equation}
	\minimize_{x\in\xset} \Psi(x) \defeq \Lambda(x) + f(x)~\mbox{where $\Lambda$ is $L$-smooth and $f$ satisfies \Cref{ass:gest}},
		\label{eq:composite}
\end{equation}
given $x_0$ such that $\norm{x_0-\xopt}\le R$ for some $\xopt\in\argmin_{x\in\xset}\Psi(x)$. 
\citet{lan2016gradient} developed a method called ``gradient sliding'' that finds an $\epsilon$-accurate solution to~\eqref{eq:composite} with complexity  $\NqueryCustom{\grad \Lambda} = O(\sqrt{LR^2/\epsilon})$ evaluations of $\grad \Lambda(x)$ and $\Nquery = O((GR/\epsilon)^2)$ evaluations of $\gest[f](x)$, which are optimal even for each component separately.\footnote{The gradient sliding result holds under a relaxed Lipschitz assumption~\cite[see][eq.~(1.2)]{lan2016gradient}. It is straightforward to extend $\epochSGD$, and hence all of our results, to that assumption as well.}

In this section, we provide an alternative algorithm that matches the complexity of gradient up to logarithmic factors and is conceptually simple. Our approach, \Cref{alg:GS}, is essentially composite AGD~\citep{nesterov2013gradient}, where at the $k$th iteration we compute a proximal point~\eqref{eq:def-prox} with respect to a partial linearization of $\Psi$ around $y_k$. In particular, letting $\bar{\Lambda}_k(v) \defeq \Lambda(y_k)+\langle\grad \Lambda(y_k),v-y_k\rangle$ and $\beta_k = \frac{2L}{k}$, we approximate $\prox_{\bar{\Lambda}+f,\beta_k}(v_{k-1})$. Similar to \Cref{alg:MS}, \Cref{alg:GS} computes two types of approximations: one is an $\epsilon_k$-approximate proximal point $\appProx_{\bar{\Lambda}+f,\beta_k}^{\epsilon_k}(v_{k-1})$ as per its definition~\eqref{eq:app-prox-def}, while the other is our bias-reduced optimum estimator from \Cref{alg:optest}. We note, however, that unlike \Cref{alg:MS} which approximates the $x_k$ update, here we approximate $v_k$, the ``mirror descent'' update. 

Below we state the formal guarantees for \Cref{alg:GS}; we defer its proof to Appendix~\ref{app:gs}.

\begin{algorithm}[htbp!]
	\DontPrintSemicolon
	\caption{Stochastic composite accelerated gradient descent}\label{alg:GS}
	\KwInput{A problem of the form~\eqref{eq:composite} with $\Lambda$, $f$, $\grad \Lambda$, $\gest[f]$.}
	\KwParameters{Step size parameters $\beta_k = \tfrac{2L}{k}$ and $\gamma_k = \tfrac{2}{k+1}$, iteration number $N$, approximation parameters $\{\eps_k,\delta_k,\sigma_k^2\}$ and $x_0=v_0$ satisfying $\|x_0-\xopt\|\le R$.}
	\For{$k=1,2,\cdots, N$}{
		
		$y_k = (1-\gamma_k)x_{k-1}+\gamma_k \proj_{\xset}(v_{k-1}) $%
		\;
		$\bar{v}_k = \appProx_{\bar{\Lambda}_k+f,\beta_k}^{\eps_k}(v_{k-1})$ for $\bar{\Lambda}_k(v)\defeq \Lambda(y_k)+\langle \grad \Lambda(y_k),v-y_k\rangle$\;
		$v_k=\optest(\gest[f], \psi_k, \beta_k, \delta_k,\sigma_k^2,\ball_R(v_0)\cap \xset)$ for $\psi_k(z)=\frac{\beta_k}{2}\|z-v_{k-1}\|^2+\bar{\Lambda}_k(z)$\;
		$x_k= (1-\gamma_k)x_{k-1}+\gamma_k\bar{v}_k$\; %
	}
	\Return $x_N$
\end{algorithm}

\notarxiv{
\begin{restatable}{theorem}{thmGS}\label{thm:GS}
	Given problem~\eqref{eq:composite} with solution $\xopt$, a point $x_0$ such that $\norm{x_0-\xopt}\le R$ and target accuracy $\epsilon>0$,  \Cref{alg:GS} with $\eps_k = LR/2kN$, $\delta_k = R/16N$, $\sigma^2_k = R^2/4N$, and $N = \Theta( \sqrt{LR^2/\epsilon})$ finds an approximate solution $x$ satisfying $\E \Psi(x)\le \Psi(\xopt)+\eps$ and has complexity
	$\mathcal{N}_{\grad \Lambda} = O\left(\sqrt{\frac{LR^2}{\eps}}\right)$ and $\E \mathcal{N}_{\gest[f]} = O\left(\prn*{\frac{GR}{\eps}}^2\log^2\frac{GR}{\eps}+ 
	\sqrt{\frac{LR^2}{\eps}}\log\left(\frac{GR}{\eps}\right)\right)$.
\end{restatable}
}

\arxiv{
\begin{restatable}{theorem}{thmGS}\label{thm:GS}
	Given problem~\eqref{eq:composite} with solution $\xopt$, a point $x_0$ such that $\norm{x_0-\xopt}\le R$ and target accuracy $\epsilon>0$,  \Cref{alg:GS} with $\eps_k = LR/2kN$, $\delta_k = R/16N$, $\sigma^2_k = R^2/4N$, and $N = \Theta( \sqrt{LR^2/\epsilon})$ finds an approximate solution $x$ satisfying $\E \Psi(x)\le \Psi(\xopt)+\eps$ and has complexity
	\[\mathcal{N}_{\grad \Lambda} = O\left(\sqrt{\frac{LR^2}{\eps}}\right)~~\text{and}~~\E \mathcal{N}_{\gest[f]} = O\left(\prn*{\frac{GR}{\eps}}^2\log^2\frac{GR}{\eps}+ 
	\sqrt{\frac{LR^2}{\eps}}\log\left(\frac{GR}{\eps}\right)\right).\]
\end{restatable}
} 

\newcommand{\diffp}{\alpha}  %
\newcommand{\de}{\beta}  %
\newcommand{\ed}{\ensuremath{(\diffp,\de)}}

\newcommand{\Ds}{\mathcal{S}} %
\newcommand{\ds}{s} %
\newcommand{\domain}{\mathbb{S}} %

\newcommand{\A}{\mathcal{A}}
\newcommand{\cO}{\mathcal{O}}

\newcommand{\xdomain}{\mc{X}} %
\newcommand{\diam}{\mathsf{diam}} %
\newcommand{\opt}{^\star} %
\newcommand{\noise}{\zeta} %

\newcommand{\lip}{G}  %
\newcommand{\sm}{\beta} %
\renewcommand{\sc}{\mu} %
\newcommand{\rad}{R} %

\renewcommand{\ss}{\eta} %

\newcommand{\MLMCorac}{\mathcal{O_\delta}} %
\newcommand{\MLMCoracd}{\mathcal{O}} %

\renewcommand{\ltwo}{\norm}

\vspace{-9pt}

\section{Efficient non-smooth private convex optimization} \label{sec:DP-SCO}

We conclude the paper with a potential application of our optimum estimator for differentially private stochastic convex optimization (DP-SCO). 
In this problem we are given $n$ i.i.d.\ sample $s_i \sim P$ taking values in a set $\domain$, and out objective is to privately minimize the population average  $f(x) = \E_{S \sim P}[\hat f(x;S)]$,
where $\hat f: \xdomain \times \domain \to \R$, 
is convex in the first argument and $\xset \subset \R^d$ is a convex set, and $\domain$ is a population of data points. That is, 
We wish to find $\hat x \in \xdomain$ with small excess loss $f(\hat x) - \min_{x \in \xdomain} f(x)$ 
while preserving %
differential privacy. %

\begin{definition}[\cite{DworkMcNiSm06}]
	\label{def:DP}
	A randomized  algorithm $\A$ is \ed-differentially private (\ed-DP) if, for all datasets $\Ds,\Ds' \in \domain^n$ that differ in a single data element and for every event $\cO$ in the output space of $\A$, we have 
	$P[\A(\Ds)\in \cO] \le e^{\diffp} P[\A(\Ds')\in \cO] +\de$.
\end{definition}

DP-SCO has received increased attention over the past few years. \citet{BassilyFeTaTh19} developed (inefficient) algorithms that attain the optimal excess loss $1/\sqrt{n} + \sqrt{d \log(1/\de)}/n\diffp$. When each function is $O(\sqrt{n})$ smooth, \citet{FeldmanKoTa20} gave algorithms with optimal excess loss and $O(n)$ gradient query complexity.
In the non-smooth setting, however, their algorithms require $O(n^2)$ subgradients. Subsequently, \citet{AsiFeKoTa21} and~\citet{KulkarniLeLi21} developed more efficient algorithms for non-smooth functions which need $O( \min(n^2/\sqrt{d}, n^{5/4}d^{1/8} ,n^{3/2}/d^{1/8}))$ subgradients which is  $O(n^{11/8})$ for the high-dimensional setting $d=n$. Whether a linear gradient complexity is achievable for DP-SCO in the non-smooth setting is still open.

In this section, we develop an efficient algorithm for non-smooth DP-SCO that queries $\Otil{n}$ subgradients conditional on the existence of an optimum estimator with the following properties.

\notarxiv{
\begin{definition}
	\label{def:bounded-MLMC}
	Let $F = f + \psi$ be $\sc$-strongly convex with minimizer  $\xopt$ and  $f$ is $\lip$-Lipschitz. %
	For $\delta >0$, %
	we say that $\MLMCorac$ is efficient bounded low-bias estimator if it 
	returns $ \xmlmc = \MLMCorac(F)$ such that  %
	$\ltwo{\E[\xmlmc - \xopt]}^2 \le \delta^2 $, $\ltwo{\xmlmc - \xopt}^2 \le C_1 \lip^2  \log(\lip/\sc \delta)/\sc^2$,
	and the expected number of gradient queries is $C_2 \log(\lip/\sc\delta)$.
\end{definition}
\vspace{-3pt}
}
\arxiv{
\begin{definition}
	\label{def:bounded-MLMC}
	Let $F = f + \psi$ be $\sc$-strongly convex with minimizer  $\xopt$ and  $f$ is $\lip$-Lipschitz. %
	For $\delta >0$, %
	we say that $\MLMCorac$ is an \emph{efficient bounded low-bias optimum estimator (EBBOE)} if it 
	returns $ \xmlmc = \MLMCorac(F)$ such that  %
	\[\ltwo{\E[\xmlmc - \xopt]}^2 \le \delta^2,~~\ltwo{\xmlmc - \xopt}^2 \le C_1\frac{ \lip^2}{\sc^2}  \log(\lip/\sc \delta),\]
	and the expected number of gradient queries is $C_2 \log(\lip/\sc\delta)$.
\end{definition}
}

Comparing to our MLMC estimator~\eqref{eq:def-mlmc} and \Cref{prop:est}, we note that the only place our current estimator falls short of satisfying~\Cref{def:bounded-MLMC} is the probability 1 bound on $\ltwo{\xmlmc - \xopt}^2$, which for~\eqref{eq:def-mlmc} holds only in expectation. Indeed, for our estimator, $\ltwo{\xmlmc - \xopt}$ can be as large as $O(G/(\mu\delta))$, meaning that it is heavy-tailed.

It is not clear whether an EBBOE as defined above exists. Nevertheless, assuming access to such estimator, \Cref{alg:MLMC-loc-erm} solves the DP-SCO problem with a near-linear amount of gradient computations. The 
algorithm builds on the recent localization-based optimization methods in~\cite{FeldmanKoTa20} which iteratively solve regularized minimization problems. 

\begin{algorithm}[htbp!]
	\DontPrintSemicolon
	\caption{Differentially-private stochastic convex optimization via optimum estimation} \label{alg:MLMC-loc-erm}
	\KwInput{ $(\ds_1, \ldots, \ds_n)\in \domain^n$,
		domain $\xdomain\subset\ball_{\rad}(x_0)$, 
		EBBOE $\MLMCoracd$ (satisfying \Cref{def:bounded-MLMC}).
	}
	Set $k = \ceil{\log n}$, $B = 20 (\log (\frac{1}{\de}) + C_2 \log^2 n)$, $\bar n = \frac{n}{k}$, $\ss =\frac{\rad}{\lip} \min \crl*{\frac{1}{\sqrt{n}}, \frac{ \diffp}{ B \log(n) \sqrt{d \log(\frac{1}{\de}) }} }$\label{line:dp-defs}\;
	\For{$i=1,2,\cdots, k$}{
		Let $\ss_i = 2^{-4i} \ss$ , %
		$
		f_i(x) = \frac{1}{\bar n} \sum_{j=1 + (k-1) \bar n}^{k \bar n} \hat f(x;\ds_j)$,
		$
		\psi_i(x) = %
		\norm{x - x_{i-1}}^2 / (\ss_i \bar n)
		$ \;
		Let $\tilde x_i = \frac{1}{\bar n} \sum_{j=1}^{\bar n} \MLMCoracd_{\delta_i}(F_i)$ with $F_i = f_i + \psi_i$ , $\delta_i^2 = \lip^2 \ss_i^2 \bar n$ \; 
		Set $x_i = \tilde x_i + \noise_i$ where $\noise_i \sim \normal(0,\sigma_i^2 I_d)$ with $\sigma_i = 8  B ( \sqrt{C_1 \log n} + 2) \ss_i \sqrt{\log(2/\de)}/\diffp_i$\;
	}
	\Return  $x_k$\;
\end{algorithm}

We average multiple draws of the (hypothetical) bounded optimum estimator to solve the regularized problems, and apply private mean estimation procedures to preserve privacy. We defer the proof of the following results~\Cref{app:DP-SCO}.

\notarxiv{
\begin{restatable}[conditional]{theorem}{thmDPSCO}
	\label{thm:SCO-convex}
	Given an efficient bounded low-bias estimator $\MLMCorac$ satisfying~\Cref{def:bounded-MLMC} for any $\delta>0$, then for $\diffp \le \log(1/\de)$, $\xset\in\ball_R(x_0)$, convex and $\lip$-Lipschitz $\hat f(x;\ds)$, 
	\Cref{alg:MLMC-loc-erm} is $(\diffp,\de)$-DP, queries $\Otil{n}$ subgradients and has (hiding logarithmic factors in $n$)
	$	\E [ f(x_k) - \min_{x \in \xdomain} f(x)]
		\le \lip \rad \cdot \wt O \left( \frac{1}{\sqrt{n}} + \frac{   \sqrt{d \log^3(1/\de)}  }{n \diffp} \right).$
\end{restatable}
\vspace{-3pt}
}

\arxiv{
\begin{restatable}[conditional]{theorem}{thmDPSCO}
	\label{thm:SCO-convex}
	Given an efficient bounded low-bias estimator $\MLMCorac$ satisfying~\Cref{def:bounded-MLMC} for any $\delta>0$, then for $\diffp \le \log(1/\de)$, $\xset\in\ball_R(x_0)$, convex and $\lip$-Lipschitz $\hat f(x;\ds)$, 
	\Cref{alg:MLMC-loc-erm} is $(\diffp,\de)$-DP, queries $\Otil{n}$ subgradients and satisfies 
	\begin{equation*}
		\E [ f(x_k) - \min_{x \in \xdomain} f(x)]
		\le \lip \rad \cdot O \left( \frac{{\log n}}{\sqrt{n}} + \frac{   \sqrt{d \log(1/\de)} (\log^2(n) + \log(1/\de)) \log(n)  }{n \diffp} \right) . 
     \end{equation*}
\end{restatable}
}

\Cref{thm:SCO-convex} provides a strong motivation for %
constructing bounded optimum estimators that satisfy~\Cref{def:bounded-MLMC} .
In~\Cref{app:DP-SCO-challenges}, we discuss the challenges in making our MLMC estimator bounded, as well as some directions to overcome them.

\section*{Acknowledgments}
HA was supported by ONR YIP N00014-19-2288 and the Stanford DAWN Consortium.
YC was supported in part by Len Blavatnik and the Blavatnik Family foundation, and the Yandex Machine Learning Initiative for Machine Learning. YJ was supported by a Stanford Graduate Fellowship. AS was supported in part by a Microsoft Research Faculty Fellowship, NSF CAREER Award
CCF-1844855, NSF Grant CCF-1955039, a PayPal research award, and a Sloan Research Fellowship.

\bibliographystyle{abbrvnat}

\newpage

\appendix
\part*{Appendix}
\section{Additional results and discussion}\label{app:discussion}

Here we provide additional discussion of three topics pertinent to our results: a zero-bias optimum estimator,  
the parallel depth of our estimator, turning expected complexity bounds into deterministic ones, and a justification for the adjective ``optimal'' in \Cref{def:ODC} of the  ``optimal distance convergence''  property. We recommend reading \Cref{sec:intro,sec:estimator,sec:proj-eff} before the subsections below. 

\subsection{Zero-bias optimum estimation given exact gradients}\label{app:remark-zero-bias}

The main tool developed in this paper is an optimum estimator with bias $\delta$ whose expected query complexity is $O(\log(G/(\mu \delta)))$. In this section, we show how to obtain a completely unbiased optimum estimator when, in addition to a stochastic subgradient oracle, we assume access to a first-order oracle, i.e., one which outputs the functions exact value and subgradient at the query point. 

To be concrete, assume that the domain is a ball of radius $R$ in $\R^d$, and that the objective $F:\ball_R(x_0)\to \R$ is $\mu$-strongly-convex and of the form $F(x)=\frac{1}{n}\sum_{i\in[n]}\hat{F}(x;i)$, where each $\hat{F}(\cdot; i)$ is $G$-Lipschitz and given by a first-order oracle. In this case, we can compute an unbiased subgradient estimator with single oracle query by sampling $i\sim\uniform([n])$ and taking $\gest[F] \in \del  \hat{F}(x;i)$. Further, value and subgradient evaluations of $F$ can be implemented at $n$-times the cost by querying each each $F_i$. In this setting, we design an unbiased estimator of $\xopt = \argmin_{x\in\ball_R(0)} F(x)$ with variance $O((G^2/\mu^2) \log(nd))$, expected query complexity $O(\log(nd))$ and expected runtime $O(d \log(nd))$. 

To obtain this result, we leverage that first-order methods can compute the minimizer of a convex function with a number of queries and runtime that depends polynomially on dimension and logarithmically on regularity parameters and the desired accuracy. In fact, any polynomial bound suffices for our purposes and effect only constants factors in our expected complexity bounds. For concreteness, we use the classic ellipsoid method \cite{yudin1976evaluation, shor1977cut, khachiyan1979polynomial} whose complexity we describe in the following lemma. (We remark, however, that improved query complexities and runtimes are achievable; see \cite{JiangLSW20} for the state-of-the-art).

\begin{lemma}[Ellipsoid method]
\label{lem:ellipsoid}	
There is an algorithm, $\Ellipsoid(x_0, f, T)$, which given $x_0 \in \R^d$, a first order oracle for $G$-Lipschitz, $\mu$-strongly-convex $f : \ball_R(x_0) \rightarrow \R$, and query budget $T \geq 0$, runs in $O(d^2 T)$ time, makes at most $T$ queries, and outputs $\xmlmc \in  \ball_R(x_0)$ with $\norm{\xmlmc - \xopt}_2^2 \leq  (8 G^2 / \mu^2) \exp(-T/(2 d^2))$ for $\xopt \defeq \argmin_{x \in \ball_R(x_0)} f(x)$.
\end{lemma}

\begin{proof}
Since $f$ is $G$-Lipschitz for all $x \in \ball_R(x_0)$ we have $|f(x) - f(x_0)| \leq G R$. Consequently, the ellipsoid method applied to $f(x) - f(x_0)$ can compute $\xmlmc \in \ball_R(x_0)$ with $f(\xmlmc) - f(\xopt) \leq 2 G R \exp(-T/(2 d^2))$ with $O(T)$ queries and $O(d^2 T)$ time~\cite[see, e.g.,][Theorem~2.4]{Bubeck15}. Since by strong convexity $\norm{\xmlmc - \xopt}_2^2 \leq \frac{2}{\mu} [ f(\xmlmc) - f(\xopt)]$ this implies that $\norm{\xmlmc - \xopt}_2^2 \leq  (4 G R / \mu) \exp(-T/(2 d^2))$. Further, since $f$ is $G$-Lipschitz and $\mu$-strongly-convex we know that for all $y \in \ball_R(x_0)$ we have $G \norm{y - \xopt} \geq f(y) - f(\xopt) \geq \frac{\mu}{2} \norm{y - \xopt}^2$ and since $\ball_{R}(x_0)$ contain a point $y$ with $\norm{y - \xopt} \geq R$ this implies $R \leq 2 G / \mu$. Combining yields the result.
\end{proof}

\begin{algorithm}[b]
	\DontPrintSemicolon
	\KwInput{Initialization $x_0 \in \R^d$, first-order oracles for $\hat{F}(x;i)$ for all $i \in [n]$ and ODC algorithm $\ODC$.}
	Let  $J_0 \defeq \lceil 4 \log_2(14 (nd^2 + d^4)) \rceil$ \;
	For all $j > 1$ let $x_{j} \defeq \begin{cases}
		\ODC(\grad \hat{F}(\cdot; i), 0, 2^j)  & \text{if } j \leq J_0\\
		\Ellipsoid(x_0, F, \lceil 2^{j/2} \rceil) 	& \text{if } j > J_0
	\end{cases}$ \\
	Draw $J  \sim \geoRV\left(\tfrac{1}{2}\right)$\;
	\Return $x_{0} + 2^{J} (x_{J -1 } - x_{J})$ \Comment{Only $x_0$, $x_{J -1}$, and $x_J$ are computed explicitly by the algorithm.}
	\caption{Unbiased optimum estimator}%
	\label{alg:unbiased_opt_est}
\end{algorithm}

Combining the ellipsoid method with an $\ODC$ algorithm (see \Cref{def:ODC}) we obtain our unbiased optimum estimator, which we formally describe in \Cref{alg:unbiased_opt_est}. The procedure is similar to the MLMC estimator~\eqref{eq:def-mlmc}, with the key difference that when~\eqref{eq:def-mlmc} would output $x_0$, we instead apply the ellipsoid method. The following theorem establishes the performance of our algorithm.

\begin{theorem}[Unbiased optimum estimator]  
\label{thm:unbiased_est} Let $F : \xset \rightarrow \R$ be $\mu$-strongly convex with, $\xset = \ball_R(x_0)$, $F(x) = \frac{1}{n} \sum_{i \in [n]} \hat{F}(x;i)$ for all $x \in \xset$ and $\E \norm{\grad F(x;i)}^2 \le G^2$ for all  $x \in \xset$ and  $i\sim\uniform([n])$. Algorithm~\ref{alg:unbiased_opt_est} outputs $\xmlmc$ with $\E \xmlmc = \xopt = \argmin_{x \in \ball_R(x_0)} F(x)$ and $\E \norm{x - \xopt}_2^2 = O(\frac{G^2}{\mu^2} \log(nd))$ with expected $O(\log(nd) )$ queries to $(\hat{F}(x;i),\grad \hat{F}(x;i))$ and expected $O(d \log(nd))$ time.
\end{theorem}

\begin{proof}
Note that 
\begin{align*}
	\E \norm{\xmlmc - \xopt}^2
	&= \sum_{j = 1}^{\infty} \frac{1}{2^j} \cdot \E \norm{x_0 - \xopt + 2^j (x_{j - 1} - x_j ) }^2 \\
	&\leq  \sum_{j = 1}^{\infty} \frac{2}{2^j} \E \left[\norm{x_0 - \xopt}^2  + 2^{2j} \norm{x_{j -1} - x_j}^2\right] \\
	&\leq 2 \E \norm{x_0 - \xopt}^2  + 4 \sum_{j = 1}^{\infty} 2^{j} \E \left[ \norm{x_{j -1} - \xopt}^2 + \norm{x_j - \xopt}^2 \right] \\
	&= 10 \norm{x_0 - \xopt}^2 + 4 \sum_{j = 1}^{\infty} (2^j + 2^{j +1}) \E \norm{x_j - \xopt}^2
	\leq 12 \sum_{j = 0}^{\infty} 2^{j} \E \norm{x_j - \xopt}^2.
\end{align*}
Further, by definition of $\ODC$ we have that $\E \norm{x_j - \xopt}_2^2 \leq (c G^2 / \mu^2) 2^{- j/2}$ for all $j \leq J_0$ where $c$ is the constant in \Cref{def:ODC}. Also, by \Cref{lem:ellipsoid} we have $\norm{x_j - \xopt}_2^2 \leq (8 G^2 / \mu^2) \exp(- \lceil 2^{j/2}\rceil /(2 d^2))$ for all $j > J_0$ (since by assumption and Jensen's inequality for $i\sim\uniform([n])$ we have $\norm{\nabla F(x)}^2 = \norm{\E \nabla F(x;i)}^2 \leq  \E \norm{\nabla F(x,i)}^2 \leq G^2$ for alll $x \in \xset$ and therefore $F$ is $G$-Lipschitz).
Note that $j \leq 5 \cdot 2^{j/4}$ for all $j \geq 1$ and $5 \cdot 2^{j/4} \ln 2 \leq 2^{j/2} / (4d^2)$ for all $j \geq 4 \log_2(14 d^2)$. Consequently, $ \frac{\lceil 2^{j/2}\rceil}{2d^2} \geq 2 j \ln 2$ and  $\norm{x_j - \xopt}_2^2 \leq (8 G^2 / \mu^2) 2^{-2j}$ for all $j > J_0$. Therefore, 
\begin{align*}
\E \norm{\xmlmc - \xopt}^2
\leq 
12 \sum_{j = 0}^{J_0} \frac{c G^2}{\mu^2}
+ 12 \sum_{j = J_0 + 1}^{\infty} \frac{8 G^2}{\mu^2} 2^{-j} 
\leq 12 (c J_0 + 8) \frac{G^2}{\mu^2}
= O\left( \frac{G^2}{\mu^2} \log(nd)  \right) ~.
\end{align*}
Further,
\[
\E \xopt = \sum_{j = 1}^{\infty} 2^{-j} [x_0 + 2^{j} (x_j - x_{j - 1})]
= x_0 + \sum_{j = 1}^{\infty} (x_j - x_{j-1})
= \lim_{j \rightarrow \infty} x_j = \xopt\,.
\]
Now, note that when $J \leq J_0$ the algorithm makes $2^J$ subgradient queries and runs in time $O(d 2^J)$. Further, when $J > J_0$ by \Cref{lem:ellipsoid} the algorithm makes $\lceil 2^{J/2} \rceil \leq 2^{1 + (J/2)}$ first-order oracle queries, costing $O(n 2^{(J/2)})$ sub-gradient in total, and runs in time $O((nd^2 + d^4) 2^{J/2})$. Consequently, the expected number of subgradient queries is upper bounded by
\begin{align*}
&\sum_{j \in [J_0]} \frac{1}{2^j} \cdot 2^j 
+ \sum_{j > J_0}^{\infty} \frac{1}{2^j} \cdot n 2^{1 + (j/2)}
= J_0 + \frac{2n}{2^{J_0 /2}} \sum_{j = 1}^{\infty} \frac{1}{2^{j/2}} 
= O(J_0) = O(\log(nd))
\end{align*}
where in the last step we used that $J_0 = \Omega(\log(n))$ and $\sum_{j = 1}^{\infty} \frac{1}{2^{j/2}} = O(1)$. Similarly, since $J_0 \geq \log_2(nd^2 + d^4)$ the expected runtime is at most
\begin{align*}
	&\sum_{j \in [J_0]} \frac{1}{2^j} \cdot O(d 2^j) 
	+ \sum_{j > J_0}^{\infty} \frac{1}{2^j} \cdot O\left((n d^2 + d^4)2^{j/2} \right) \\
	& \hspace{16pt}= O(J_0 \cdot d) +  O\left( 2^{-J_0} (n d^2 + d^4)\right) \cdot \sum_{j = 1}^{\infty} \frac{1}{2^{j/2}}
	= O(J_0 \cdot d) = O(d \log(nd))~.
\end{align*}
\end{proof}
\subsection{The sequential depth of our optimum estimator}\label{app:remark-parallel}

Let us discuss the implications of our development---or more precisely, the lack thereof---on the parallel complexity of non-smooth optimization. Following the standard setting for this problem, consider the task of minimizing a $G$-Lipschitz convex function $f$ in a domain of diameter $R$ in $\R^d$ given the ability to query a subgradient oracle for $f$ in batches of $B$ parallel queries. That is, at round $t$ we query points $x_t^{(1)}, \ldots, x_t^{(B)}$ and observe subgradients $g_t^{(i)}\in\del f(x_t^{(i)})$ for $i\in[B]$. In sufficiently high dimension, the ability to query $B$ points in parallel does not improve worst-case complexity: for required accuracy $\epsilon$ and algorithm with batch size $B=\mathsf{poly}(1/\epsilon)$, there exists a problem instance in dimension $d=O( (\frac{GR}{\epsilon})^4 \log \frac{GR}{\epsilon})$ for which the algorithm must make $T=\Omega( (\frac{GR}{\epsilon})^2)$ queries in sequence in order to find an $\epsilon$-accurate solution~\cite{bubeck2019complexity}. 

At first glance, our algorithms---and \Cref{alg:proj_eff} in particular---seem to contradict the lower bound described above. Indeed, the algorithm performs $O(\frac{GR}{\epsilon})$ iterations, where each iterations consists of averaging  $\Otil{\frac{GR}{\epsilon}}$ copies of the optimum estimator~\eqref{eq:def-mlmc}. Since we can compute copies of the estimator in parallel, the sequential depth of the algorithm appears to be only $O(\frac{GR}{\epsilon})$. To resolve the apparent contradiction, recall that each evaluation of~\eqref{eq:def-mlmc} itself involves a sequential computation. In particular, while an evaluation of~\eqref{eq:def-mlmc} has depth $\Otil{1}$ on average, it also has depth $\Omega(\frac{GR}{\epsilon})$ with probability $\Omega(\frac{\epsilon}{GR})$. Therefore, for a batch of $O(\frac{GR}{\epsilon})$ copies of the estimator, one of them would have depth $\Omega(\frac{GR}{\epsilon})$ with constant probability, implying an overall bound of $\Omega( (\frac{GR}{\epsilon})^2)$ on the sequential depth of \Cref{alg:proj_eff}.

Viewed another way, the parallelism lower bound implies a limitation on the sequential depth \emph{distribution} of any lower bias optimum estimator. More specifically, let $\hat{T}$ be a random variable representing the sequential depth of a single copy of a low-bias optimum estimator and let $\hat{T}_1, \ldots, \hat{T}_{BK}$ be i.i.d.\ copies of that random variable, with $B$ and $K$ denoting batch size and AGD depth respectively. Then, when setting $B=K=O(\frac{GR}{\epsilon})$ we must have
\begin{equation*}
	\sum_{k\in [K]} \max_{b\in [B]} \crl*{\hat{T}_{b+(k-1)B}} = \Omega\prn*{ \prn*{\frac{GR}{\epsilon}}^2}
\end{equation*}
with high probability. In particular, it is impossible to create a low-bias optimum estimator whose depth is $\Otil{1}$ with high probability. This fact might serve as a useful sanity check when designing new optimum estimators.

\subsection{Obtaining deterministic complexity bounds}\label{app:remark-total-work}
This paper measures complexity via $\Nquery$, the number of gradient estimator evaluations by the algorithm. The performance guarantees of our algorithms bound the \emph{expected complexity} while guaranteeing correctness with constant probability. In particular our guarantees in \Cref{sec:proj-eff,sec:mtm,sec:gs} have the following general form: the algorithm outputs $x$ such that 
$f(x)-\min_{z\in\xset}f(z)\le \epsilon$ with probability at least $p$, and $\E \Nquery \le \mathsf{C}(\epsilon)$. To guarantee a probability 1 bound on $\Nquery$, we may terminate the algorithm and output an arbitrary point whenever $\Nquery$ exceeds $\frac{2}{p}\mathsf{C}(\epsilon)$. By Markov's inequality such termination occurs with probability at most $p/2$ and therefore by the union bound we will output  a correct $x$  (satisfying $f(x)-\min_{z\in\xset}f(z)\le \epsilon$) with probability at least $p/2$.

In \Cref{sec:DP-SCO} we describe a differentially-private algorithm with bounded expected error and expected gradient estimation complexity. Here too, we may terminate the algorithm if the number of gradient estimations exceeds the bound on the expectation by more than a constant, and maintain a constant probability bound on the error. Since the random amount of gradient estimations in this algorithm is independent of the input (and in fact can be computed ahead of the algorithm's execution), the termination strategy described above does not affect the algorithm's privacy guarantee. 

\subsection{The optimal distance convergence rate}\label{app:remark-odc}
\newcommand{\Fhard}{F}
\Cref{def:ODC} of an optimal-distance-convergence (ODC) algorithm implies a claim on the optimal rate of convergence (in Euclidean norm) to the minimizer  of strongly-convex and Lipschitz functions. \Cref{lem:epoch-SGD} shows that this rate is achievable, and here we sketch a matching lower bound, showing that this rate is not improvable and therefore optimal. More precisely, we exhibit a function $F$ that is $G$-Lipschitz, $\mu$ strongly-convex, has minimizer $\xopt$ and satisfies the following: for every algorithm that queries points in the span of previously observed subgradients and outputs $x_T$ after $T$ queries, we have $\norm{x_T-\xopt} \ge \Omega({G}/{(\mu\sqrt{T})})$. The restriction of queries to the span of previous gradients is a standard simplifying assumptions~\cite{nesterov2018lectures}, and we can extend the claim to any randomized algorithm by choosing a random coordinate system~\cite{woodworth2016tight,bubeck2019complexity}. 

Let us describe our hard instance construction for algorithms that execute $T$ steps, which we denote by $\Fhard$. The function $\Fhard:\R^{2T}\to \R$  is a strongly-convex variant of Nemirovski's function~\cite{nemirovski1983problem,nemirovski1994parallel,diakonikolas2019lower,bubeck2019complexity}, defined as follows
\begin{equation*}
	\Fhard(x) \defeq \frac{G}{2}\max_{i\in [2T]} \crl*{x_{[i]}} + \frac{\mu}{2} \norm{x}^2.
\end{equation*}
Note that the function is $\mu$-strongly-convex, and---when constrained to a ball of radius $G/(2\mu)$ around the origin---is $G$-Lipschitz as required. It is also easy to verify that the minimizer of the function is
\begin{equation*}
	\xopt = - \frac{G}{4\mu T} \ones,
\end{equation*}
where $\ones$ denotes to the all-ones vector in $\R^{2T}$, since a calculation shows that $0\in\del \Fhard(\xopt)$. 

To establish our claimed lower bound, consider a subgradient oracle for $\max_{i\in [2T]} \crl{x_{[i]}}$ which only outputs 1-sparse subgradients of $F$ (it is also possible to design differentiable hard instances via Moreau-Yoshida smoothing, see, e.g.,~\cite{diakonikolas2019lower}). Then, the query $x_T$ at iteration $T$ is in the span of $T$ 1-sparse vectors, which means that at least  $T$ of its  coordinates are zero. Recalling the expression of $\xopt$, this implies the claim that 
\begin{equation*}
	\norm{x_T-\xopt} \ge \frac{G}{4\mu T}\sqrt{T} = \Omega\prn*{\frac{G}{\mu \sqrt{T}}}\,.
\end{equation*}

\section{Proofs and additional results from Section~\ref{sec:estimator}}\label{app:estimator}

\subsection{Analysis of $\epochSGD$}\label{app:epochSGD}

\Cref{alg:epochSGD} is a composite variant of the ``epoch SGD'' algorithm of~\citet{hazan2014beyond}. We note that when $\psi(x) = \frac{\mu}{2}\norm{x-z}^2$ (as it is in all of our applications), the gradient step in \cref{line:composite-update} of the algorithm is simply
\begin{equation*}
	x_k^{t+1} = \proj_{\xset}\prn*{\frac{1}{1+\mu\eta_k}\brk*{x_k^t + \mu \eta_k z - \eta_k \gest[f](x_k^t)} },
\end{equation*}
where $\proj_{\xset}$ is the Euclidean projection to $\xset$. 
To analyze \Cref{alg:epochSGD}, we first prove the following standard single-epoch optimization guarantee. Below, we let $V_{x}(x') \defeq \frac{1}{2}\|x'-x\|_2^2$ denote the Bregman divergence induced by $\half\norm{\cdot}_2^2$.

\begin{algorithm}
	\DontPrintSemicolon
	\caption{$\epochSGD(\gest[f],\psi,\mu,\xset,T)$}	\label{alg:epochSGD}
	\KwInput{A $\mu$-strongly-convex function $F = f+\psi:\xset\to \R$ with $f$ satisfying \Cref{ass:gest}, iteration budget $T$.}
	\KwParameters{Initial step size $\eta_1 = 1/(4\mu)$ and epoch length $T_1 = 16$.}
	Initialize $x_1^0 \in \arg\min_{x\in\xset}\psi(x)$, and set  $k=1$\;
	\While{$\sum_{i\in[k]} T_i\le T$}{
	$x^{1}_k = \arg\min_{x\in\xset} \left(\eta_k\psi(x)+\frac{1}{2}\|x-x^t_k\|^2\right)$\;\label{line:composite-update-first}
	\For{$t = 1,2,\cdots T_k-1$}{
	$x^{t+1}_k = \arg\min_{x\in\xset} \left(\eta_k\left( \langle \gest(x^t_k),x\rangle+\psi(x)\right)+\frac{1}{2}\|x-x^t_k\|^2\right)$\; \label{line:composite-update}
	}
	Set $x^0_{k+1} = \tfrac{1}{T_k}\sum_{t\in[T_k]}x^t_k$, update $T_{k+1}= 2T_k$, $\eta_{k+1}= \eta_k/2$ and $k\gets k+1$\;
	}
		\Return $x = x^0_k$
\end{algorithm}

\begin{lemma}\label{lem:perepoch}
	Let $f:\xset\rightarrow \R$ and $\gest[f]$ satisfy \Cref{ass:gest}. For any $k\ge 1$, $T\ge 	1$ and $u\in\xset$, the iterates of \Cref{alg:epochSGD} satisfy
	\[\E\left[F\left(\frac{1}{T}\sum_{t\in[T]}x^t_k\right)\right]-F(u)\le \frac{V_{x^0_k}(u)}{\eta T} + \frac{\eta}{2} G^2.\]
\end{lemma}

\begin{proof}
	$x_k^t \equiv x^t$, $\eta_k \equiv \eta$, and $T_k \equiv T$. We furthermore let $x^{T+1} \equiv u$, and let $g^t \defeq \gest[f](x^t)$ for $t \geq 1$ and $g^0 \defeq 0$. 
	By the optimality conditions of the minimization in \cref{line:composite-update-first} and \cref{line:composite-update}, we have
	\[
	\langle\eta\left(g^{t-1}+\nabla \psi(x^t)\right)+x^{t}-x^{t-1},x^{t}-u\rangle\le 0~\mbox{for all}~t\in [T],
	\]
	and consequently
	\begin{equation*}
		\langle g^{t-1}+\nabla \psi(x^{t}),x^t-u\rangle \le \frac{1}{\eta} \prn*{V_{x^{t-1}}(u)-V_{x^{t}}(u)-V_{x^{t-1}}(x^t) }~\mbox{for all}~t\in [T].
	\end{equation*}
    Using the convexity of $\psi$ and the bound above, we obtain 
	\begin{flalign*}
	\hspace{32pt} & \hspace{-32pt} \sum_{t\in[T]}\langle g^{t-1},x^t-u\rangle+\sum_{t\in[T]}\left(\psi(x^t)-\psi(u)\right)\
	\\ &
	\le \sum_{t\in[T]}\langle g^{t-1}+\nabla \psi(x^{t}),x^t-u\rangle
	\\ & \le 
	 \frac{1}{\eta}\sum_{t\in[T]}\left(V_{x^{t-1}}(u)-V_{x^{t}}(u)-V_{x^{t-1}}(x^t)\right)
	 \\ & 
	 \le \frac{1}{\eta}V_{x^0}(u)-\frac{1}{\eta}\sum_{t=0}^T V_{x_t}(x^{t+1}).
	\end{flalign*}
Adding $\sum_{t\in[T]}\langle g^t,x^t-x^{t+1}\rangle$ to both sides, recalling that $x^{T+1}\equiv u$ and $g^t = \gest[f](x^t)\indic{t>0}$, and rearranging terms, we have
\begin{flalign*}
	\hspace{32pt} & \hspace{-32pt} \sum_{t\in[T]}\langle\gest(x^{t}),x^t-u\rangle+\sum_{t\in[T]}\left(\psi(x^t)-\psi(u)\right)	\\&
	\le  \frac{1}{\eta}V_{x^0}(u)-\frac{1}{\eta}\sum_{t=0}^T V_{x_t}(x^{t+1})+\sum_{t\in[T]}\langle \gest(x^{t}),x^t-x^{t+1}\rangle
		\\& \le  \frac{1}{\eta}V_{x^0}(u)+\sum_{t\in[T]}\frac{\eta}{2}\norm{\gest(x^t)}^2,
	\end{flalign*}
where in the last transition we used  $\inner{g}{x-y} \le \frac{1}{\eta}V_y(x) + \frac{\eta}{2} \norm{g}^2$. 
Taking expectation, applying \Cref{ass:gest} and using convexity of $f$, we have
\begin{equation*}
	\E \sum_{t\in[T]} \prn*{ F(x^t) - F(u)} \le \frac{1}{\eta} V_{x^0}(u) + \frac{T}{2} \eta G^2.
\end{equation*}
Dividing by $T$ and applying Jensen's inequality to bound $F\left(\frac{1}{T}\sum_{t\in[T]}x^t\right) \le \frac{1}{T}\sum_{t\in[T]}  F(x^t)$  yields the claimed bound. 
\end{proof}

We now are ready to prove the main guarantee of  Algorithm~\ref{alg:epochSGD} (see also Lemma 8, Theorem 5 in~\citet{hazan2014beyond}), which implies \Cref{lem:epoch-SGD}.

\begin{proposition}\label{prop:epoch-sgd}
	Let $F:\xset \to \R$ by a $\mu$-strongly-convex function of the form $F=f+\psi$, such that $f$ satisfies \Cref{ass:gest} and $\xopt=\argmin_{x\in\xset}F(x)$. Then, for any $T\ge 1$, we have that $x=\epochSGD(\gest[f], \psi, \mu, \xset, T)$ satisfies 
	\begin{equation*}
		\E F(x) - F(\xopt) \le \frac{16G^2}{\mu T}~~\mbox{and}~~\E\norm{x-\xopt}^2 \le \frac{32G^2}{\mu^2 T}.
	\end{equation*}
	Consequently, $\epochSGD$ is an ODC algorithm with constant $c=32$. 
\end{proposition}

\begin{proof}

First we claim that $F(x_1^0)-F(\xopt)\le \frac{G^2}{2\mu}$. To see this  we have by $\mu$-strong-convexity of $F$ that 
\begin{align*}
F(\xopt) & \ge F(x_1^0)+\langle \grad f(x_1^0),\xopt-x_1^0\rangle + \langle\grad \psi(x_1^0),\xopt-x_1^0\rangle+\frac{\mu}{2}\norm{x_1^0-\xopt}^2\\
& \ge F(x_1^0)+\langle \grad f(x_1^0),\xopt-x_1^0\rangle +\frac{\mu}{2}\norm{x_1^0-\xopt}^2,
\end{align*}
where we use the definition that $x_1^0 \in  \arg\min_{x\in\xset}\psi(x)$ and its first-order optimality condition for the second inequality. Rearranging terms gives 
\begin{align*}
	F(x_1^0) - F(\xopt)  & \le -\langle \grad f(x_1^0),\xopt-x_1^0\rangle - \frac{\mu}{2}\norm{x_1^0-\xopt}^2\\
& \le \max_{x} \left(-\langle \grad f(x_1^0),x-x_1^0\rangle -\frac{\mu}{2}\norm{x_1^0-x}^2\right) = \frac{\norm{\grad f(x_1^0)}^2}{2\mu}\le\frac{G^2}{2\mu}.
\end{align*}

For $\xopt=\argmin_{x\in\xset}F(x)$, we so define the potential $\Delta_k = F(x^0_k)-F(\xopt)$ and use induction to prove that  $\E\Delta_k\le \frac{G^2}{2^{k}\mu}$ for all $k$, with the base case $k=1$ established above. Suppose that $\E\Delta_k\le \frac{G^2}{2^{k}\mu}$ for a fixed $k$. Then for $k+1$ \Cref{lem:perepoch} yields
\[
\E\Delta_{k+1} \le \frac{\E V_{x^0_k}(\xopt)}{\eta_k T_k}+\frac{\eta_k}{2} G^2 \overle{(i)} \frac{\E\Delta_k}{\mu\eta_k T_k}+\frac{\eta_k}{2} G^2 \overeq{(ii)} \frac{\E\Delta_k}{4}+\frac{G^2}{2^{k+2}\mu} \overle{(iii)} \frac{G^2}{2^{k+1}\mu},
\]
with the transitions above following from $(i)$ strong convexity of $F$, which implies that $V_{x_k^0}(\xopt) = \frac{1}{2}\norm{x_k^0-\xopt}^2 \le \frac{1}{\mu}\Delta_k$; $(ii)$ the choice of parameters ensures $\eta_k T_k = \tfrac{4}{\mu}$ and $\eta_k = \tfrac{1}{2^{k+1}\mu}$; and $(iii)$ the inductive hypothesis $\E\Delta_k\le \frac{G^2}{2^{k}\mu}$. This completes the induction.

Let $K$ be such that the algorithm outputs $x=x_K^0$, and note that $T\le 16\cdot (2^{K}-1)-1$. Therefore, we have
\begin{equation*}
	\E F(x) - F(\xopt) = \E \Delta_K \le \frac{G^2}{2^{K}\mu} \le \frac{16 G^2}{\mu T},
\end{equation*}
and
\begin{equation*}
	\E \norm{x-\xopt}^2 \le \frac{2}{\mu} (\E F(x) - F(\xopt)) \le \frac{32 G^2}{\mu^2T}.
\end{equation*}
Recalling \Cref{def:ODC}, we conclude that $\epochSGD$ is an ODC algorithm with constant $c=32$.
\end{proof}

\begin{remark}[{Using $\epochSGD$ for optimum estimation}]
	When using $\epochSGD$ as the ODC algorithm in our MLMC optimum estimator~\eqref{eq:def-mlmc}, we need only call once with $T=2^J$ and take $x_0, x_{J-1}$ and $x_{J}$ to be the iterates $x_1^0, x_{K-1}^0$ and $x_K^0$ of $\epochSGD$, for $K$ the last value of $k$ that $\epochSGD$ reaches.
\end{remark}

\subsection{Proof of \Cref{thm:optest}}\label{app:optest}

\restateThmOptest*
\begin{proof}
	Write the algorithm's output as $\xmlmc=\frac{1}{N}\sum_{i=1}^N \xmlmc^{(i)}$ where $\xmlmc^{(1)},\ldots, \xmlmc^{(N)}$ are independent draws of the estimator~\eqref{eq:def-mlmc}, with 
	\begin{equation*}
		\Tmax= \ceil*{\frac{(2c)^2G^2}{\mu^2\min\crl{\delta^2, \half\sigma^2}}}~\mbox{and}~N= \ceil*{ \frac{2(4c)^2 G^2}{\mu^2\sigma^2}\log (\Tmax) } 
	\end{equation*}
	as in \Cref{alg:optest}. Then, \Cref{prop:est} implies that 
	\[\norm{\E \xmlmc^{(1)} - \xopt} \le \min\crl*{\delta, \frac{1}{\sqrt{2}}\sigma}
	~\mbox{and}~
	\E\norm{ \xmlmc^{(1)} - \E \xmlmc^{(1)}}^2 \le \frac{N}{2}\sigma^2.\]
	 Noting that $\E \xmlmc = \E \xmlmc^{(1)}$ and 
	 \[\E\norm{ \xmlmc - \xopt}^2 = \frac{1}{N}\E\norm{ \xmlmc^{(1)} - \E \xmlmc^{(1)}  }^2 + \norm{\E \xmlmc^{(1)} - \xopt}^2,\]
	 we obtain the claimed bias and error bounds. Finally, \Cref{prop:est} guarantees that $\E \Nquery = O( N \cdot \log(\Tmax) )$, giving the claimed bound on the number of evaluations.
\end{proof}

\subsection{Properties of the proximal operator and Moreau envelope}\label{app:moreau-properties}

\newcommand{\proxset}{\xset}
\newcommand{\envmin}{\rho}
\newcommand{\xl}{x_{\lambda}}
\newcommand{\yl}{y_{\lambda}}
\newcommand{\halflam}{\frac{\lambda}{2}}

For a convex function $f:\xset\to \R$ we  recall the definitions of 
\begin{equation*}
\begin{aligned}
	\text{the proximal operator}~\prox_{f,\lambda}(x) & \defeq \argmin_{y\in\xset}\crl*{f(y) + \tfrac{\lambda}{2}\norm{y-x}^2}\\
	\text{and the Moreau envelope}~f_\lambda(x) & \defeq \min_{y\in\xset}\crl*{f(y) + \tfrac{\lambda}{2}\norm{y-x}^2}.
\end{aligned}
\end{equation*}
Below, we collect several well-known properties that we use throughout the paper. 
\begin{fact} \label{fact:moreau}
Given a convex function $f:\xset\rightarrow \R$, and $\lambda>0$ defined on a closed convex set $\xset$, the following properties of the Moreau envelope $f_\lambda:\R^d\rightarrow \R$ and the proximal operator $\prox_{f,\lambda}:\xset \to \xset$ hold for all $x\in\xset$
\begin{enumerate}
\item \label{item:moreau-convex} \emph{Convexity}: $f_\lambda$ is convex.
\item \label{item:moreau-smooth} \emph{Differentiablility}: $f_\lambda$ is $\lambda$-smooth and  $\grad f_\lambda(x) = \lambda(x-\prox_{f,\lambda}(x))$. 
\item \label{item:moreau-approx} \emph{Approximation}: If $f$ is $G$-Lipschitz then $f(x) -  \frac{G^2}{2\lambda} \le f_\lambda(x) \le f(x)$ .
\item \label{item:moreau-subgrad} \emph{Subgradient}: $\grad f_{\lambda}(x) \in \del f( \prox_{f,\lambda}(x))$,
\item \label{item:moreau-three-point} \emph{Three point inequality}: for all $u\in \xset$:
\[
\inner{\grad f_{\lambda}(x)}{\prox_{f,\lambda}(x)-u} \le \halflam \norm{u-x}^2 - \halflam \norm{u-\prox_{f,\lambda}(x)}^2 - \halflam\norm{x - \prox_{f,\lambda}(x)}^2
\,.
\]

\end{enumerate}
\end{fact}
See~\cite[][Section 4.1]{hiriart1993convex_ii} as well as \cite[][Lemma 1]{carmon2021thinking} and \cite[][Lemma 1]{thekumparampil2020projection} for proofs and additional background and properties.

\section{Proofs from Section~\ref{sec:proj-eff}}\label{app:proj-eff}

\newcommand{\yt}{\widetilde{y}_{k-1}}

In this section, we give a proof of \Cref{thm:PE}. Before we give the technical details, we briefly comment on our algorithm and its analysis. \Cref{alg:proj_eff} is at its core an instantiation of Nesterov's accelerated gradient method applied to the Moreau envelope $f_{\lambda}(x) = \min_{y \in \ball_R(0)} \left\{ f(y) + \frac{\lambda}{2} \norm{y-x}^2 \right\}$. We compute stochastic gradient estimates of $f_{\lambda}$ via \Cref{alg:gradest}, and apply techniques from \cite{Zhu17,ZhuO17} to bound the accumulated error.

Based on the iterates $\{x_k, v_k \}$ of \Cref{alg:proj_eff}, we define
\[
E_k = \moreau(x_k) - \moreau(u), R_k = \frac{1}{2} \norm{v_k - u}^2, \text{ and } P_k = k(k+1) E_k + 12 \lambda R_k
\]
for any fixed $u \in \ball_R(0)$. We first prove that (conditioned on the iterates $x_{k-1}, v_{k-1}$) the potential $P_k$ cannot increase significantly in expectation. 

\begin{lemma}
\label{lem:PE-pot}
Consider an execution of \Cref{alg:proj_eff} with parameters given by \Cref{thm:PE}. Fix any $u \in  \ball_R(0)$. For any $k \geq 1$ we have  $y_{k-1} \in \ball_R(0)$ and
\[
\mathbb{E} \left[ P_k | x_{k-1}, v_{k-1} \right] \leq P_{k-1} + \eps k\,.
\]
\end{lemma}
\begin{proof}
We first remark that $x_{k-1} \in \xset \subseteq \ball_R(0)$ and $v_{k-1} \in  \ball_R(0)$ by construction. As a result, $y_{k-1} \in \ball_R(0)$ as well. Following \cite{Zhu17,ZhuO17}, we define the function
\[
\mathsf{Prog}(y;g) \defeq \min_{x \in \xset} \left\{ \frac{3\lambda}{2} \norm{x - y}^2 + \left\langle  g, x - y \right\rangle \right\}.
\]
We observe
\begin{align}
\mathsf{Prog}(y_{k-1};g_k) &=  \min_{x \in \xset} \left\{ \frac{3\lambda}{2} \norm{x - y_{k-1}}^2 + \left\langle  g_k, x - y_{k-1} \right\rangle \right\} \label{eqn:katyusha} \\
&\substack{(i)\\=} \frac{3\lambda}{2} \norm{x_{k} - y_{k-1}}^2 + \left\langle g_k, x_{k} - y_{k-1} \right\rangle \nonumber \\ 
&= \left(\halflam \norm{x_{k} - y_{k-1}}^2 + \left\langle \nabla \moreau(y_{k-1}), x_{k} - y_{k-1} \right\rangle\right) \nonumber \\
&\phantom{\quad} +  \lambda \norm{x_{k} - y_{k-1}}^2 + \left\langle g_k - \nabla \moreau(y_{k-1}), x_{k} - y_{k-1} \right\rangle \nonumber \\
&\substack{(ii) \\ \geq} \moreau(x_{k}) - \moreau(y_{k-1})+  \lambda \norm{x_{k} - y_{k-1}}^2 + \left\langle g_k - \nabla \moreau(y_{k-1}), x_{k} - y_{k-1} \right\rangle \nonumber \\
&\substack{(iii) \\ \geq} \moreau(x_{k}) - \moreau(y_{k-1}) - \frac{1}{4\lambda} \norm{g_k - \nabla \moreau(y_{k-1})}^2 \nonumber.
\end{align}
Here, we use $(i)$ the definition of $x_{k}$, $(ii)$ smoothness of $\moreau$ (\Cref{item:moreau-smooth} of \Cref{fact:moreau}), and $(iii)$ Young's inequality $\left\langle a, b \right\rangle + \frac{1}{2} \norm{b}^2 \ge  -\frac{1}{2} \norm{a}^2$ with $a=\frac{1}{2\lambda}(g_k - \grad f_\lambda(y_{k-1}))$ and $b=2\lambda(x_k-y_{k-1})$. Define the point
\[
\yt = \frac{k-1}{k+1} x_{k-1} + \frac{2}{k+1} v_{k}. 
\]
We observe that  
\[
y_{k-1} - \yt = \left( \frac{k-1}{k+1} x_{k-1} + \frac{2}{k+1} v_{k-1} \right) - \left( \frac{k-1}{k+1} x_{k-1} + \frac{2}{k+1} v_{k} \right) = \frac{2}{k+1} \left( v_{k-1} - v_{k} \right)\,.
\]
Consequently, we have
\begin{align}
\frac{k}{6 \lambda} \left\langle g_k, v_{k-1} - u \right\rangle &= \frac{k}{6 \lambda} \left\langle g_k, v_{k-1} - v_k \right\rangle +\frac{k}{6 \lambda} \left\langle g_k, v_k - u \right\rangle \nonumber \\ 
&\substack{(i) \\ \leq} \frac{k}{6 \lambda} \left\langle g_k, v_{k-1} - v_{k} \right\rangle + \frac{1}{2} \left(\norm{v_{k-1} - u}^2 - \norm{v_{k} - u}^2 - \norm{v_{k-1} - v_{k}}^2 \right) \nonumber \\
&\substack{(ii) \\ =} \frac{k(k+1)}{12 \lambda} \left\langle g_k, y_k - \yt \right\rangle - \frac{(k+1)^2}{8} \norm{y_k -  \yt}^2 + R_{k-1} - R_k \nonumber \\ 
&\substack{(iii) \\ \leq } \frac{k(k+1)}{12 \lambda} \left( \left\langle g_k, y_{k-1} - \yt \right\rangle - \frac{3  \lambda}{2} \norm{y_k -  \yt}^2 \right) + R_{k-1} - R_k \nonumber \\
&\substack{(iv) \\ \leq } - \frac{k(k+1)}{12 \lambda} \mathsf{Prog}(y_{k-1}; g_k) +  R_{k-1} - R_k \nonumber \\
&\substack{(v) \\ \leq} \frac{k(k+1)}{12 \lambda} \left(\moreau(y_{k-1}) -\moreau(x_{k}) + \frac{1}{4\lambda} \norm{g_k - \nabla \moreau(y_{k-1})}^2 \right) + R_{k-1} - R_k \label{eqn:coupling}.
\end{align}
Here we use $(i)$ the proximal three-point inequality (\Cref{item:moreau-three-point} of \Cref{fact:moreau}), $(ii)$ the definition of $\yt$, $(iii)$ $\frac{(k+1)^2}{8} \geq  \frac{3\lambda}{2} \cdot \frac{k(k+1)}{12 \lambda}$ and $\norm{y_{k-1} - \yt}^2 \geq 0$, $(iv)$ the definition of $\mathsf{Prog}$, and $(v)$ Equation~\eqref{eqn:katyusha}.
Thus, 
\begin{align*}
\frac{k}{6 \lambda} \left( \moreau(y_{k-1}) - \moreau(u)\right) &\leq \frac{k}{6 \lambda} \left\langle \nabla \moreau(y_{k-1}), y_{k-1} - u \right\rangle \\
&\leq \frac{k}{6 \lambda} \left\langle \nabla \moreau(y_{k-1}), y_{k-1} - v_{k-1} \right\rangle + \frac{k}{6 \lambda} \left\langle \nabla \moreau(y_{k-1}), v_{k-1} - u \right\rangle\\
&\substack{(i) \\ =} \frac{k(k-1)}{12 \lambda} \left\langle \nabla \moreau(y_{k-1}), x_{k-1} - y_{k-1} \right\rangle + \frac{k}{6 \lambda} \left\langle \nabla \moreau(y_{k-1}), v_{k-1} - u \right\rangle \\ 
&\leq \frac{k(k-1)}{12 \lambda} \left( \moreau(x_{k-1}) - \moreau(y_{k-1}) \right) + \frac{k}{6 \lambda} \left\langle \nabla \moreau(y_{k-1}), v_{k-1} - u \right\rangle \\
&\substack{(ii) \\ \leq} \frac{k(k-1)}{12 \lambda} \left( \moreau(x_{k-1}) - \moreau(y_{k-1}) \right)  + R_{k-1} - R_{k} \\
&\phantom{=} + \frac{k(k+1)}{12 \lambda} \left( \moreau(y_{k-1}) - \moreau(x_{k})+ \frac{1}{4\lambda} \norm{g_k - \nabla \moreau(y_{k-1})}^2 \right) \\
&\phantom{=} + \frac{k}{6 \lambda} \left\langle \nabla \moreau(y_{k-1}) - g_k, v_{k-1} - u \right\rangle,
\end{align*}
where we use $(i)$ $y_{k-1} - v_{k-1} = \frac{k-1}{2} (x_{k-1} - y_{k-1})$ and $(ii)$ Equation~\eqref{eqn:coupling}. Rearranging, we obtain
\begin{align}
\frac{1}{12 \lambda} \left( P_k - P_{k-1} \right) &= \frac{k(k+1)}{12 \lambda} E_k + R_k - \frac{k(k-1)}{12 \lambda} E_{k-1} - R_{k-1} \nonumber \\
&\leq \frac{k(k+1)}{48 \lambda^2} \norm{g_k - \nabla \moreau(y_{k-1})}^2 + \frac{k}{6 \lambda}  \left\langle \nabla \moreau(y_{k-1}) - g_k, v_{k-1} - u \right\rangle  \label{eq:potential1}
\end{align}
Applying \Cref{coro:gradest}, we observe
\[
\mathbb{E}\left[ \norm{g_k - \nabla \moreau(y_k)}^2 | x_{k-1}, v_{k-1} \right] \leq \sigma^2_{k} = \frac{2 \eps \lambda }{k+1}
\]
and 
\[
\mathbb{E} \left[ \left\langle \nabla \moreau(y_{k-1}) - g_k, v_{k-1} - u \right\rangle | x_{k-1}, v_{k-1} \right] \leq \norm{ \mathbb{E} \left[ g_k \right] -  \nabla \moreau(y_{k-1}) } \norm{v_{k-1} - u} \leq 2 R \delta_{k} = \frac{\eps}{4}
\]
by the Cauchy-Schwarz inequality, the constraint that $u, v_k \in \ball_R(0)$, and the choice of parameters $\sigma_k, \delta_k$. Taking expectations and applying these to \Cref{eq:potential1}, we obtain
\[
\frac{1}{12 \lambda} \left( \mathbb{E} \left[P_k | x_{k-1}, v_{k-1} \right] - P_{k-1} \right) \leq \frac{\eps k}{24 \lambda} + \frac{\eps k}{24 \lambda} =  \frac{\eps k}{12 \lambda}.
\]
Multiplying both sides by $12 \lambda$ yields the claim. 
\end{proof}

With Lemma~\ref{lem:PE-pot} in hand, we complete the proof of \Cref{thm:PE}.

\thmPE*

\begin{proof}[Proof of \Cref{thm:PE}]
Applying the law of total probability and inductively applying \Cref{lem:PE-pot}, we obtain
\[
\mathbb{E} \left[ P_T \right] \leq P_0 +  \eps \sum_{k=1}^T k =  P_0 + \frac{\eps}{2} T (T+1).
\]
We choose $u = \xopt$ and observe $P_T = T(T+1) E_T + 12 \lambda R_T \geq T(T+1) \left( \moreau(x_T) - \moreau(\xopt) \right)$ and $P_0 = 12 \lambda R_0 \leq 6 \lambda D^2$.  Plugging these in, we have
\[
\mathbb{E} \left[ \moreau(x_T) \right] - \moreau(\xopt)  \leq \frac{6 \lambda D^2}{T(T+1)} + \frac{\eps}{2}.
\]
As $f$ is $G$-Lipschitz, we apply \Cref{item:moreau-convex} of \Cref{fact:moreau} and our choices of $\lambda$ and $T$: this gives
\[
\mathbb{E} \left[ f(x_T) \right] - f(\xopt) \leq \frac{G^2}{2 \lambda} + \frac{6 \lambda D^2}{T(T+1)} + \frac{\eps}{2} \leq \frac{\eps}{4} + \frac{12 G^2 D^2}{\eps T^2} + \frac{\eps}{2} \leq \frac{\eps}{4} + \frac{12 \eps}{49} + \frac{\eps}{2} < \eps
\]
as desired.

To finish, we bound the number of oracle queries. The bound on the number of projection oracle calls is immediate since we only call it once per iteration of the algorithm. To bound the number of stochastic gradients needed, we apply Corollary~\ref{coro:gradest} together with the fact that $y_k \in \ball_R(0)$ at all times. Thus, we need 
\[
O\left( \sum_{k = 0}^{T-1} \frac{G^2}{\sigma_k^2} \log^2 \left(\frac{G}{\delta_k} \right) \right) = O \left(\sum_{k = 0}^{T-1} \frac{G^2 (k+1)}{\eps \lambda} \log^2 \left(\frac{GR}{\eps}\right) \right) = O\left( \frac{G^2 D^2}{\eps^2} \log^2 \left(\frac{GR}{\eps}\right) \right) 
\] 
subgradient computations as desired.
\end{proof}

%
\section{Proofs from \Cref{sec:mtm}}\label{app:mtm}
	
\subsection{Analysis of the stochastic accelerated proximal method}\label{app:mtm-sapm}

In this section we provide a complete analysis of the stochastic accelerated proximal method. We first prove \Cref{lem:sapm-step-bound}, which shows potential decrease in (conditional) expectation for the iterates of \Cref{alg:MS}. Then we give \Cref{lem:sapm-martingale} which provides an in-expectation bound on the potential when the algorithm terminates. In Lemma~\ref{lem:sapm-deterministic-Ak-bound} we give a deterministic error bound resulting from the growth of the $A_k$ sequence. We then combine  these ingredients to prove \Cref{prop:sapm}.

\newcommand{\hg}{\hat{g}}
\newcommand{\filt}[1][k]{\mathcal{F}_{#1}}
\newcommand{\epsbar}{\bar{\varepsilon}}

\paragraph{Notation.} 
Define the filtration 
\[\filt = \sigma(x_1, v_1, A_1, \zeta_1 \ldots, x_k, v_k, A_k, \zeta_k)\]
where $\zeta_i$ is the internal randomness in $\nextLambda(x_i, v_i, A_i)$. 
Throughout, we let
\[
\hx_k = \prox_{f,\lambda_k}(y_{k-1})
\]
denote the exact proximal mapping which iteration $x_k$ of the algorithm approximate. We note that $A_{k+1}, y_k, \hx_{k+1}, \in \filt$, i.e., they are deterministic when conditioned on $x_k, v_k, A_k, \zeta_k$.

For each iteration of Algorithm~\ref{alg:MS}, we obtain the following bound on  potential decrease.

\begin{lemma}\label{lem:sapm-step-bound}
	Let $f:\xset\to \R$ satisfy \Cref{ass:gest}. 
	If $\xset\subseteq \ball_R(x_0)$, we have
	\begin{flalign*}
		&\Ex*{
			A_{k+1} (f(x_{k+1}) - f(\xopt)) + \norm{v_{k+1}-\xopt}^2  | \filt
		}
		\\ & \hspace{1cm} 
		\le 
		A_{k} (f(x_{k}) - f(\xopt)) +  \norm{v_{k}-\xopt}^2 - \frac{1}{6} \lambda_{k+1} A_{k+1} \norm{\hx_{k+1}-y_k}^2 
		\\ &  \hspace{1cm} \hphantom{\le}+ \lambda_{k+1}a_{k+1}^2\vphi_{k+1} +  a_{k+1}^2 \sigma_{k+1}^2 + 2R a_{k+1}  \delta_{k+1}.
	\end{flalign*}
\end{lemma}
\begin{proof}
 We let
	\[
	\hg_{k}=\grad f_{\lambda_{k}}\left(y_{k-1}\right)=\lambda_{k}\left(y_{k-1}-\hx_{k}\right)
	\]
	and bound from both sides the quantity $a_{k+1}\inner{\hg_{k+1}}{v_{k}-\xopt}$.
	First, note that
	\[
	v_{k}-\xopt=\hx_{k+1}-\xopt+\frac{A_{k}}{a_{k+1}}\left(\hx_{k+1}-x_{k}\right)-\frac{A_{k+1}}{a_{k+1}}\left(\hx_{k+1}-y_{k}\right).
	\]
	Since $\hg_{k+1}\in\partial f\left(\hx_{k+1}\right)$ (see \Cref{item:moreau-subgrad} in \Cref{fact:moreau}),
	$f$ is convex and $\inner{\hg_{k+1}}{\hx_{k+1}-y_{k}}=-\lambda_{k+1}\norm{\hx_{k+1}-y_{k}}^{2}$,
	we have that
	\begin{flalign*}
		\inner{\hg_{k+1}}{v_k-\xopt} & =\inner{\hg_{k+1}}{\hx_{k+1}-\xopt}+\frac{A_{k}}{a_{k+1}}\inner{\hg_{k+1}}{\hx_{k+1}-x_{k}}-\frac{A_{k+1}}{a_{k+1}}\inner{\hg_{k+1}}{\hx_{k+1}-y_{k}}\\
		& \ge f\left(\hx_{k+1}\right)-f\left(\xopt\right)+\frac{A_{k}}{a_{k+1}}\prn*{f(\hx_{k+1})-f(x_{k})}-\frac{A_{k+1}}{a_{k+1}}\inner{\hg_{k+1}}{\hx_{k+1}-y_{k}}\\
		& =\frac{A_{k+1}}{a_{k+1}}\left(f\left(\hx_{k+1}\right)-f\left(\xopt\right)\right)-\frac{A_{k}}{a_{k+1}}\prn*{f(x_{k})-f(\xopt)}+\frac{\lambda_{k+1}A_{k+1}}{a_{k+1}}\norm{\hx_{k+1}-y_{k}}^{2}.
	\end{flalign*}
	Moreover, by definition of $x_{k+1}$ we have that
	\[
	\Ex*{f\left(x_{k+1}\right)|\filt}\le\Ex*{f\left(x_{k+1}\right)+\frac{\lambda_{k+1}}{2}\left\Vert x_{k+1}-y_{k}\right\Vert ^{2}|\filt}\le f\left(\hx_{k+1}\right)+\frac{\lambda_{k+1}}{2}\left\Vert \hx_{k+1}-y_{k}\right\Vert ^{2}+\vphi_{k+1}.
	\]
	Substituting back, we have
	\begin{flalign}
		a_{k+1}\inner{\hg_{k+1}}{v_k-\xopt}\ge & A_{k+1}\left(\Ex*{f(x_{k+1})|\filt}-f\left(\xopt\right)\right)-A_{k}\left(f\left(x_{k}\right)-f\left(\xopt\right)\right)\nonumber \\
		& +\frac{\lambda_{k+1}A_{k+1}}{2}\norm{\hx_{k+1}-y_{k}}^{2}-A_{k+1}\vphi_{k+1}.\label{eq:sapm-inner-lb}
	\end{flalign}
	To upper bound $a_{k+1}\inner{\hg_{k+1}}{v_k-\xopt}$, note that, since $\xopt\in\xset$,
	\[
	\norm*{v_{k+1}-\xopt} ^{2}\le\norm*{ v_k-\frac{1}{2}a_{k+1}g_{k+1}-\xopt}^{2}= \norm{v_{k}-\xopt} ^{2}-a_{k+1}\inner{g_{k+1}}{v_{k}-\xopt}+\frac{a_{k+1}^{2}}{4}\left\Vert g_{k+1}\right\Vert ^{2}.
	\]
	Our Moreau Envelope gradient estimtor (see \Cref{coro:gradest}) guarantees that
	\begin{flalign*}
		\Ex*{\inner{g_{k+1}}{v_{k}-\xopt}|\filt} &
		\ge
		\inner{\hg_{k+1}}{v_{k}-\xopt}-\left\Vert \Ex*{g_{k+1}|\filt}-\hg_{k+1}\right\Vert \left\Vert v_{k}-\xopt\right\Vert \\& \ge\inner{\hg_{k+1}}{v_{k}-\xopt}-2R\delta_{k+1},
	\end{flalign*}
	and moreover
	\begin{flalign*}
		\Ex*{\left\Vert g_{k+1}\right\Vert ^{2}|\filt} & =\left(1+\frac{1}{3}\right)\E\left\Vert \hg_{k+1}\right\Vert ^{2}+\left(1+3\right)\Ex*{\left\Vert g_{k+1}-\hg_{k+1}\right\Vert ^{2}|\filt}\\
		& \le\frac{4}{3}\left\Vert \hg_{k+1}\right\Vert ^{2}+4\sigma_{k+1}^{2}.
	\end{flalign*}
	Combining the last three displays and rearranging, we obtain
	\begin{flalign}
		a_{k+1}\inner{\hg_{k+1}}{v_{k}-\xopt}\le & \left\Vert v_{k}-\xopt\right\Vert ^{2}-\Ex*{\left\Vert v_{k+1}-\xopt\right\Vert ^{2}|\filt}+\frac{\lambda_{k+1}^{2}a_{k+1}^{2}}{3}\norm{\hx_{k+1}-y_{k}}^{2} \nonumber \\
		& +a_{k+1}^{2}\sigma_{k+1}^{2}+2Ra_{k+1}\delta_{k+1}\label{eq:sapm-inner-ub}
	\end{flalign}
	Combining \eqref{eq:sapm-inner-lb} and \eqref{eq:sapm-inner-ub}
	and simplifying using $A_{k+1}=\lambda_{k+1}a_{k+1}^{2}$, we obtain
	the claimed bound.
\end{proof}

Combining Lemma~\ref{lem:sapm-step-bound} with the optional stopping theorem, one obtains the following bound on the potential at the final iteration $K$ of the algorithm.

\begin{lemma}\label{lem:sapm-martingale}
	Let $K \le \Kmax$ be the iteration in which \Cref{alg:MS} returns and let
	 \[
	 \epsbar \ge \max_{k\le \Kmax} \crl*{\lambda_k a_k \vphi_k +  a_k \sigma_k^2 + 2R \delta_{k}}
	 \]
	 with probability 1. Then, under the assumptions of \Cref{lem:sapm-step-bound}, we have
	\begin{equation*}
		\Ex*{ A_{K} ( f(x_K) - f(\xopt) - \epsbar) + \frac{1}{6} \sum_{i \le K} \lambda_i A_i \norm{\hx_i - y_{i-1}}^2 } \le A_0 ( f(x_0) - f(\xopt) ) + R^2. 
	\end{equation*}
\end{lemma}
\begin{proof}

Define $M_k = A_k ( f(x_k) - f(\xopt) - \epsbar) + \frac{1}{6} \sum_{i \le k} \lambda_i A_i \norm{\hx_i - y_{i-1}}^2 + \norm{v_k - \xopt}^2$ for all $k\in[K]$.  We argue that it is a supermartingale adapted to filtration $\filt$. Clearly, $\E[|M_k|]<\infty$ for each $k$ due to boundedness of $f,K,\lambda_i$ and $A_i$. It therefore suffices to show that $\E[M_{k+1}|F_{k}]\le M_{k}$ for all $k+1\in  [K]$. By \Cref{lem:sapm-step-bound} we have
\begin{align*}
\E\left[M_{k+1}|F_{k}\right] 
& \le A_{k} (f(x_{k}) - f(\xopt)) +  \norm{v_{k}-\xopt}^2 - \frac{1}{6} \lambda_{k+1} A_{k+1} \norm{\hx_{k+1}-y_k}^2 
		\\ &  \ \hphantom{\le}+ \lambda_{k+1}a_{k+1}^2\vphi_{k+1} +  a_{k+1}^2 \sigma_{k+1}^2 + 2R a_{k+1}  \delta_{k+1} - A_{k+1}\epsbar+\frac{1}{6}\sum_{i \le k+1} \lambda_i A_i \norm{\hx_{i+1} - y_{i}}^2
		\\& \le A_{k} (f(x_{k}) - f(\xopt)-\epsbar ) +  \norm{v_{k}-\xopt}^2 +\frac{1}{6}\sum_{i \le k} \lambda_i A_i \norm{\hx_{i+1} - y_{i}}^2 = M_k,
\end{align*}
where the second inequality used the definition of $\epsbar$ and $A_{k+1}=A_k+a_{k+1}$ for the second inequality. This completes the proof that $M_k$ being a supermartingale adapted to filtration $\filt$.

Now note $K$ is a stopping time adapted to $\filt$ as it only depends on $A_{k+1}$. Also, $K$ as a random variable is finitely bounded by $\Kmax$ with probability $1$. Thus,  by optional stopping theorem for supermartingale~\cite{grimmett2020probability}, we have 
\[
\E M_{K}\le M_0 = A_0 ( f(x_0) - f(\xopt) - \epsbar) + \norm{v_0 - \xopt}^2\le  A_0 ( f(x_0) - f(\xopt) ) + R^2.
\]
\end{proof}

Further, following a similar argument to~\citet{carmon2020acceleration,carmon2021thinking}, we obtain a deterministic growth bound on the coefficients $A_k$.

\begin{lemma}\label{lem:sapm-deterministic-Ak-bound}
	Fix $k > 0$ and let 
	\begin{equation*}
		T_{\lambda} = \sum_{i\le k} \indic{\lambda_i <  2\lambda_{\min}}~~\mbox{and}~~T_r = \sum_{i \le k} \indic{\norm{\hx_i - y_{i-1}} \ge 3r/4}
	\end{equation*}
	count the number of times $\lambda_i < 2\lambda_{\min}$ and $\norm{\hx_i - y_{i-1}} \ge 3r/4$, respectively. Then, the following holds with probability 1,
	\begin{equation*}
		\frac{1}{A_k}\prn*{ 9 R^2 - \frac{1}{6} \sum_{i \le k}\lambda_i A_i \norm{\hx_i - y_{i-1}}^2 } \le O\prn*{\min\crl*{ \frac{\lambda_{\min}R^2}{T_\lambda^2}, \frac{R^2}{A_0}\exp\prn*{-\Omega(1)\frac{r^{2/3}}{R^{2/3}} T_r }}}.
	\end{equation*}
\end{lemma}

\begin{proof}
	When $9 R^2 - \frac{1}{6} \sum_{i \le k}\lambda_i A_i \norm{\hx_i - y_{i-1}}^2 \le 0$, the inequality holds true trivially. Thus, we only consider the case when $\frac{1}{6} \sum_{i \le k}\lambda_i A_i \norm{\hx_i - y_{i-1}}^2 \le 9R^2$. Consider the following iterate index subsets
	\[
	\mathcal{I}_\lambda\defeq \{i\le k:\lambda_i<2\lambda_{\min}\}
	\] 
	and, for  $t\le k$
	 \[
	 \mathcal{I}_{r,t}\defeq \{i\le t:\norm{\hat{x}_i-y_{i-1}}\ge 3r/4\}.
	 \] 
	 
	 We first show that
	\begin{equation}
	\frac{1}{A_k}\prn*{ 9 R^2 - \frac{1}{6} \sum_{i \le k}\lambda_i A_i \norm{\hx_i - y_{i-1}}^2 } \le O\prn*{\frac{R^2}{A_0}\exp\prn*{-\Omega(1)\frac{r^{2/3}}{R^{2/3}} T_r }}.\label{eq:growth-condition-1}
\end{equation}
	To see this, observe that for any $t\le k$ by definition of $\mathcal{I}_{r,t}$,
	\[
	\frac{1}{6} \sum_{i \in \mathcal{I}_{r,t}}\lambda_i A_i \cdot \left(\frac{9}{16}r^2\right)\le \frac{1}{6} \sum_{i \le t}\lambda_i A_i \norm{\hx_i - y_{i-1}}^2 \le 9R^2,
	\]
	which by rearranging terms implies
	\begin{align}
	\sum_{i \in \mathcal{I}_{r,t}}\lambda_i A_i\le \frac{96R^2}{r^2}.\label{eq:potential-bound-rearranged}
	\end{align}
	Note the reverse H\"older's inequality with $p=2/3$ states that for any $u,v\in \R^d_{>0}$,
	\[
	\langle u,v \rangle\ge \left(\sum_{i\in[d]}u_i^{2/3}\right)^{3/2}\cdot \left(\sum_{i\in[d]}v_i^{-2}\right)^{-1/2}.
	\] 
	We have 
\begin{flalign*}
	\sqrt{A_t} \overge{(i)} \frac{1}{2} \sum_{i \in \mathcal{I}_{r,t}} \frac{1}{\sqrt{\lambda_i}} 
	&\overge{(ii)} \frac{1}{2} \left( \sum_{i \in \mathcal{I}_{r,t}} \left(\sqrt{A_i}\right)^{2/3} \right)^{3/2} \cdot \left( \sum_{i \in\mathcal{I}_{r,t}} \left( \frac{1}{\sqrt{A_i \lambda_i}} \right)^{-2} \right)^{-1/2} \\& \overge{(iii)} 
	\frac{r}{8\sqrt{6}R} \cdot \left( \sum_{i \in \mathcal{I}_{r,t}} \left(\sqrt{A_i}\right)^{2/3} \right)^{3/2} ,
\end{flalign*}
where we used $(i)$ Lemma 23 of~\cite{carmon2020acceleration} and $\mathcal{I}_{r,t}\subseteq [t]$, $(ii)$ the reverse H\"older's inequality with $u_i=\sqrt{A_i}$, and $v_i=1/\sqrt{A_i\lambda_i}$, and $(iii)$ the bound \eqref{eq:potential-bound-rearranged}. Rearranging, we have
\begin{equation}
\label{eq:At_bound_1}
A_t^{1/3} \geq  \frac{r^{2/3}}{4\sqrt[3]{6} R^{2/3}}  \left( \sum_{i \in \mathcal{I}_{r,t}} A_i^{1/3}\right), ~~\text{for all}~~t\le k,
\end{equation}
which by applying Lemma 32 of~\cite{carmon2020acceleration} and noting that $T_r = \abs*{\mc{I}_{r,k}}$ gives
\[
A_k^{1/3} \geq  \exp\left(\frac{r^{2/3}}{4\sqrt[3]{6} R^{2/3}}T_r\right)A_0^{1/3} ,
\]
and thus 
\[
	\frac{1}{A_k}\prn*{ 9 R^2 - \frac{1}{6} \sum_{i \le k}\lambda_i A_i \norm{\hx_i - y_{i-1}}^2 } \le \frac{9R^2}{A_k} \le  O\prn*{\frac{R^2}{A_0}\exp\prn*{-\Omega(1)\frac{r^{2/3}}{R^{2/3}} T_r }}.
	\]
	
Next, we show that
\begin{equation}\frac{1}{A_k}\prn*{ 9 R^2 - \frac{1}{6} \sum_{i \le k}\lambda_i A_i \norm{\hx_i - y_{i-1}}^2 } \le O\prn*{ \frac{\lambda_{\min}R^2}{T_\lambda^2}}.\label{eq:growth-condition-2}
\end{equation}
Using Lemma 23 of~\cite{carmon2020acceleration} again, along with  $\mathcal{I}_\lambda\subseteq[k]$ and $\abs*{\mathcal{I}_\lambda} = T_\lambda$, we have 
\[\sqrt{A_k} \ge  \frac{1}{2} \sum_{i \in \mathcal{I}_{\lambda}} \frac{1}{\sqrt{\lambda_i}} \ge\frac{T_\lambda}{2\sqrt{2\lambda_{\min}}}.\]
Rearranging the terms, we see that $1/A_k \le O(\lambda_{\min}/T_\lambda^2)$ as desired.

Combining Equations~\eqref{eq:growth-condition-1} and~\eqref{eq:growth-condition-2} we obtain the claimed bound.
\end{proof}

Putting these pieces together gives \Cref{prop:sapm}, which we prove below.

\propSAPM*
\begin{proof}
	First, let us prove correctness of the algorithm. The settings of $\vphi_k,\delta_k$ and $\sigma_k$ in the proposition guarantee that
	\begin{equation*}
		\max_{k\le \Kmax} \crl*{\lambda_k a_k \vphi_k +  a_k \sigma_k^2 + 2R \delta_{k}} = \frac{\epsilon}{20}.
	\end{equation*}
	Therefore, \Cref{lem:sapm-martingale} with $\epsbar = \epsilon/20 \le R^2 / (2.2\Amax)$ yields
	\begin{equation*}
		\Ex*{ A_{K} ( f(x_K) - f(\xopt)) + \frac{1}{6} \sum_{i \le K} \lambda_i A_i \norm{\hx_i - y_{i-1}}^2 } \le R^2 +  \epsbar \cdot \E A_K + A_0 ( f(x_0) - f(\xopt) ).
	\end{equation*}
	
	Note that $A_{K-1}\le \Amax$ by definition. Therefore, $\lambda_{\min} \ge \frac{1}{\Amax}$ implies that
	\begin{equation*}
		a_K = \frac{1}{2}\sqrt{\frac{1}{\lambda_K^2}+\frac{4A_{K-1}}{\lambda_K}} \le \frac{\sqrt{5}}{2}\Amax,
	\end{equation*}
	and therefore $A_K \le  2.2\Amax \le R^2/\epsbar$ with probability 1. Moreover the choice of $A_0=R/G$ and the fact that $f$ is $G$ Lipschitz imply that $A_0 ( f(x_0) - f(\xopt) ) \le R^2$. \notarxiv{Therefore,
	$\Ex*{ A_{K} ( f(x_K) - f(\xopt)) + \frac{1}{6} \sum_{i \le K} \lambda_i A_i \norm{\hx_i - y_{i-1}}^2 }$ is at most $3R^2$.}
	\arxiv{
	Therefore,
	\[\Ex*{ A_{K} ( f(x_K) - f(\xopt)) + \frac{1}{6} \sum_{i \le K} \lambda_i A_i \norm{\hx_i - y_{i-1}}^2 }\le 3R^2.\]
	}
	Since the term in the expectation is non-negative, we conclude that  with probability  at least $2/3$ it is bounded by $9R^2$, which implies 
	\begin{equation*}
		f(x_K) - f(\xopt) \le \frac{1}{A_K} \prn*{ 9 R^2 - \frac{1}{6} \sum_{i \le K}\lambda_i A_i \norm{\hx_i - y_{i-1}}^2 }.
	\end{equation*}
	If $A_K \ge \Amax=9R^2 / \epsilon$ we are done. Otherwise, $K=\Kmax$ and by the assumption on $\nextLambda$ we have $T_\lambda + T_r \ge  \Kmax$  for $T_\lambda$ and $T_r$ defined in \Cref{lem:sapm-deterministic-Ak-bound}. Therefore, either $T_r \ge \Kmax/2$ or $T_\lambda \ge \Kmax/2$, and in either case taking $\Kmax = O\prn*{ \prn*{\frac{R}{r}}^{2/3} \log\prn*{\frac{GR}{\epsilon}} + \sqrt{\frac{\lambda_{\min} R^2}{\epsilon} }}$ and applying \Cref{lem:sapm-deterministic-Ak-bound} yields $f(x_K) - f(\xopt) \le \epsilon$ and establishing correctness.
	
	Next, let us prove the stated complexity bound. We note each step of computing $x_{k}$ in Line~\ref{line:sapm-x} requires $O(G^2/\lambda_{k}\varphi_k)$ queries via~\Cref{prop:epoch-sgd} and the definition~\eqref{eq:app-prox-def} of the approximate proximal mapping. Moreover, by \Cref{coro:gradest} computing $g_{k}$ in Line~\ref{line:sapm-v} requires \[O\prn*{\log \prn*{\frac{G}{\min\{\delta_k,\sigma_k\}}}+\frac{G^2}{\sigma_k^2}\log^2\prn*{\frac{G}{\min\{\delta_k,\sigma_k\}}}}\] queries in expectation. Summing over $k\in [K]$ and substituting $\vphi_k,\delta_k,\sigma_k$, we obtain 
	\begin{align*}
	\E\Nquery & = \sum_{k\in[K]}	O\prn*{\frac{G^2}{\lambda_{k}\varphi_k}} + \sum_{k\in[K]}O\prn*{\log \prn*{\frac{G}{\min\{\delta_k,\sigma_k\}}}+\frac{G^2}{\sigma_k^2}\log^2\prn*{\frac{G}{\min\{\delta_k,\sigma_k\}}}}\\
	& = \sum_{k\in[K]}O\prn*{\log\prn*{\frac{GR}{\eps}} + \frac{a_k G^2}{\eps}\log^2\prn*{\frac{GR}{\eps}}} =O\prn*{\log\prn*{\frac{GR}{\eps}}\cdot K + \frac{A_K G^2}{\eps}\log^2\prn*{\frac{GR}{\eps}}}\\
	& = O\prn*{\Kmax\log\prn*{\frac{GR}{\eps}} + \frac{G^2R^2}{\eps^2}\log^2\prn*{\frac{GR}{\eps}}},
	\end{align*}
where we have used $A_K = O(\Amax) = O(R^2/\eps)$ once more.
\end{proof}

\subsection{Minimizing the maximum of $N$ functions}\label{app:mtm-main}

\arxiv{
In this section, we first revisit the problem setup of minimizing the maximum of $N$ functions and reintroduce key notation. Then we provide the procedure of estimating the gradient of the softmax using rejection sampling in \Cref{alg:fsm-grad} and prove its guarantees in \Cref{lem:fsm-grad}. Next, we bound the query complexity of \cref{line:sapm-x,line:sapm-v} 
of \Cref{alg:MS} in \Cref{lem:BROO-complexity-y,lem:BROO-complexity-g} respectively. Citing~\cite{carmon2021thinking}, we provide a bisection procedure in Algorithm~\ref{alg:bisecx} and state its guarantee in Lemma~\ref{prop:bisecx}. To run this bisection procedure we use the Ball Regularization Optimization Oracle (\BOO) implementation of~\cite{carmon2021thinking}; see \Cref{def:broo} and \Cref{lem:sgd-broo}. Combining these components with the previous developments in \Cref{sec:mtm-sapm}, we prove \Cref{thm:min-the-max}.
}
\notarxiv{
In this section, we first revisit the problem setup of minimizing the maximum of $N$ functions and reintroduce key notation. Then we provide the procedure of estimating the gradient of the softmax using rejection sampling in \Cref{alg:fsm-grad} and prove its guarantees in \Cref{lem:fsm-grad}. Next, we bound the query complexity of \cref{line:sapm-x,line:sapm-v} 
of \Cref{alg:MS} in \Cref{lem:BROO-complexity-y,lem:BROO-complexity-g} respectively. Citing~\cite{carmon2021thinking}, we provide a bisection procedure in Algorithm~\ref{alg:bisecx} and state its guarantee in Lemma~\ref{prop:bisecx}. For this procedure \sidford{maybe clarify what ``this procedure'' refers to? e.g. ``to implement this bisection procedure''} we use the Ball Regularization Optimization Oracle (\BOO) implementation of~\cite{carmon2021thinking}; see \Cref{def:broo} and \Cref{lem:sgd-broo}. Combining these components with the developments of the previous subsection, we prove \Cref{thm:min-the-max}.\sidford{maybe include the sub sub section refs? don't have a strong view.}
}

\paragraph{Notation.} 
Consider the problem of approximately minimizing the maximum of $N$ convex functions: given $\fhat[i]$ such that for every $i\in [N]$ the function $\fhat[i]: \R^d \to \R$ is convex, $G$-Lipschitz, with a subgradient oracle $\grad \fhat[i]$ and a target accuracy $\epsilon$ we wish to
\begin{equation}\label{eq:problem}
\mbox{find a point $x$ such that }~
\fmax(x) - \inf_{\xopt \in \R^d} \fmax(\xopt) \le \epsilon
	~~\mbox{where}~~
	\fmax(x) \defeq \max_{i\in[N]} \fhat[i](x) ~.
\end{equation}

A common approach to solving this problem is to consider the following ``softmax'' approximation of $\fmax$,
\begin{equation}\label{eq:softmax}
	\fsm(x) \defeq 
	\epsilon' \log\prn*{\sum_{i\in[N]} e^{\fhat[i](x) / \epsilon'}},~~\mbox{where }
	\epsilon' = \frac{\epsilon}{2\log N}.
\end{equation}
It is straightforward to show that $0\le \fsm(x)-\fmax(x) \le \frac{\epsilon}{2}$
for all $x\in\R^d$, and that the subgradients of $\fsm$ are of the form
\begin{equation}\label{eq:softmax-prob}
	\grad \fsm(x) = \sum_{i\in [N]} p_i(x) \grad \fhat[i](x)~~\text{where}~~p_i(x) = \frac{e^{\fhat[i](x)/\epsilon'}}{\sum_{j\in [N]} e^{\fhat[j](x)/\epsilon'}}
\end{equation}
for $\grad \fhat[i](x)\in \del \fhat[i](x)$ for all $i\in[N]$. 
The small radius
\begin{equation*}
	\reps \defeq \frac{\epsilon'}{G} = \frac{\eps}{2G\log N}
\end{equation*}
plays a key role in our analysis, since---as we now discuss in detail---this is a domain size where we can efficiently minimize $\fsm$ using stochastic gradient methods.

\subsubsection{Gradient estimation via rejection sampling}

\newcommand{\accprob}{q_{\mathrm{accept}}}

We first construct the gradient estimator of $\fsm(x)$ using rejection sampling. The high-level idea of the technique is as follows. Given a ball $\ball_{\reps}(\bx)$ where $\reps = \eps'/G$, Lipschitz continuity of $\fhat[i]$ implies 
\begin{equation}\label{eq:ball-approx-bound}
	\frac{\abs*{\fhat[i](x) - \fhat[i](\bx)}}{\eps'} \le \frac{G\reps}{\eps'} = 1.
\end{equation}
 As a result, we can perform a full data pass~\emph{once} to compute $p(\bx)$, and use it to sample from $p(x)$ at nearby points  $x\in\ball_{\reps}(\bx)$ via rejection sampling. In particular, we draw $i$ from $p(\bx)$ and accept it with probability  $\accprob=\exp(\fhat[i](x)/\eps'-\fhat[i](\bx)/\eps'-1)$, and otherwise repeat the process. The the bound~\eqref{eq:ball-approx-bound} guarantees that $\accprob < 1$ (so it is indeed a probably), and therefore the output $i$ has distribution $p$. The bound~\eqref{eq:ball-approx-bound} also guarantees that $\accprob = \Omega(1)$ and consequently that the query complexity of the procedure is $O(1)$. We sate the procedure formally in \Cref{alg:fsm-grad} and give its guarantees in~\Cref{lem:fsm-grad}.
 
\begin{algorithm}
	\DontPrintSemicolon
	\KwInput{Functions $\fhat[i]$, pre-computed  $\fhat[i](\bx)$ and $\bar{p}_i = p_i(\bx)$ for $i\in[N]$, query point $x\in\ball_{\reps}(\bx)$.}
	\KwOutput{An unbiased estimator for $\grad \fsm$ with norm at most $G$.}
	\SetKwFor{Loop}{Loop}{}{}
	\Loop{}{
	Sample $i$ from $\bar{p}$\;
	Let $\accprob=\exp(\fhat[i](x)/\eps'-\fhat[i](\bx)/\eps'-1)$\;
	Draw $A\sim \bernoulli(\accprob)$\;
	\lIf{$A=1$}{\Return $\nabla \fhat[i](x)$}
	}
	\caption{$\gradestfsm(\{\fhat[i]\},\{\bar{p}_i\},\bx,x)$}	\label{alg:fsm-grad}
\end{algorithm}

\begin{lemma}[Rejection sampling]\label{lem:fsm-grad}
Given $G$-Lipschitz functions $\fhat[i]$ and $\bar{p} = p(\bx)$, $\forall i\in[N]$, the procedure $\gradestfsm$ with input $x\in\ball_{\reps}(\bx)$ returns a vector $\gest[\fsm](x)$ such that $\E\gest[\fsm](x)\in \del \fsm(x)$ and $\norm{\gest[\fsm](x)}\le G$. The procedure has complexity $\E \NqueryCustom{\fhat[i]}=O(1)$ and $\NqueryCustom{\del \fhat[i]}=1$.
\end{lemma}
\begin{proof}
We first prove correctness. Note that $G$-Lipschitz continuity of the $\fhat[i]$'s along with $\norm{x-\bx}\le \reps = \eps'/G$ guarantees that, $\fhat[i](x)\eps'-\fhat[i](\bx)/\eps'\le1$ and therefore $\accprob \le 1$ is a valid probability of every value of $i$. Therefore, the probability of sample and accepting $i$ is proportional to 
\[
\bar{p}_i\cdot \exp\prn*{\frac{\fhat[i](x)}{\eps'}-\frac{\fhat[i](\bx)}{\eps'}}\propto \exp\left(\frac{\fhat[i](x)}{\eps'}\right)\propto p_i(x),
\]
which proves the correctness of the sampling distribution for the output $i$ and, via~\cref{eq:softmax-prob}, the unbiasedness of the gradient estimator. The norm bound on the output of the procedure is immediate from Lipschitzness of  $\fhat[i]$.

Next, we prove the complexity bound. Clearly, the algorithm only queries a single subgradient at termination. To bound the number of function value queries, note that Lipschitz continuity and the ball radius imply $\fhat[i](x)\eps'-\fhat[i](\bx)/\eps' \ge -1$, and therefore the probability of acceptance is at least $e^{-2}$. Consequently, the expected number of iterations before accepting a sample is at most $e^2=O(1)$.
\end{proof}

\subsubsection{Estimating the proximal mapping and Moreau envelope gradient}

Using gradient estimator for $\gest[\fsm]$ developed above, we can implement \cref{line:sapm-x,line:sapm-v} in \Cref{alg:MS}, provided that the true proximal bound $\hx=\prox_{\lambda,\fsm}(y)$ satisfies $\norm{\hx-y}\le r=\reps$. We begin the implementation of the approximate proximal step in \cref{line:sapm-x}, which we obtain by directly applying $\epochSGD$. The following is an immediate consequence of \Cref{lem:fsm-grad} and \Cref{prop:epoch-sgd}.

\begin{lemma}\label{lem:BROO-complexity-y}
	Let $\fhat[i]$ be convex and $G$-Lipschitz for all $i\in[N]$, let $\epsilon, \varphi > 0$ and $\reps =  \eps/(2\log G N)$. 
	For any $\bx\in\R^d$ and $\lambda > 0$, if $\prox_{\fsm,\lambda}(\bx)\in\ball_{\reps}(\bx)$ then $\epochSGD(\gest[\fsm], \frac{\lambda}{2}\norm{\cdot-\bx}, \lambda, \xset \cap \ball_{\reps}(\bx), \ceil{16G^2/(\lambda \vphi)})$ (with $\gest[\fsm]$ implemented with \Cref{alg:fsm-grad}) outputs a valid point $\appProx_{\fsm,\lambda}^{\vphi}(\bx)$, and has complexity
\begin{equation*}%
		\E\NqueryCustom{\fhat[i]} = O\left(N + \frac{G^2}{\lambda\vphi}\right)~~\text{and}~~\NqueryCustom{\del\fhat[i]} = O\left(\frac{G^2}{\lambda\vphi}\right).
\end{equation*}
\end{lemma}

Similarly combining Lemma~\ref{lem:epoch-SGD} with Corollary~\ref{coro:gradest}, one can also obtain the following expected oracle complexity guarantee for estimating the Moreau envelope gradient .

\begin{lemma}\label{lem:BROO-complexity-g}
	Let $\fhat[i]$ be convex and $G$-Lipschitz for all $i\in[N]$, let $\sigma,\epsilon, \delta > 0$ and $\reps =  \eps/(2\log N\cdot G)$. 
	For any $\bx\in\R^d$ and $\lambda >0$,  if $\prox_{\fsm,\lambda}(\bx)\in\ball_{\reps}(\bx)$ then  $\hat{g}=\gradest(\gest[\fsm],\lambda,\bx,\delta,\sigma^2,\xset\cap\ball_{\reps}(\bx))$  (with $\gest[\fsm]$ implemented with \Cref{alg:fsm-grad}) is an estimator of the Moreau envelope gradient $\grad {\fsm}_{,\lambda}(\bx)$ with bias at most $\delta$ and expected square error at most $\sigma^2$. Its complexity is
\begin{equation*}%
\begin{aligned}
		\E\NqueryCustom{\fhat[i]}  & = O\left(N + \frac{G^2}{\sigma^2}\log^2\left(\frac{G}{\min\{\delta,\sigma\}}\right)+\log\left(\frac{G}{\min\{\delta,\sigma\}}\right)\right)\\
		\E\NqueryCustom{\del\fhat[i]}  & = O\left(\frac{G^2}{\sigma^2}\log^2\left(\frac{G}{\min\{\delta,\sigma\}}\right)+\log\left(\frac{G}{\min\{\delta,\sigma\}}\right)\right).
\end{aligned}
\end{equation*}
\end{lemma}

\subsubsection{Implementing $\nextLambda$ via bisection}

\newcommand{\bisectionTarget}{\Delta}
\newcommand{\idealBisectionTarget}{\widehat{\Delta}}
\newcommand{\brooAcc}{\rho}

The third and final component in our algorithm is an implementation of the subroutine $\nextLambda$ in \cref{line:next-lambda} of \Cref{alg:MS} that guarantees the following things on $\lambda_{k+1}$ and $\hx_{k+1}=\prox_{\fsm,\lambda}(y_k)$: $(i)$ that $\norm{\hx_{k+1}-y_k}\le r$ and $(ii)$ either $\norm{\hx_{k+1}-y_k} \ge 3r/4$ or $\lambda < 2\lambda_{\min}$; we later set $r=\reps$ and  $\lambda_{\min}= \Otil{\epsilon / (\reps^{4/3} R^{2/3})}$ but for the development of the bisection procedure we keep them general. Our implementation of $\nextLambda$ is identical to the one in \cite{carmon2021thinking}, and we reproduce it here for completeness.

We start by introducing the notion of a \emph{Ball Regularization Optimization Oracle} (\BOO).
\begin{definition}[{\cite[][Definition 1]{carmon2021thinking}}]\label{def:broo}
	We say that a mapping $\oracles[\lambda,\brooAcc]{\cdot}:\xset \to \xset$ is a Ball Regularized Optimization Oracle of radius $r$ ($r$-\BOO) for $f$, if for every query point $\bx$, regularization parameter $\lambda$ and desired accuracy $\brooAcc$, it return $\tilde{x} = \oracles[\lambda,\brooAcc]{\bx}$ satisfying
\begin{equation}\label{eq:broo-req}
		f(\tilde{x})+\frac{\lambda}{2}\|\tilde{x}-\bx\|^2\le \min_{x \in \ball_{r}(\bx)\cap \xset} \crl*{
		f(x) + \frac{\lambda}{2} \norm{x-\bx}^2}+\frac{\lambda}{2}\brooAcc^2.
\end{equation}
\end{definition}

While a \BOO is quite similar to the approximate proximal mapping  $\appProx_{\lambda}^\vphi$, there are two important differences. First, in the \BOO definition we constrain the minimization to $\ball_r(\bx)$ where the approximate proximal mapping is defined for the all domain---this allows us to efficiently compute a \BOO via stochastic methods even for values of $\lambda$ where the true (unconstrained) proximal point is far from $\bx$. Second, we require the sub-optimality guarantee to hold deterministically (a requirement that we will satisfy with high probability), as opposed the requirement~\eqref{eq:app-prox-def} of an expected suboptimality bound. In addition, note that the accuracy parameter $\vphi$ and $\rho$ are related via $\vphi = \lambda\rho^2/2$ and that $\rho$ has units of distance. Strong convexity of the \BOO optimization objective then implies that $\norm{\oracles[\lambda,\brooAcc]{x} - \prox_{f,\lambda}(x)}\le \rho$ whenever $\prox_{f,\lambda}(x)\in\ball_r(x)$. 

We have the following high-probability complexity guarantee for implementing a \BOO.

\newcommand{\pf}{p_{\mathrm{f}}} %

\begin{lemma}[{\cite[][Corollary 1]{carmon2021thinking}}]\label{lem:sgd-broo}
	Let $\fhat[i]$ be convex and $G$-Lipschitz for all $i\in[N]$, 
 	let $\pf\in(0,1)$, $\epsilon, \rho > 0$ and $\reps =  \eps/(2\log N\cdot G)$. 
	For any $\bx\in\R^d$ and $\lambda \le O(G / \reps)$, with probability at least $1-\pf$, $\textsc{Epoch-SGD-Proj}$~\cite[][Algorithm 2]{carmon2021thinking} that  outputs a valid $\reps$-\BOO response for $\fsm$ to query $\bx$ with regularization $\lambda$ and accuracy $\brooAcc$, and has  complexity
	\begin{equation}\label{eq:sgd-complexity-bound}
		\NqueryCustom{\fhat[i]} = O\left(N + \frac{G^2}{\lambda^2\brooAcc^2}\log\left(\frac{\log(G/(\lambda\brooAcc))}{\pf}\right)\right)~\mbox{and}~\NqueryCustom{\del\fhat[i]} = O\left( \frac{G^2}{\lambda^2\brooAcc^2}\log\left(\frac{\log(G/(\lambda\brooAcc))}{\pf}\right)\right)
	\end{equation}
\end{lemma}

Given a \BOO implementation \Cref{alg:bisecx} outputs values of $\lambda$
meeting the requirements of \Cref{prop:sapm}. The algorithm and the formal guarantee below are reproduced from \cite{carmon2021thinking} for completeness, and we refer the reader to Appendix B.3 of that paper for additional description and discussion. 
\begin{algorithm}
	\DontPrintSemicolon
	\caption{$\linesearch(x,v,A)$\label{alg:bisecx}}	
	\KwInput{Points $x,v\in\xset$, scalar $A\ge 0$.}
	\KwParameters{ \BOO $\oracles{\cdot}$ (see \Cref{def:broo}), bisection bounds $\lambda_{\min},\lambda_{\max}$, Lipschitz bound $G$, distance bounds $R$ and $r$.}
	For all $\lambda'$, let $y_{\lambda'} \defeq 
	\alpha_{2A\lambda'} \cdot x + (1-\alpha_{2A\lambda'}) \cdot v$, where
	$\alpha_\tau \defeq \frac{\tau}{1+\tau+\sqrt{1 + 2\tau}}$\;
	Define $\bisectionTarget(\lambda) \defeq \norm{\oracles[\lambda, \frac{r}{17}]{y_\lambda} - y_\lambda}$ 
	\Comment*{approximation of ball optimizer to $y_\lambda$}
	Let $\lambda = \lambda_{\max}$ \;
	\lWhile{$\lambda \geq \lambda_{\min}$ \textup{\textbf{and}}  $\bisectionTarget(\lambda) \leq \frac{13r}{16}$  \label{line:while1start}}{
		$\lambda \gets \lambda/2$
		\Comment*[f]{terminates in $O(\log\frac{\lambda_{\max}}{\lambda_{\min}})$ steps}
	}\label{line:while1end}

	\lIf{$\lambda \leq \lambda_{\min}$}{\Return $2\lambda$ \label{line:ls_lower_boundary}
	\Comment*[f]{happens only if ball optimizer is $O(\epsilon)$-optimal}
}
Let $\lambda_u = 2\lambda$, $\lambda_{\ell} = \lambda$ and $\lambda_m = \sqrt{\lambda_u \lambda_{\ell}}$
	
	\lIf{$\bisectionTarget(\lambda_{\ell}) \leq  \frac{15r}{16}$}{\Return $\lambda_{\ell}$
	\Comment*[f]{happens only if $\bisectionTarget(\lambda_{\ell}) \in [\frac{13r}{16}, \frac{15r}{16}]$}
	}\label{line:stop-middle}
		
	\While{$\bisectionTarget(\lambda_m)\notin [\frac{13r}{16}, \frac{15r}{16}]$ \textup{\textbf{and}} $\log_2 \frac{\lambda_u}{\lambda_{\ell}} \ge  \frac{r}{8(R+G/\lambda_{\ell})}$ \label{line:while2start} } 
	{
	\leIf{
		$\bisectionTarget(\lambda_m) < \frac{13r}{16}$}{$\lambda_u=\lambda_m$}{$\lambda_\ell = \lambda_m$}
	$\lambda_m = \sqrt{\lambda_u \lambda_{\ell}}$ 	\label{line:while2end}
	} 

	\Return $\lambda_m$ \Comment*[f]{
		the while loop terminates in $O\prn[\big]{\log\prn[\big]{\frac{R}{r} + \frac{G}{\lambda_{\min} r}}}$ steps}
\end{algorithm}

\begin{proposition}[{\cite[][Proposition 2]{carmon2021thinking}}]\label{prop:bisecx}
Let $f : \R^d \rightarrow \R$ be $G$-Lipschitz and convex, and let $x,v\in\R^d$, $\eps, r, R \in \R_{> 0}$ satisfy $\eps \leq G R$, $r \leq R$ and $\norm{x-v}\le 2R$. Given $\lambda_{\max}\ge \tfrac{2G}{r}$ and $\lambda_{\min}\in(0,\lambda_{\max})$, $\linesearch(x,v,A)$ outputs $\lambda \in [\lambda_{\min}, \lambda_{\max}]$ such that 
\[\norm{\prox_{f,\lambda}(y_\lambda)-y_\lambda}\le r.\] 
The subroutine uses $O(\log(\tfrac{\lambda_{\max}}{\lambda_{\min}})+\log (\tfrac{R + G/\lambda_{\min}}{r}))$ calls to $\oracles[\lambda', \frac{r}{17}]{\cdot}$ with $\lambda' \in [\half\lambda, \lambda_{\max}]$. Moreover, for  $\alpha_{2\lambda A} = \frac{2\lambda A}{1+2\lambda A + \sqrt{1+4\lambda A}}$ and  $y_{\lambda} \defeq \alpha_{2 \lambda A } x + (1-\alpha_{2 \lambda A}) v$ one of the following outcomes must occur:
	\begin{enumerate}
		\item \label{item:outcome-large-move} $ 
		\lambda \in [2\lambda_{\min}, \lambda_{\max}]$ and  $\norm{\prox_{f,\lambda}(y_\lambda)-y_\lambda} > \frac{3r}{4}$, or
		\item \label{item:outcome-small-lambda} $\lambda < 2\lambda_{\min}$.
	\end{enumerate}
	When taking $\lambda_{\max}=\tfrac{2G}{\eps}$ and $\lambda_{\min} = \Omega(\frac{\epsilon}{rR})$, the number of calls to $\oracles[\lambda', \frac{r}{17}]{\cdot}$ is $O(\log\tfrac{G R^2}{r\eps})$.
\end{proposition}

\subsubsection{Proof of \Cref{thm:min-the-max}}

Finally, we combine the guarantees collected above to prove our near-optimal rate for minimizing the maximum-loss.

\newcommand{\KmaxBis}{\Kmax^{\mathrm{bisect}}}
\newcommand{\pBOO}{p_{\textup{\BOO}}}
\thmMinTheMax*

\begin{proof}
	We first prove correctness. Since $\eps' = \frac{\eps}{2\log N}$, we have  $0 \le \fsm(x) -\fmax(x)\le \eps/2$~\cite[see, e.g.,][Lemma 45]{carmon2020acceleration}. Therefore, it suffices to find an $\eps/2$-approximate solution of $\fsm$ over the domain $\xset\subseteq \ball_R(x_0)$. Let $\pBOO$ be the probability that all the \BOO calls within \Cref{alg:bisecx} (implemented as described in \Cref{lem:sgd-broo}) result in a valid output. Then, noting that $y_\lambda$ defined in \Cref{prop:bisecx} is precisely $y_k$ defined in \Cref{alg:MS}, the guarantees of \Cref{prop:bisecx} imply that, for  $\hx_{k+1}=\prox_{\fsm,\lambda}(y_k)$, we have $\norm{\hx_{k+1}-y_k}\le r$ and either $\norm{\hx_{k+1}-y_k} \ge 3r/4$ or $\lambda < 2\lambda_{\min}$ with probability at least $\pBOO$. Consequently, \Cref{prop:sapm} (with $\epsilon\to\epsilon/2$ and $\Kmax$ as required by the proposition), the output $x$ of \Cref{alg:MS} satisfies $\fsm(x)-\fsm(x\opt)\le \epsilon/2$ with probability at least $1-(1-\frac{2}{3}) - (1-\pBOO)=\pBOO-\frac{1}{3}$. 
	
	To finish the proof of correctness, it remains to verify that $\pBOO \ge 5/6$. To that end, let $\KmaxBis=O(\log \frac{GR^2}{\reps\eps})$ to be the total number of \BOO calls in a single execution of $\linesearch$, as per \Cref{prop:bisecx}. Then, if the probability of failure of a single \BOO implementation is $\pf$ and we perform at most $\Kmax$ calls to $\linesearch$, we have $\pBOO \ge 1 - \Kmax \KmaxBis \pf$. Therefore, taking 
	\[\pf \le  \frac{1}{6\Kmax \KmaxBis}\]
	guarantees correctness.

	We now proceed to bound the algorithm's complexity. To that end, we set
	\[
	\lambda_{\min}=\frac{\eps}{\reps^{4/3}R^{2/3}}\log^2 \prn*{\frac{GR}{\epsilon}}.
	\]
	Recalling that $\reps = \tfrac{\eps}{2G\log N}$, the total number of iterations in \Cref{alg:MS} is  at most
	\begin{align}
	\Kmax & = O\left(\left(\frac{R}{\reps}\right)^{2/3}\log\frac{GR}{\eps}+\sqrt{\frac{\lambda_{\min}R^2}{\eps}}\right) = O\left(\left(\frac{GR\log N}{\eps}\right)^{2/3}\log\frac{GR}{\eps}\right).\label{eq:Kmax-simplified}
	\end{align}
	Setting the approximation parameters to be 
	$\vphi_{k} = O(\frac{\epsilon}{\lambda_{k}a_{k}})$,  $\delta_{k} = O(\frac{\epsilon}{R})$ and $\sigma_{k}^2 = O(\frac{\epsilon}{a_{k}})$ as required in \Cref{prop:sapm}, the complexity of \cref{line:sapm-x,line:sapm-v} in the $k$th iteration of \Cref{alg:MS} is bounded by \Cref{lem:BROO-complexity-y} and~\Cref{lem:BROO-complexity-g} as
	\begin{align*}
	\E\NqueryCustom{\fhat[i]}^{(k),1} & = O\left(N + \frac{G^2}{\lambda_k\vphi_k} + \frac{G^2}{\sigma_k^2}\log^2\left(\frac{G}{\min\{\delta_k,\sigma_k\}}\right)+\log\left(\frac{G}{\min\{\delta_k,\sigma_k\}}\right)\right)\\
	& = O\left(N + \frac{G^2a_k}{\eps}\log^2\left(\frac{GR}{\eps}\right)+\log\left(\frac{GR}{\eps}\right)\right)\\
	\E\NqueryCustom{\del\fhat[i]}^{(k),1} & = O\left(\frac{G^2}{\lambda_k\vphi_k} + \frac{G^2}{\sigma_k^2}\log^2\left(\frac{G}{\min\{\delta_k,\sigma_k\}}\right)+\log\left(\frac{G}{\min\{\delta_k,\sigma_k\}}\right)\right)\\
	& = O\left(\frac{G^2a_k}{\eps}\log^2\left(\frac{GR}{\eps}\right)+\log\left(\frac{GR}{\eps}\right)\right).
	\end{align*}
	
	To bound the complexity of the bisection procedure at the $k$th iteration of \Cref{alg:MS}, note that it makes a total of  $\KmaxBis = O(\log \frac{GR^2}{\reps\eps}) = O(\log \frac{GR\log N}{\eps}) = O(\log \frac{GR}{\eps})$ \BOO calls, $\reps = \tfrac{\eps}{2G\log N}$ and $\log N\le\frac{GR}{2\eps}$. Applying \Cref{lem:sgd-broo} with $\pf$ and $\lambda_{\min}$ as determined above, the complexity is bounded by
	\begin{align*}
	\NqueryCustom{\fhat[i]}^{(k),2} & = O\left(\left(N + \frac{G^2}{\lambda_{\min}^2\reps^2}\log\left(\frac{\log(G/(\lambda_{\min}\reps))}{\pf}\right)\right)\KmaxBis\right)\\
	 & = O\left(\left(N + \frac{G^2\reps^{2/3}R^{4/3}}{\eps^2\log^4\left(\frac{GR}{\eps}\right)}\log\left(\frac{GR}{\eps}\right)\right)\log\left(\frac{GR}{\eps}\right)\right),\\
	\NqueryCustom{\del\fhat[i]}^{(k),2} & = O\left(\frac{G^2}{\lambda_{\min}^2\reps^2}\log\left(\frac{\log(G/(\lambda_{\min}\reps))}{\pf}\right)\KmaxBis\right) = O\left(\frac{G^2\reps^{2/3}R^{4/3}}{\eps^2\log^2\left(\frac{GR}{\eps}\right)}\right).
	\end{align*}
	
	Summing the bounds above over iterations $1$ to $K\le \Kmax$ and noting that $\sum_{k\le K} a_k = A_K \le 2A_{\max}=O(R^2 / \epsilon)$ (see proof of \Cref{prop:sapm}) we obtain the total complexity bounds 
	\begin{align*}
	\E\NqueryCustom{\fhat[i]} & = \sum_{k\le K}\left(\E\NqueryCustom{\fhat[i]}^{(k),1} + \E\NqueryCustom{\fhat[i]}^{(k),2}\right)\\
	 & = O\left(\Kmax N\log\frac{GR}{\eps}+\frac{G^2A_K}{\eps}\log^2\frac{GR}{\eps} + \Kmax \frac{G^2\reps^{2/3}R^{4/3}}{\eps^2\log^4\left(\frac{GR}{\eps}\right)}\cdot\log^2\left(\frac{GR}{\eps}\right)\right)\\
	& = O\left(\left(\frac{GR\log N}{\eps}\right)^{2/3}N\cdot \log^2\frac{GR}{\eps}+\frac{G^2R^2}{\eps^2}\log^2\frac{GR}{\eps}\right),\end{align*}
	and 
	\begin{align*}
	\E\NqueryCustom{\del\fhat[i]}  = \sum_{k\le K}\left(\E\NqueryCustom{\del\fhat[i]}^{(k),1} + \E\NqueryCustom{\del\fhat[i]}^{(k),2}\right) = O\left(\frac{G^2R^2}{\eps^2}\log^2\frac{GR}{\eps}\right),\end{align*}
	 where we have used formula~\eqref{eq:Kmax-simplified} for $\Kmax$. This concludes the proof.
\end{proof}

\section{Proofs from Section~\ref{sec:gs}}\label{app:gs}

In the section we prove \Cref{thm:GS}, the convergence guarantee for \Cref{alg:GS}, our gradient-efficient composite optimization method. We first provide a lemma (\Cref{lem:helper-seq-conv}) that helps us analyze the behavior of the $\beta_k$ and $\gamma_k$ sequences in the algorithm. Then we combine it with the approximation guarantees of our estimator to show the convergence rate of Algorithm~\ref{alg:GS} in~\Cref{prop:GS-rate}. Finally we apply this proposition and bound the expected number of gradient queries complete the proof of Theorem~\ref{thm:GS}.

The following helper lemma is also used in~\citet{lan2015bundle,lan2016gradient}; we provide it here for completeness of analysis.
\begin{lemma}[Convergence of geometric sequence, cf. Lemma~2 of~\citet{lan2016gradient}]\label{lem:helper-seq-conv}
	Given $\gamma_k\in(0,1)$, for all $k\in\N$, and $\Gamma_1>0$, define the sequence
	\[
	\Gamma_k \defeq (1-\gamma_k)\Gamma_{k-1},\quad \forall k\ge 2.
	\]
	If a sequence $E_k$ satisfies $E_k\le (1-\gamma_k)E_{k-1}+B_k$, for all $k\ge 1$, then we have for any $k\ge 1$,
	\[
	E_k\le \Gamma_k\left[\frac{1-\gamma_1}{\Gamma_1} E_0+\sum_{i\in[k]}\frac{B_i}{\Gamma_i}\right].
	\]
\end{lemma}

Using the helper lemma, we can show the following convergence rate for Algorithm~\ref{alg:GS}.

\begin{proposition}[Convergence rate]\label{prop:GS-rate}
	Given problem~\eqref{eq:composite} with optimizer $x^\star$ and initial point $\|x_0-x^\star\|\le R$, let $\sigma_k^2 = \frac{R^2}{4N}$, $\delta_k = \frac{R}{16N}$, $\eps_k = \frac{LR^2}{2kN}$, and let parameters $\beta_k=\tfrac{2L}{k}$,  $\gamma_k = \tfrac{2}{k+1}$. Then, the iterates of \Cref{alg:GS} satisfy
	\[
	\Psi(x_N)-\Psi(\xopt)\le O\left(\frac{LR^2}{N^2}\right).
	\]
\end{proposition}

\begin{proof}
We first observe that
\begin{align*}	
	\Lambda(x_k) \stackrel{(i)}{\le} &  \Lambda(y_k)+\langle\grad \Lambda(y_k),x_k-y_k\rangle+\frac{L}{2}\|x_k-y_k\|^2\\
\stackrel{(ii)}{=}  & (1-\gamma_k)\left[\Lambda(y_k)+\langle\grad \Lambda(y_k),x_{k-1}-y_k\rangle\right]\\
& +\gamma_k\left[\Lambda(y_k) +\langle\grad \Lambda(y_k),\bar{v}_{k}-y_k\rangle\right] + \frac{L\gamma_k^2}{2}\|\bar{v}_k-\proj_{\xset}(v_{k-1})\|^2\\
\stackrel{(iii)}{\le} & (1-\gamma_k)\Lambda(x_{k-1})+\gamma_k\left[\Lambda(y_k)+\langle\grad \Lambda(y_k),\bar{v}_{k}-y_k\rangle + \frac{\beta_k}{2} \|\proj_{\xset}(v_{k-1})-\bar{v}_k\|^2\right]\\
& -\frac{\gamma_k\beta_k-L\gamma_k^2}{2}\|\proj_{\xset}(v_{k-1})-\bar{v}_k\|^2\\
\stackrel{(iv)}{\le} & (1-\gamma_k)\Lambda(x_{k-1})+\gamma_k\left[\Lambda(y_k)+\langle\grad \Lambda(y_k),\bar{v}_{k}-y_k\rangle + \frac{\beta_k}{2} \|\proj_{\xset}(v_{k-1})-\bar{v}_k\|^2\right],
\end{align*}
where we use $(i)$ $L$ smoothness of function $\Lambda$, $(ii)$ expanding $x_{k} = (1-\gamma_k)x_{k-1}+\gamma_k\bar{v}_k$ and replacing $y_k-x_k = \gamma_k (\proj_{\xset}(v_{k-1})-\bar{v}_{k})$, $(iii)$ convexity of $\Lambda$, and $(iv)$ that $\beta_k\ge L\gamma_k$.

Similarly using convexity of the non-smooth component $f$ and the definition of $x_k$ and $\bar{v}_k$, we obtain
\begin{align*}
    f(x_k)\le(1-\gamma_k)f(x_{k-1})+\gamma_kf(\bar{v}_k). 
\end{align*}

Thus, summing the two inequalities and recalling the definition $\bar{\Lambda}_k(v)=\Lambda(y_k)+\langle \grad \Lambda(y_k),v-y_k\rangle$, this is equivalent to
\[
\Lambda(x_k)+f(x_k)\le (1-\gamma_k)\left(\Lambda(x_{k-1})+f(x_{k-1})\right)+\gamma_k\left[\bar{\Lambda}_k(\bar{v}_k)+f(\bar{v}_k)+\frac{\beta_k}{2}\|\proj_{\xset}(v_{k-1})-\bar{v}_k\|^2\right].
\]

Now we recall the definition of composite objectives $\Psi(x) = \Lambda(x)+f(x)$ and define
\[\Phi_k(x) = \bar{\Lambda}_k(x)+f(x)+\frac{\beta_k}{2}\|x-\proj_{\xset}(v_{k-1})\|^2.\]
By convexity of $\Psi$ one has the recursion
\[
\Psi(x_k)-\Psi(u)\le (1-\gamma_k)\left(\Psi(x_{k-1})-\Psi(u)\right)+\gamma_k \left(\Phi_k(\bar{v}_k)-\Phi_k(u)+\frac{\beta_k}{2}\|\proj_{\xset}(v_{k-1})-u\|^2\right).\]

Let $v^\star_k$ be the exact minimizer of  $\Phi_k$ restricted to  $\bar{\xset}\defeq \ball_R(v_0)\cap \xset$. We have, for any $u\in\bar{\xset}$, that $\Phi_k(u)\ge \Phi_k(v^\star_k) + \frac{\beta_k}{2}\|v^\star_k-u\|^2$, and consequently
\begin{align*}
\Psi(x_k)-\Psi(u)\le & (1-\gamma_k)\left(\Psi(x_{k-1})-\Psi(u)\right)\\
& +\gamma_k \left(\Phi_k(\bar{v}_k)-\Phi_k(v_k^\star)+\frac{\beta_k}{2}\left(\|\proj_{\xset}(v_{k-1})-u\|^2-\|v_k^\star-u\|^2\right)\right).
\end{align*}

Conditioning on past events and taking expectation over randomness of $v_k$ and $\bar{v}_{k}$, this gives  for any $u\in\bar{\xset}$, 
\begin{align*}
\E \Psi(x_k)-\Psi(u)   & \le (1-\gamma_k)\left(\Psi(x_{k-1})-\Psi(u)\right)+\gamma_k \left(\E \Phi_k(\bar{v}_k)-\Phi_k(v_k^\star)\right)\\
 &\phantom{\le}+\frac{\gamma_k\beta_k}{2} \left(\|\proj_{\xset}(v_{k-1})-u\|^2-\E\|v_k-u\|^2+ \E\|v_k-v_k^\star\|^2+\E2\langle v_k^\star-u,v_k-v_k^\star\rangle\right)\\
& \stackrel{(i)}{\le}  (1-\gamma_k)\left(\Psi(x_{k-1})-\Psi(u)\right)+\gamma_k \left(\E \Phi_k(\bar{v}_k)-\Psi(v_k^\star)\right)\\
 & \phantom{\le} +\frac{\gamma_k\beta_k}{2} \left(\|\proj_{\xset}(v_{k-1})-u\|^2-\E\|v_{k}-u\|^2+ \E\|v_k-v_k^\star\|^2+4R\|\E v_k -v_k^\star\|\right)\\
 & \stackrel{(ii)}{\le} (1-\gamma_k)\left(\Psi(x_{k-1})-\Psi(u)\right)+\gamma_k \left(\E \Phi_k(\bar{v}_k)-\Phi_k(v_k^\star)\right)\\
&\phantom{\le}  +\frac{\gamma_k\beta_k}{2} \left(\|\proj_{\xset}(v_{k-1})-u\|^2-\E\|\proj_{\xset}(v_{k})-u\|^2+ \E\|v_k-v_k^\star\|^2+4R\|\E v_k -v_k^\star\|\right)\\
\end{align*}
where we use $(i)$ the triangle inequality and $v_k^\star\in\bar{\xset}$ to conclude  $\|v_k^\star-u\|\le \|v_k^\star-x_0\|+\|x_0-u\|\le 2R$, and $(ii)$ the projection property that $\|\proj_{\xset}(v_{k})-u\|^2\le \|v_{k}-u\|^2$ for any $u\in\bar{\xset}$.

Note that $\E\Phi_k(\bar{v}_k)-\Phi_k(v_k^\star) \le \epsilon_k$ by the definition of  $\bar{v}_k=\appProx_{\bar{\Lambda}_k+f,\beta_k}^{\eps_k}(v_{k-1})$. Moreover, \Cref{thm:optest} guarantees that $\E\|v_k-v_k^\star\| \le \delta_k$ and that $\E\|v_k-v_k^\star\|^2 \le \sigma_k^2$. Therefore, writing
\begin{equation*}
	E_k = \E \Psi(x_k)-\Psi(u)
\end{equation*}
and
\begin{equation*}
	B_k = \frac{\gamma_k\beta_k }{2}\prn*{ \E \|\proj_{\xset}(v_{k-1})-u\|^2 - \E\|\proj_{\xset}(v_{k})-u\|^2} + {\gamma_k \beta_k}\prn*{\frac{\epsilon_k}{\beta_k} + \frac{\sigma_k^2}{2} + 2R\delta_k},
\end{equation*}
we conclude that $E_k \le (1-\gamma_k) E_{k-1} + B_k$. Applying \Cref{lem:helper-seq-conv}, we obtain
\begin{align*}
\E \Psi(x_N)-\Psi(u) \le & \Gamma_N \frac{1-\gamma_1}{\Gamma_1}\left[\Psi(x_0)-\Psi(u)\right]\\
& +\Gamma_N\sum_{k=1}^N\frac{\beta_k\gamma_k}{2\Gamma_k}\left(\E\|\proj_{\xset}(v_{k-1})-u\|^2-\E\|\proj_{\xset}(v_{k})-u\|^2\right)\\
& +\Gamma_N\sum_{k\in[N]}\frac{\beta_k\gamma_k}{\Gamma_k} \prn*{\frac{\epsilon_k}{\beta_k} + \frac{\sigma_k^2}{2} + 2R\delta_k}\\
\stackrel{(i)}{\le} &  \Gamma_N L \|v_0-u\|^2  +\Gamma_N\sum_{k\in[N]}\frac{\beta_k\gamma_k}{\Gamma_k}\prn*{\frac{\epsilon_k}{\beta_k} + \frac{\sigma_k^2}{2} + 2R\delta_k} \stackrel{(ii)}{\le}  \frac{4LR^2}{N(N+1)},
\end{align*}
where $(i)$ follows from telescoping and $\gamma_1=1$, and $(ii)$ is due to $\Gamma_k = \prod_{k\ge 2} (1-\gamma_k) = \frac{2}{k(k+1)}$, so that
 $\frac{\beta_k\gamma_k}{\Gamma_k}=2L$, and 
 $\frac{\epsilon_k}{\beta_k} + \frac{\sigma_k^2}{2} + 2R\delta_k \le \frac{ R^2}{2N}$ 
by the choices $\sigma_k^2 = \frac{R^2}{4N}$, $\delta_k = \frac{R}{16N}$ and  $\eps_k = \frac{LR^2}{2kN}$.
\end{proof}

We are now ready to prove the main theorem of the section.

\thmGS*

\begin{proof}

By Proposition~\ref{prop:GS-rate}, it suffices to run Algorithm~\ref{alg:GS} for $N = O(\sqrt{LR^2/\eps})$ iterations, which immediately implies the stated bound on $\NqueryCustom{\grad \Lambda}$.

Now we consider the cost of attaining the requiring accuracy $\epsilon_k$ when computing $\bar{v}_k$. Using the $\epochSGD$ and \Cref{prop:epoch-sgd} we can do so with
\[
N^{(1)}_k = O\left(\frac{G^2}{\beta_k\eps_k}\right) = O\left(\frac{G^2k^2N}{L^2R^2}\right)
\]
queries to $\gest[f]$.

Applying \Cref{thm:optest}, the expected cost of attaining bias $\delta_k = \frac{R}{16N}$ and variance $\sigma_k^2 = \frac{R^2}{4N}$ is 
\begin{align*}
\E N^{(2)}_k & = O\left(\log\left(\frac{GNk}{LR}\right) + \frac{NG^2}{\beta_k^2R^2}\log^2\left(\frac{GNk}{LR}\right)\right)\\
& = O\left(\log\left(\frac{GNk}{LR}\right) +\frac{G^2k^2N}{L^2R^2}\log^2\left(\frac{GNk}{LR}\right) \right)
\end{align*}
queries to $\gest[f]$.

Summing these over all $k\le N = O(\sqrt{LR^2/\epsilon})$, we obtain the the required complexity bound
\begin{equation*}
\begin{aligned}
    \E \mathcal{N}_{\gest[f]} & =\sum_k\left(N_k^{(1)}+\E N_k^{(2)}+1\right) =  O\left(N\log\left(\frac{GN^2}{LR}\right)+\frac{G^2N^4}{L^2R^2}\log^2\left(\frac{GN^2}{LR}\right) \right)\\
    & = O\left(\sqrt{\frac{LR^2}{\eps}}\log\left(\frac{GR}{\eps}\right)+ \frac{G^2R^2}{\eps^2}\log^2\left(\frac{GR}{\eps}\right)\right).
    \end{aligned}
\end{equation*}
\end{proof}

\section{Proofs and additional remarks from Section~\ref{sec:DP-SCO}}\label{app:DP-SCO}

In this section we prove~\Cref{thm:SCO-convex} which gives an optimal complexity and generalization bound for differentially private stochastic convex optimization, conditional on the existence of an improved optimum estimator  (\Cref{def:bounded-MLMC}). We begin by stating a standard privacy guarantee for the Gaussian mechanism applied on mappings with bounded $\ell_2$ sensitivity, and a lemma that helps us bound the sensitivity of the conjunctured bounded estimator. With these results in hand, we prove \Cref{thm:SCO-convex}. Finally, we discuss some challenges and prospects for constructing bounded estimators that satisfy~\Cref{def:bounded-MLMC}.

\newcommand{\dham}{d_{\mathsf{ham}}}

\subsection{Helper lemmas}
\paragraph{Privacy of the Gaussian mechanism.}
In this section, we present the privacy guarantees of the Gaussian mechanism which will be useful for the proof of~\Cref{thm:SCO-convex}.
First, for an estimator (or a function) $h : \domain^n \to \R^d $,
the $\ell_2$-sensitivity of the estimator is upper bounded by $\Delta$ if $\sup_{\Ds,\Ds' \in \domain^n: \dham(\Ds,\Ds') \le 1} \ltwo{h(\Ds) - h(\Ds')} \le \Delta$, where $\dham$ is the hamming distance between the two samples (i.e., $\Ds,\Ds'$ with hamming distance $\dham(\Ds,\Ds') \le 1$ have at most a single different element).  We can now state the privacy guarantees of the Gaussian mechanism.

\begin{lemma}[Gaussian mechanism~{\cite[][Theorem~A.1]{DworkRo14}}]
	\label{lemma:gauss-mech}
	Let $h: \domain^n \to \R^d$ have $\ell_2$-sensitivity $\Delta$.
	Then the Gaussian mechanism $\A(\Ds) = h(\Ds) + \normal(0,\sigma^2 I_d)$ with $\sigma = 2 \Delta \log(2/\de)/\diffp$ is $(\diffp,\de)$-DP.
\end{lemma}

\paragraph{Bounding the number of estimator copies that use a particular sample.}
To prove~\Cref{thm:SCO-convex}, we begin with a lemma which bounds the number of optimum estimator copies that each sample can participate in. To this end, 
let $S_{i,t}$ denote the set of samples used in iteration $i$ of \Cref{alg:MLMC-loc-erm} during the computation of the $t$'th optimum estimator copy. For a sample $\ds_\ell$, we let $K_{i,\ell}$ denote the number of sets $S_{i,t}$ such that $z_\ell \in S_{i,t}$. Recalling that the number of iterations $k=\ceil{\log n }$ and that $\bar{n} = n/k$, we have the following lemma.
\begin{lemma}
	\label{lemma:orac-sample-appearances}
	Let $\sc_i = \frac{1}{\ss_i \bar n}$.
	Assume we use an optimum oracle $\MLMCoracd$ satisfying \Cref{def:bounded-MLMC} with constant $C_2$ and $\delta_i^2 = \frac{\lip^2}{\sc_i^2 \bar n}$. %
	Then, for any $\de \le 1/n$,
	\begin{equation*}
	\P \left( \max_{1 \le i \le k, 1 \le \ell \le n} K_{i,\ell}  \ge 20 \log (1/\de) + 6 C_2 \log^2 n  \right)
	\le \de/2.
	\end{equation*}
\end{lemma}
\begin{proof}
	We first prove the claim for a fixed $i$ and $\ell$ and then we apply a union bound. Fix $1 \le i \le k$ and $1 \le \ell \le n$
	and define $Y_t = \indic{z_\ell \in S_{i,t}}$. Now we upper bound $p = \P(Y_t=1)$. Let the random variable $N_t$ denote the number of subgradients the $t$'th query to $\MLMCorac$ at iteration $i$ uses.
	First, note that whenever $N_t = j$, we have
	\begin{align*}
	\P(Y_t=1 \mid N_t=j ) 
	 \le j / \bar n,
	\end{align*}  
	by the union bound.
	Thus, \Cref{def:bounded-MLMC} now implies
	\begin{align*}
	\P(Y_t=1) 
	& = \sum_{k=1}^{\infty} \P(Y_t=1 \mid N_t=j) \P(N_t=j) \\
	& \le \frac{1}{\bar{n}} \sum_{k=1}^{\infty} \P(N_t=j) j  =\frac{ \E[N]}{\bar{n}} 
	  = \frac{C_2}{\bar{n}} \log\frac{\lip}{\sc_i\delta_i}.
	\end{align*}		
	We can now use a Chernoff bound to prove the claim. Indeed, as $K_{i,\ell} = \sum_{t=1}^n Y_t$ and $Y_t \sim \mathsf{Bernoulli}(p)$ are i.i.d., \Cref{lemma:chernoff} below implies that for $c \ge 6$,
	\begin{align*}
	\P\prn*{K_{i,\ell} \ge c \E[K_{i,\ell}] }
	= \P\prn*{ \sum_{t=1}^n Y_t \ge c n p }
	\le 2^{-cnp}.
	\end{align*}
	As $p \le C_2 \log(n) \log(\lip/\sc_i\delta_i)/n$, we take $ c \ge 6$ such that $c n p \ge 20 \log (1/\de)$, hence we have
	\[\P(K_{i,\ell} \ge 20 \log (1/\de) + 6 C_2 \log(n) \log(\lip/\sc_i\delta_i)) ) \le \de^4.\] Applying a union bound over all $n$ samples and all $k=\ceil{\log n}$ iterations, we have that 
	\begin{equation*}
	\P\prn*{ \max_{1 \le i \le k, 1 \le \ell \le n} K_{i,\ell}  \ge 20 \log (1/\de) + 6 C_2 \log(n)  \log(\lip/\sc_i\delta_i) ) }
	\le \de/2.
	\end{equation*}
	The claim now follows by noting that $\frac{\lip}{\sc_i \delta_i} \le \sqrt{n}$ using our choice of $\delta_i$ in~\Cref{alg:MLMC-loc-erm}.

\end{proof}

\begin{lemma}[\cite{MitzenmacherUp05}, Ch.~4.2.1]
	\label{lemma:chernoff}
	Let $X = \sum_{i=1}^n X_i$ for $X_i \simiid \mathsf{Bernoulli}(p)$.
	Then for $c \ge 6$,
	\begin{align*}
	\\P(X \ge c np ) \le 2^{-c np}.
	\end{align*}
\end{lemma}

\subsection{Proof of~\Cref{thm:SCO-convex}}
\thmDPSCO*

\newcommand{\goodev}{\mathfrak{E}}
\begin{proof}
	We begin by proving the privacy claim. We show that each iterate is $(\diffp,\de)$-DP which completes the proof by post-processing as each sample is used in exactly one iterate. To this end, first we show that, with high probability, each sample $z_\ell$ is used in at most $B=20 (\log (\frac{1}{\de}) + C_2 \log^2 n)$ different optimum-estimator queries; we let $\goodev$ denote this event. More precisely, let $S_{i,t}$ denote the set of samples used in iteration $i$ during the application of the $t$'th oracle. Then for every $i$ and sample $z_\ell$, letting $K_{i,\ell}$ be the number of sets $S_{i,t}$ such that $z_\ell \in S_{i,t}$. Using this notation, the event $\goodev$ is equivalent to $\max_{1 \le i \le k, 1 \le \ell \le n} K_{i,\ell} \le B$. \Cref{lemma:orac-sample-appearances} implies that $P[\goodev] \ge 1 - \de/2$, therefore we only have to prove $(\diffp,\de^2 /2)$-differential privacy assuming event $\goodev$ happens as we have using $e^\diffp \le 1/\de$ that
	\begin{align*}
	P[\A(\Ds)\in \cO] 
	& \le P[\A(\Ds')\in \cO \mid \goodev] P[\goodev] + ( 1 - P[\goodev]) \\
	& \le e^{\diffp} P[\A(\Ds')\in \cO \mid \goodev ]P[\goodev] +\de/2 \\
	& \le e^{\diffp} P[\A(\Ds')\in \cO ] + \de .
	\end{align*} 

	We therefore assume $\goodev$ holds and proceed to bound the $\ell_2$-sensitivity of $\tilde x_i$. 
	To this end, let  $\sc_i = 1/(\ss_i \bar n)$ and $\hat x_i = \argmin_{x \in \xdomain} F_i(x)$.
	First, note that each optimum estimation oracle output satisfies 
	\begin{align*}
	\ltwo{\MLMCoracd_{\delta_i}(F_i) - x_{i-1}} 
	& \le \ltwo{\MLMCoracd_{\delta_i}(F_i) - \hat x_i} + \ltwo{\hat x_i - x_{i-1}} \\
	& \overle{(\star)}   \sqrt{C_1} \lip \sqrt{\log n}/\sc_i + \lip/\sc_i \\
	& = ( \sqrt{C_1 \log n} + 1) \lip /\sc_i,
	\end{align*}
	where the first term in inequality $(\star)$ above holds since the estimator $\MLMCoracd_{\delta_i}$ satisfies~\Cref{def:bounded-MLMC} and $F_i = f_i + \psi_i$ where $f_i$ is $\lip$-Lipschitz and $\psi_i$ is $\sc_i$-strongly convex with $\lip/(\sc_i \delta_i) \le \sqrt{n}$.
	The second term of the inequality holds since $\psi_i(x) = 	\sc_i \norm{x - x_{i-1}}^2$, thus as $f_i$ is $\lip$-Lipschitz we have
	\begin{equation*}
	\sc_i \ltwo{\hat x_i - x_{i-1}}^2
	\le f_i(x_{i-1}) - f_i(\hat x_i)
	\le \lip \ltwo{\hat x_i - x_{i-1}}.
	\end{equation*}
	As event $\goodev$ holds, each sample participates in at most $B$ of the optimum estimator computations queries, hence we have that the $\ell_2$-sensitivity of $\tilde x_i$
	is at most $2 \frac{B}{\bar n} ( \sqrt{C_1 \log n} + 2)  \lip / \sc_i $.
	Privacy properties of the Gaussian mechanism (\Cref{lemma:gauss-mech}) and our choice of $\sigma_i$ now imply that each iterate is $(\diffp,\de^2/2)$-DP whenever event $\goodev$ holds, which proves the claim about privacy.
	
	Let us now prove utility following steps similar to the proof of Theorem 4.4 in~\cite{FeldmanKoTa20}.
	We define the non-private minimizers, $\hat x_i = \argmin_{x \in \xdomain} F_i(x)$ and $\hat x_0 = \xopt$. We have
	\begin{equation}
	\label{eqn:main-loc}
	f(x_k) - f(\xopt)
	= \sum_{i=1}^k \brk*{f(\hat x_i) - f(\hat x_{i-1})}
	+ f (x_k) - f(\hat x_k).
	\end{equation}
	Using the definitions of $\sigma_i$ and $\eta_i$ in \Cref{alg:MLMC-loc-erm}, we also have that for every $i \ge 1$
	\begin{align*} 
	\E [\ltwo{\hat x_i - x_i}^2] 
	& \le 2 \E [\ltwo{\hat x_i - \tilde x_i}^2] + 2 \E [\ltwo{\tilde x_i - x_i}^2] \\
	& \le 2 \E [\ltwo{\hat x_i - \tilde x_i}^2] + O \left( \frac{\lip^2 B^2 \ss_i^2 d  \log(n)\log(1/\de) }{\diffp^2} \right) \\
	&  \le 2 \E [\ltwo{\hat x_i - \tilde x_i}^2] + O \left( \frac{\lip^2 B^2 \ss^2 d  \log(n) \log(1/\de) }{\diffp^2 2^{8i}} \right) .
	\end{align*}
	Moreover, using properties of the bounded-optimum estimator from~\Cref{def:bounded-MLMC}, that is, $\ltwo{\MLMCoracd_{\delta_i}(F_i,x_{i-1}) - \hat x_i}^2 \le C_1 \lip^2 \log(n) /\sc_i^2$ and $\ltwo{\E[\MLMCoracd_{\delta_i}(F_i,x_{i-1}) - \hat x_i]}^2  \le \delta_i^2$, we have by choosing $\delta_i^2 = \lip^2/ \sc_i^2 \bar n = \lip^2 \ss_i^2 \bar n$,
	\begin{align*} 
	\E \ltwo{\tilde x_i - \hat x_i}^2 
	& = \E{\norm[\Bigg]{\frac{1}{\bar n} \sum_{j=1}^{\bar n} \MLMCoracd_{\delta_i}(F_i,x_{i-1}) - \hat x_i}^2} \\
	& \le \frac{C_1 \lip^2 \log(n)}{\sc_i^2 \bar n} + \rho^2 
	\le (C_1 + 1) {\lip^2 \ss_i^2 \bar n \log(n)}.
	\end{align*}
	We can now bound the terms in~\eqref{eqn:main-loc}. For the second term, the choice of $\ss$ gives %
	\begin{align*} 
	\E[f (x_k) -  f(\hat x_k)]
	& \le \lip \E [\ltwo{x_k - \hat x_k}] \\
	& \le \lip \cdot O \left(\lip \ss_k \sqrt{\bar n \log(n)}  + \frac{\rad B}{2^{6k}} \right) \\
	& \le \lip \cdot O \left(\frac{2 \lip \ss \sqrt{\bar n \log(n) }}{2^{4k}}  + \frac{\rad B}{2^{6k}} \right) \\
	& \le O \left( \frac{\rad \lip}{n} \right).
	\end{align*}
	
	For the first term in~\eqref{eqn:main-loc}, as $F_i$ is $\lip$-Lipschitz over $ \xdomain_i= \{x\in \xdomain: \norm{x - x_{i-1}} \le {2\lip \ss_i  \bar n} \}$, Theorems 6 and 7 in~\cite{ShwartzShSrSr09} imply that for all $y \in \xdomain_i$
	\begin{equation*}
	\E[ f(\hat x_i) - f(y) ] \le \frac{\E[\ltwo{y - x_{i-1}}^2]}{\ss_i \bar n} + 2 \lip^2 \ss_i ,
	\end{equation*}
	hence we now have
	\begin{align*}
	\sum_{i=1}^k \E [f  (\hat x_i) - f (\hat x_{i-1})]
	& \le \sum_{i=1}^k \sc_{i-1} {\E[\norm{\hat x_{i-1} - x_{i-1}}^2] } + 2 \lip^2 \ss_i  \\
	& \le O \left(  \frac{\rad^2}{\ss \bar n} 
	+  \sum_{i=2}^k  \sc_i \left( \frac{\lip^2 \log(n) }{\sc_i^2 \bar n} +  \frac{\lip^2 B^2 \ss_i^2 d \log(n) \log(1/\de)}{\diffp_i^2}\right) +  \lip^2 \ss_i   \right) \\
	& \le O \left( \frac{\rad^2}{\ss \bar n} 
	+  \sum_{i=2}^k   {\lip^2 \ss_i \log(n) }  +  \frac{\lip^2 B^2 \ss_i d \log(n) \log(1/\de)}{\diffp_i^2 \bar n } \right)  \\
	& \le O \left( \frac{\rad^2}{\ss \bar n} 
	+  \lip^2 \ss \log(n)   +  \sum_{i=2}^k  2^{-i} \frac{\lip^2 B^2 \ss d \log(n)  \log(1/\de)  }{\diffp^2 \bar n } \right)  \\
	& \le \lip \rad \cdot O \left( \frac{{\log n}}{\sqrt{n}} + \frac{ B  \log(n) \sqrt{d \log(1/\de)  }}{n \diffp} \right) ,
	\end{align*}
	where the last inequality follows since $\bar n = n/\ceil{\log(n)}$,
	and $\ss =\frac{\rad}{\lip} \min ( 1/{\sqrt{n}}, { \diffp}/{ B \log(n) \sqrt{d \log(1/\de) }} )$.
\end{proof}

\subsection{The challenges of obtaining a bounded optimum estimator}\label{app:DP-SCO-challenges}
To highlight the challenge of finding bounded estimators that satisfy \Cref{def:bounded-MLMC}, let us explain why our MLMC optimum estimator~\eqref{eq:def-mlmc} fails to do so. For this estimator, we have (when $2^J \le \Tmax$)
\begin{equation*}
	\norm{\xmlmc - \xopt} \le \norm{\xopt - x_0} + 2^J \norm{x_J - x_{J-1}},
\end{equation*}  
where $x_j$ is the output of an ODC algorithm with query budget $2^j$. The ODC property and the triangle inequality then roughly imply that $\norm{x_j - x_{j-1}} = O(2^{-j/2} G / \mu)$ and consequently (since $\norm{\xopt - x_0}=O(G/\mu)$) we have $\norm{\xmlmc - \xopt}=O(2^{J/2} G/\mu)=O(\sqrt{\Tmax}G/\mu)$ which clearly is not enough to guarantee an $\Otil{G/\mu}$ bound on $\norm{\xmlmc - \xopt}$. Indeed, to guarantee such bound with a similar analysis we would have needed $\norm{x_j- x_{j-1}} = O(2^{-j} G / \mu)$. However, this would imply that, by the triangle inequality,
\begin{equation*}
	\norm{x_j - \xopt} = \norm{x_j - x_\infty} \le \sum_{k=j+1}^\infty \norm{x_k - x_{k-1}} = \sum_{k=j+1} O(2^{-k} G/ \mu) = O(2^{-j} G/\mu),
\end{equation*}
which contradicts the lower bound on the optimal distance convergence rate in \Cref{app:remark-odc}. 

Having explained why the analysis strategy underlying our estimator~\eqref{eq:def-mlmc} cannot directly yield a bounded optimum estimator, we discuss two approaches with a potential to solve the problem.
The first approach is to apply ODC algorithms on a smooth surrogate of the true objective $F$, for which the faster convergence to the optimum is possible, e.g., using randomized smoothing~\cite{duchi2012randomized,KulkarniLeLi21}.

\notarxiv{
The second approach is try to directly bound the $\ell_2$ sensitivity of our MLMC-based approach. In particular, it might be possible to leverage the structure of our estimator (or an improved version thereof)  in order to control the $\ell_2$ sensitivity without relying on the boundedness of the estimator as we currently do in the proof of \Cref{thm:SCO-convex}.
}

\arxiv{
The second approach is try to directly bound the $\ell_2$ sensitivity of our MLMC-based approach. In particular, it might be possible to leverage the structure of our estimator (or an improved version thereof)  in order to control the $\ell_2$ sensitivity without relying on the boundedness of the estimator as we currently do in the proof of \Cref{thm:SCO-convex}.
}

\end{document}